\documentclass[11pt,times]{amsart}

\usepackage{tikz}
\usetikzlibrary{arrows,snakes,backgrounds}
\usepackage{graphicx}
\usepackage{subfigure}               
\usepackage{relsize}
\usepackage[foot]{amsaddr}
\usepackage[all]{xy}
\usepackage[english]{babel}
\usepackage[margin=1in]{geometry}
\usepackage[utf8]{inputenc}
\usepackage{amsmath}
\usepackage{amsthm}
\usepackage[normalem]{ulem}
\usepackage{tikz-cd}
\usepackage[new]{old-arrows}
\usepackage{hyperref}
\hypersetup{
    colorlinks=true, 
    linkcolor=blue, 
    urlcolor=red, 
    linktoc=all, 
    citecolor=blue
}
\usepackage{framed}
\usepackage{amsfonts}
\usepackage{graphicx}
\usepackage{amssymb}
\usepackage[colorinlistoftodos]{todonotes}
\usepackage{indentfirst}
\usepackage{mathtools}
\usepackage{stmaryrd}
\usepackage{bbm}
\usepackage{algorithm} 
\usepackage{algpseudocode}
\usepackage{thmtools}
\usepackage{thm-restate} 

\makeatletter 
\@namedef{subjclassname@1991}{Mathematics Subject Classification}
\makeatother

\newtheorem{theorem}{Theorem}
\newtheorem{atheorem}{Theorem}[section]
\newtheorem{alemma}{Lemma}[section]
\newtheorem{adefinition}{Definition}[section]
\newtheorem{aremark}{Remark}[section]
\newtheorem{aproposition}[atheorem]{Proposition}
\newtheorem{acorollary}{Corollary}[section]
\newtheorem{aexample}{Example}[section]
\newtheorem{proposition}{Proposition}[section]
\newtheorem{corollary}[proposition]{Corollary}
\newtheorem{lemma}[proposition]{Lemma}
\newtheorem{remark}[proposition]{Remark}
\newtheorem{example}[proposition]{Example}

\theoremstyle{definition}
\newtheorem{definition}[proposition]{Definition}
\newtheorem{claim}{Claim}

\newtheorem{conjecture}{Conjecture}

\newcommand{\mX}{\mathcal{X}}
\newcommand{\mY}{\mathcal{Y}}

\newcommand{\card}[1]{\mathrm{card}(#1)}
\newcommand{\diamna}{\mathrm{diam}}
\newcommand{\supp}{\mathrm{supp}}
\newcommand{\vol}{\mathrm{vol}}
\newcommand{\nvol}{\mathrm{nvol}}

\newcommand{\mw}{\mathcal{M}_w}
\newcommand{\mwtr}{\mathcal{M}_w^{\mathrm{tr}}}
\newcommand{\pr}{\mathrm{pr}}
\newcommand{\nr}{\mathrm{nr}}
\newcommand{\cmdsop}{\mathfrak{K}}
\newcommand{\tr}{\mathrm{Tr}}
\newcommand{\trng}{\mathrm{Tr}_{\mathrm{neg}}}
\newcommand{\dis}{\mathrm{dis}}

\newcommand{\N}{\mathbb{N}}
\newcommand{\Z}{\mathbb{Z}}

\newcommand{\R}{\mathbb{R}}

\newcommand{\Sp}{\mathbb{S}}

\author{Sunhyuk Lim}
\address{Sungkyunkwan University (SKKU), Suwon-si, Gyeonggi-do, Republic of Korea.}
\email{lsh3109@skku.edu}

\author{Facundo M\'emoli}
\address{The Ohio State University, Columbus, Ohio, USA.}
\email{facundo.memoli@gmail.com}

\begin{document}
\title{Classical Multidimensional Scaling on Metric Measure Spaces}

\maketitle

\begin{abstract}
We study a generalization of the classical Multidimensional Scaling procedure (cMDS) which is applicable in the setting of  metric measure spaces. Metric measure spaces can be seen as natural ``continuous limits" of finite data sets. 
Given a metric measure space $\mX = (X,d_X,\mu_X)$, the generalized cMDS procedure involves studying an operator which may have infinite rank, a possibility which leads to studying its traceability.
\newline We  establish that several continuous exemplar metric measure spaces such as spheres and tori (both with their respective geodesic metrics) induce traceable cMDS operators, a fact which allows us to obtain the complete characterization of the metrics induced by their resulting cMDS embeddings. To complement this, we also exhibit a metric measure space whose associated cMDS operator is not traceable.
\newline Finally, we establish the stability of the generalized cMDS method with respect to the Gromov-Wasserstein distance.
\end{abstract}

\tableofcontents

\section{Introduction}\label{sec:intro}
The direct manipulation of high dimensional data is often difficult if not intractable. At the same time,  in the past few decades, many areas of science (e.g., neuroscience, climate science,  network science, etc) have  increasingly been confronted with large volumes of high dimensional data whose acquisition or generation has been facilitated by rapid technological advances.  Dimension reduction techniques have become central to data science in the past several decades due to their ability to obtain low-dimensional data representations which retain  meaningful intrinsic properties of the original data. 

When a dataset is modeled as a collection of points in some Euclidean space, principal component analysis (PCA) provides a natural change of basis which can be used for ulterior processing such as dimension reduction \cite{pearson1901}.

Multidimensional Scaling methods (MDS) are a family of \emph{dimension reduction techniques} that are applicable to data represented as a dissimilarity or distance matrix which may or may not arise from Euclidean distances. Given a dissimilarity matrix, MDS methods produce a configuration of points in Euclidean space whose interpoint distance matrix approximates, according to different criteria, the input dissimilarity matrix. Some well known MDS methods are Sammon's mapping, metric MDS  and  classical MDS; see \cite[Sections 2.2 and 2.4]{cox2000multidimensional}, \cite[Chapter 14]{mardiamultivariate}, \cite[Sections 14.8 and 14.9]{hastie2009elements}.

Classical Multidimensional Scaling (cMDS, from now on)  satisfies the property that its output  coincides with that of PCA whenever the input dissimilarities arise from Euclidean distances \cite[Section 14.3]{mardiamultivariate}. 
While initially developed for visualization of data in Behavioral Science back in the mid-to-late 1930s, thus predating the current Data Science era, cMDS serves as a method for dimension reduction \cite[Section 2.2.3]{cox2000multidimensional}.\footnote{There, the authors say ``$\ldots$  appropriate as an exploratory data technique for dimension reduction."}  cMDS is closely related to the ISOMAP method  \cite{tenenbaum2000global} for nonlinear dimension reduction. The latter method consists of two main internal procedures: first, a  step which  estimates pairwise geodesic distances and a second step that applies cMDS  to the  distance matrix resulting from the first step.

From a mathematical perspective, the goal of cMDS is to find an \emph{optimal} embedding of a finite metric space $(X,d_X)$ into Euclidean space. This optimality condition is encoded through the  minimization of a certain notion of distortion (cf. Theorem \ref{thm:optimalthm}). The following result proved by Schoenberg in the early 1930s \cite{schoenberg1935remarks} (cf. Theorem \ref{thm:Schoenberg}) can be regarded as a cornerstone of the theory supporting cMDS:

\begin{quote}
\textit{A finite metric space $\mX=(X=\{x_1,\dots,x_n\},d_X)$ can be embedded into a Euclidean space if and only if the matrix $K_\mX:=((K_\mX)_{ij})$, where $$(K_\mX)_{ij}:=-\frac{1}{2}\left(d_X^2(x_i,x_j)-\frac{1}{n}\sum_{s=1}^n d_X^2(x_i,x_s)-\frac{1}{n}\sum_{r=1}^n d_X^2(x_r,x_j)+\frac{1}{n^2}\sum_{r,s=1}^n d_X^2(x_r,x_s)\right),$$ is positive semi-definite.}
\end{quote}

Schoenberg's theorem not only provides a direct method for ascertaining the embedability of a given finite metric space into some Euclidean space, but it also  reveals that the obstacle to achieving this dwells in the presence of negative eigenvalues of the matrix $K_\mX$. In this way it  
leads to one idea about how to produce an embedding of $\mX$ to Euclidean space by abandoning the negative eigenvalues of $K_\mX$, which is exactly what cMDS does (cf. Algorithm \ref{alg:cMDSfin}).

Schoenberg's theorem was independently rediscovered by scientists working in psychophysics  \cite{richardson1938multidimensional,young1938discussion} whose original motivation came from the area of \emph{sensory analysis}. More precisely, their goal was to estimate the dimensionality of the stimulus space, as well as the positions of the stimuli within this space \cite{mead1992review}. A few years later, Torgerson and Gower gave the formulation of cMDS that we know today and provided a supporting mathematical theory  \cite{torgerson1952multidimensional,torgerson1958theory,gower1966some}.

The list of dimension reduction techniques has been constantly growing. Some notable examples are isomap \cite{tenenbaum2000global}, locally linear embeddings \cite{roweis2000nonlinear}, Laplacian eigenmaps \cite{belkin2003laplacian}, diffusion maps \cite{coifman2006diffusion}, Hessian eigenmaps \cite{donoho2003hessian}, and more recently t-SNE \cite{tsne} and UMAP \cite{umap}. While their differences mainly come from how the distance (or, proximity) of  data points is computed, these methods all attempt to preserve the local structure of data.

\subsection*{Contributions}
The cMDS method has been mostly considered through its applicability to data analysis tasks, which means that its understanding has been mostly restricted to the case of finite datasets. Properties such as its \emph{stability} and \emph{consistency} (i.e. the suitably convergent behavior of the result of applying cMDS to a sample of a space as the cardinality of the sample increases) have not yet  been fully clarified.

In order to tackle these and other questions, in this paper we study a generalization of  Schoenberg's result to the \emph{metric measure space} setting \cite{gromov2007metric,memoli2011gromov} which will permit (1) modeling both these ideal limiting objects and their finite approximations within a flexible framework and (2) reasoning about such questions in an effective manner.

A \emph{metric measure space}, or mm-space for short, is a triple $(X,d_X,\mu_X)$ where $\mX=(X,d_X)$ is a  metric space and $\mu_X$ is a fully supported Borel probability measure on $X$. In this paper we will concentrate on the analysis of compact mm-spaces and will therefore slightly abuse terminology: mm-space will always refer to compact mm-spaces unless specified otherwise.

In this setting, we aim to quantify how much distortion will be incurred when trying to embed $\mX$ into  Euclidean space of a given dimension via a certain generalized formulation of the cMDS procedure. This extension of cMDS to mm-spaces was first considered in \cite{kassab2019multidimensional,adams2020multidimensional} where the authors were able to compute the spectrum of the cMDS operator for the case of the circle $\Sp^1$ with geodesic distance and normalized length measure. They also formulated a number of questions, such as extending the analysis to the case of spheres with arbitrary dimension (with geodesic distance and normalized volume measure) and  establishing the stability of the cMDS procedure. The results in the present paper fully answer those and other questions about cMDS on continuous spaces.

When studying the spectral properties of the generalized cMDS operator in the setting of possibly continuous mm-spaces (such as spheres with their geodesic distance), one must pay careful attention to the possibility that  the rank of the generalized cMDS operator can be infinite. This leads to considering  the collection of mm-spaces with \emph{traceable} cMDS operators. It turns out that the traceability of the cMDS operator is a natural condition  permitting  the study of the more specialized cMDS embedding into $\ell^2$, which is an infinite dimensional analogue of the usual cMDS embeddings into finite dimensional Euclidean spaces (see Proposition \ref{prop:traceclasscmds}). Understanding the cMDS embedding into $\ell^2$ is important because it expresses the limiting behaviour of the usual cMDS embedding into low-dimensional Euclidean spaces in the sense that, in the $\ell^2$ setting, cMDS loses the least amount of geometric information  (see item (2) of Proposition \ref{prop:generrorbdd}).

Ascertaining whether the cMDS operator associated to a given non-euclidean mm-space is traceable (i.e. in trace class) or not is far from trivial.  In this paper we do this for some concrete examples\footnote{Note that all finite mm-spaces induce traceable cMDS operators.}, including spheres of arbitrary dimension with geodesic distances, tori and metric graphs (with their geodesic distances). In particular, we compute all eigenvalues and eigenfunctions of the cMDS operator for spheres $\Sp^{d-1}$ with its geodesic distance $d_{\Sp^{d-1}}$ and  normalized volume measure (i.e. regarded as a mm-space), and establish that the cMDS operator is indeed trace class.

Furthermore, we fully characterize the metric  on $\Sp^{d-1}$ induced by the cMDS procedure:  we prove that this metric precisely coincides  with $ \sqrt{\pi\cdot d_{\Sp^{d-1}}}$ (cf. Theorem \ref{thm:cMDSSd-1dist}). 

In order to highlight the difficulty in establishing that a given mm-space induces a traceable cMDS operator, we describe the construction of a pathological  space which induces a non-traceable cMDS operator (cf. \S \ref{sec:non-trace-class}). We also formulate a number of questions suggested by our analysis.
 
Finally, we  prove results which clarify the stability properties of the cMDS procedure with respect to the Gromov-Wasserstein distance. 

\medskip

This paper is part of the first author's PhD dissertation \cite[Chapter 5]{lim2021thesis}.

\medskip
Kroshnin, Stepanov and Trevisan  contemporaneously and independently developed \cite{cmds-stepanov}  containing related results. Their paper has been recently published \cite{kroshnin2022infinite} (see \S\ref{sec:discussion} for a comparison).

\subsection*{Organization}
In \S\ref{sec:fms}, we provide some necessary definitions and results on the classical multidimensional scaling on finite metric spaces.

In \S\ref{sec:gencMDS}, we consider a generalized version of cMDS for mm-spaces and prove some basic properties. We define a notion of distortion involving the negative eigenvalues of the cMDS operator, and show that it somehow measures  non-flatness of a mm-space. Also, we prove that this generalized cMDS operator is the optimal choice in a certain rigorous sense. Next, we show that the generalized cMDS procedure (mapping any mm-space into $\ell^2$) is well-defined whenever the generalized cMDS operator is traceable. Moreover, we provide an example of a non-compact mm-space associated to a non-traceable cMDS operator. Lastly, we study how a certain notion of thickness of Euclidean data interacts with the embeddings produced by the cMDS operator.

In \S\ref{sec:ntrvalexmple}, we analyze the behaviour of the generalized cMDS operator on a  variety of  mm-spaces which include the circle $\Sp^1$,  products such as the $n$-torus, and also metric graphs.

In \S\ref{sec:Spd-1}, we carefully analyze the spectrum of the generalized cMDS operator for  higher dimensional spheres $\Sp^{d-1}$ ($d\geq 3$) with their geodesic distance and prove that this operator is in  trace class. Furthermore, we establish a precise relationship between the geodesic metric of $\Sp^{d-1}$ and the Euclidean metric resulting from applying the generalized cMDS to $\Sp^{d-1}$.

In \S\ref{sec:stability}, we prove that, in a certain sense, the generalized cMDS procedure is stable with respect to the Gromov-Wasserstein distance and briefly discuss an application to proving the consistency of cMDS.

In \S\ref{sec:discussion},we provide a detailed comparison between our paper and the closely related recent work \cite{kroshnin2022infinite} by Kroshnin, Stepanov and Trevisan. Additionally, we explore several potential avenues for future research that expand upon the concepts presented in this paper.

In Appendix \S\ref{sec:otherproofs} we provide several relegated proofs. In Appendix \S\ref{sec:sec:functanal} we provide an account of concepts from the  spectral theory of operators which are used throughout the paper. In Appendix \S\ref{sec:sphehar}, we collect some of standard results on the spherical harmonics which will be employed in order to analyze the cMDS embedding of $\Sp^{d-1}$.


\section{cMDS: The case of finite metric spaces}\label{sec:fms}
Throughout this paper, $\R^k$ will denote $k$-dimensional Euclidean space and, as an abuse of notation, $\Vert\cdot\Vert$ will  denote both of the Euclidean norm and the metric induced by it.\medskip

In this section we will provide a terse overview of cMDS on finite metric spaces. The main reference for this subsection is \cite[Chapter 14]{mardiamultivariate}. Throughout this section we will also intersperse restricted versions of novel results that we prove in this paper.
\medskip

For each $n\in\mathbb{Z}_{>0}$, let $\mathcal{M}_n$ be the collection of all finite metric spaces $\mX=(X,d_X)$ with $\card{X}=n$.  Schoenberg's Theorem (Theorem \ref{thm:Schoenberg}) provides a criterion for checking whether a finite metric space $\mX=(X,d_X)\in\mathcal{M}_n$ admits an isometric embedding into Euclidean space $\R^k$. We need some preparation before stating this theorem.

\subsection{The formulation of cMDS on finite metric spaces}\label{sec:sec:cmdsfin}
For a finite metric space $\mX=(X=\{x_1,\dots,x_n\},d_X)\in\mathcal{M}_n$, consider an $n$ by $n$ matrix $A_\mX$ where $(A_\mX)_{ij}:=d_X^2(x_i,x_j)$ for each $i,j=1,\dots,n$. Next, construct another $n$ by $n$ matrix
$$K_\mX:=-\frac{1}{2}H_nA_\mX H_n$$
where $H_n:=I_n-\frac{1}{n}e_ne_n^\mathrm{T}$, $I_n$ is the $n$ by $n$ identity matrix, and $e_n$ is the $n$ by $1$ column vector where every entry is $1$.

Let us collect some of simple properties of $A_\mX$ and $K_\mX$.
\begin{lemma}\label{lemma:simpletech}
$\mX=(X,d_X)\in\mathcal{M}_n$ is given. $A_\mX$ and $K_\mX$ are as before. Then,
\begin{enumerate}
    \item $A_\mX$ and $K_\mX$ are symmetric matrices. 
    
    \item $(K_\mX)_{ij}=-\frac{1}{2}\left((A_\mX)_{ij}-\frac{1}{n}\sum_{s=1}^n (A_\mX)_{is}-\frac{1}{n}\sum_{r=1}^n (A_\mX)_{rj}+\frac{1}{n^2}\sum_{r,s=1}^n (A_\mX)_{rs}\right)$ for each $i,j=1,\dots,n$.
    
    \item $(A_\mX)_{ij}=d_X^2(x_i,x_j)=(K_\mX)_{ii}+(K_\mX)_{jj}-2(K_\mX)_{ij}$ for each $i,j=1,\dots,n$.
    
    \item $e_n$ is an eigenvector of $K_\mX$ with zero eigenvalue. In particular, $\mathrm{dim}(\ker( K_\mX))\geq 1$.
    
    \item If $v\in\R^n$ is an eigenvector for a nonzero eigenvalue of $K_\mX$, then $\sum_{i=1}^nv(i)=0$.
\end{enumerate}
\end{lemma}

\begin{remark}
Observe that, by item (2) of Lemma \ref{lemma:simpletech}, we  have that the trace of $K_\mX$ satisfies
$$\tr(K_\mX)=\frac{1}{2n}\sum_{i,j=1}^n d_X^2(x_i,x_j).$$
Hence, for any non-singleton metric space $\mX=(X,d_X)$, $K_\mX$ must have at least one strictly positive eigenvalue.
\end{remark}

\begin{theorem}[Schoenberg's Theorem \cite{schoenberg1935remarks,young1938discussion,mardiamultivariate}]\label{thm:Schoenberg}
Let $(X=\{x_1,\dots,x_n\},d_X)\in\mathcal{M}_n$ and define $K_\mX$ as before. Then, $X$ is isometric to some $\{z_1,\dots,z_n\}\subset\R^k$ for some positive integer $k$ if and only if $K_\mX$ is positive semi-definite.
\end{theorem}

\begin{remark}\label{rmk:rnkEuclid}
The proof of this theorem reveals that, if $X$ is isometric to $Z=\{z_1,\dots,z_n\}\subset\R^k$, then the rank of $K_\mX$ is less than or equal to $k$ since $K_\mX=(ZH_n)^T(ZH_n)$ where $ZH_n$ is a $k$ by $n$ matrix (by interpreting $Z$ as a $k$ by $n$ matrix with columns $z_1^\mathrm{T},\dots,z_n^\mathrm{T}$).
\end{remark}

\begin{definition}
For each square symmetric matrix $A$ of size $n$,
$$\pr(A):=\text{the number of positive eigenvalues of }A\text{ counted with their multiplicity.}$$
\end{definition}

Of course, not every finite metric space can be isometrically embedded into Euclidean space. However, Schoenberg's result shows that the negative eigenvalues of $K_\mX$ are the obstruction to achieving this. Indeed, by eliminating such negative eigenvalues and by keeping the $k$ largest eigenvalues and their corresponding orthonormal eigenvectors, the classical cMDS procedure \cite{mardiamultivariate} (see item (6) in Algorithm \ref{alg:cMDSfin} below) gives a precise way of building a cMDS embedding  $\Phi_{\mX\!\!,k}:X\rightarrow \R^k$ together with its induced finite metric space $\widehat{\mX}_k:=(\Phi_{\mX\!\!,k}(X),\Vert\cdot\Vert)$: 

\begin{algorithm}
\caption{cMDS embeding for finite metric spaces}\label{alg:cMDSfin}
\begin{algorithmic}[1]
    \State \textbf{Input}: $\mX=(X,d_X)\in\mathcal{M}_n$ and a positive integer $k\leq \pr(K_\mX)$.
    
    \State Construct $A_\mX$ and $K_\mX$ as before.
    
    \State Compute the eigenvalues $\lambda_1\geq\dots\geq\lambda_{\pr(K_\mX)}>0=\lambda_{\pr(K_\mX)+1}\geq\dots\geq\lambda_n$ and corresponding orthonormal eigenvectors $v_1,\dots,v_n$ of $K_\mX$.
    
    \State Construct $\widehat{\Lambda}_k^{\frac{1}{2}}\,\widehat{V}_k^T$ where $\widehat{\Lambda}_k:=\mathrm{diag}(\lambda_1,\lambda_2,\dots,\lambda_k)$ and $\widehat{V}_k:=[v_1\dots v_k]$. 
    
    \State \textbf{Let $\Phi_{\mX\!\!,k}(x_i)\in\R^k$ be the $i$-th column vector of $\widehat{\Lambda}_k^{\frac{1}{2}}\,\widehat{V}_k^T$ for each $i=1,\dots,n$.}
    
    \State \textbf{Output}: the embedding $\Phi_{\mX\!\!,k}:X\rightarrow \R^k$, and its induced finite metric space $\widehat{\mX}_k.$
\end{algorithmic}
\end{algorithm}

For the case when $\mX\in\mathcal{M}_n$ and $k=n$, often we will use the notation $\widehat{\mX}$ instead of $\widehat{\mX}_n$.

\begin{remark}\label{rmk:cMDSfacts}
Here are some simple facts from the above cMDS procedure.
\begin{enumerate}
    \item $\Phi_{\mX\!\!,k}(x_i)=(\lambda_1^{\frac{1}{2}}v_1(i),\dots,\lambda_k^{\frac{1}{2}}v_k(i))^T$ for each $i=1,\dots,n$.
    
    \item The centroid of $\Phi_{\mX\!\!,k}(x_1),\dots,\Phi_{\mX\!\!,k}(x_n)$ is the origin $(0,....,0)\in \R^k$ by Lemma \ref{lemma:simpletech} item (5).
    
    \item $\lambda_1\geq\dots\geq\lambda_k$ are all positive eigenvalues of $K_{\widehat{\mX}_k}$ and all the other eigenvalues are zero.
\end{enumerate}
\end{remark}

\begin{example}\label{example:simplest}
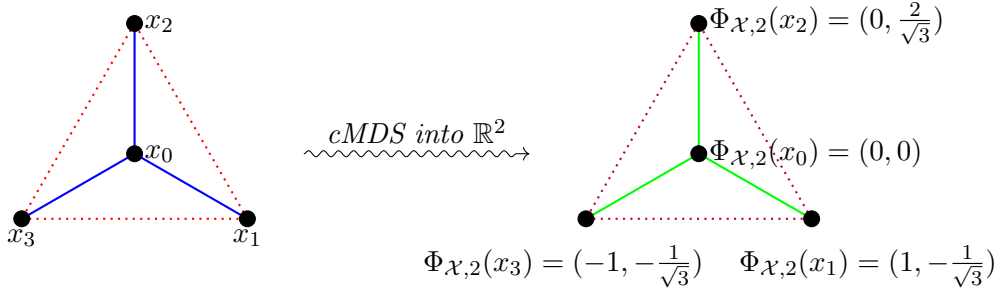
\begin{figure}[h!]
\centering
\begin{tikzpicture}[scale=1.5]
\node (a1) at (0,0) {};
\node (b1) at (0,1.155) {};
\node (c1) at (-1,-0.577) {};
\node (d1) at (1,-0.577) {};

\draw [color= blue, thick] (a1.center)--(b1.center);
\draw [color= blue, thick] (a1.center)--(c1.center); 
\draw [color= blue, thick] (a1.center)--(d1.center);
\draw [dotted,color= red, thick] (b1.center)--(c1.center);
\draw [dotted,color= red, thick] (c1.center)--(d1.center); 
\draw [dotted,color= red, thick] (d1.center)--(b1.center);

\filldraw (a1) circle[radius=2pt] node[anchor=west] {$x_0$};
\filldraw (b1) circle[radius=2pt] node[anchor=west] {$x_2$};
\filldraw (c1) circle[radius=2pt] node[anchor=north] {$x_3$};
\filldraw (d1) circle[radius=2pt] node[anchor=north] {$x_1$};

\node (a2) at (5,0) {};
\node (b2) at (5,1.155) {};
\node (c2) at (4,-0.577) {};
\node (d2) at (6,-0.577) {};

\draw [color= green, thick] (a2.center)--(b2.center);
\draw [color= green, thick] (a2.center)--(c2.center); 
\draw [color= green, thick] (a2.center)--(d2.center);
\draw [dotted,color= purple, thick] (b2.center)--(c2.center);
\draw [dotted,color= purple, thick] (c2.center)--(d2.center); 
\draw [dotted,color= purple, thick] (d2.center)--(b2.center);

\filldraw (a2) circle[radius=2pt] node[anchor=west] {$\Phi_{\mX\!\!,2}(x_0)=(0,0)$};
\filldraw (b2) circle[radius=2pt] node[anchor=west] {$\Phi_{\mX\!\!,2}(x_2)=(0,\frac{2}{\sqrt{3}})$};
\filldraw (c2) circle[radius=2pt] node at (3.8,-1) {$\Phi_{\mX\!\!,2}(x_3)=(-1,-\frac{1}{\sqrt{3}})$};
\filldraw (d2) circle[radius=2pt] node at (6.5,-1) {$\Phi_{\mX\!\!,2}(x_1)=(1,-\frac{1}{\sqrt{3}})$};

\draw [->,snake=snake,
segment amplitude=.4mm,
segment length=2mm,
line after snake=1mm] (1.5,0) -- (3.5,0)
node [above,text width=3cm,text centered,midway] {cMDS into $\R^2$};
\end{tikzpicture} 
\caption{Description of the cMDS procedure into $\R^2$ from Example \ref{example:simplest}. Blue lines indicate length $1$,  red lines indicate length $2$,  green lines indicate length $\frac{2}{\sqrt{3}}$, whereas  dotted purple lines indicate length $2$.}\label{fig:cMDSexample}
\end{figure}
Consider $(X=\{x_0,x_1,x_2,x_3\},d_X)\in\mathcal{M}_4$ such that $d_X(x_0,x_i)=1$ for each $i=1,2,3$ and $d_X(x_i,x_j)=2$ for each $i\neq j\in\{1,2,3\}$. Then,
$$A_\mX=\left(\begin{array}{cccc} 0 & 1 &1 &1\\ 1 & 0 &4 &4\\ 1 & 4 &0 &4\\ 1 & 4 &4 &0\end{array}\right)\mbox{  and  }
K_\mX=\frac{1}{16}\cdot\left(\begin{array}{cccc} -3 & 1 &1 &1\\ 1 & 21 &-11 &-11\\ 1 & -11 &21 &-11\\ 1 & -11 &-11 &21\end{array}\right)$$
The eigenvalues of $K_\mX$ are $2,2,0,-1/4$. Their corresponding mutually orthogonal eigenvectors are $(0,1,0,-1),(0,-\frac{1}{2},1,-\frac{1}{2}),(1,1,1,1),(-3,1,1,1)$. Therefore, the result of cMDS into $\R^2$ is the following:
$\Phi_{\mX\!\!,2}(x_0)=(0,0),\Phi_{\mX\!\!,2}(x_1)=(1,-\frac{1}{\sqrt{3}}),\Phi_{\mX\!\!,2}(x_2)=(0,\frac{2}{\sqrt{3}})$, and $\Phi_{\mX\!\!,2}(x_3)=(-1,-\frac{1}{\sqrt{3}})$.
\end{example}

Moreover, this $\widehat{\mX}_k$ is the optimal one in the following sense identified by Mardia \cite{mardia1978some}.\footnote{Note that in contrast to the `metric' formulation of multidimensional scaling, the underlying optimality criterion does not directly compare distance matrices, but instead compares the Gram matrices they induce.}

\begin{theorem}[\cite{mardiamultivariate,mardia1978some}]\label{thm:optimalthm}
Suppose $\mX=(X,d_X)\in\mathcal{M}_n$ and $k\leq \pr(K_\mX)$ are given. Then, for any positive semi-definite matrix $A$ of size $n$ with $\mathrm{rank}(A)\leq k$, we have
$$\Vert K_\mX-A \Vert_{\mathrm{F}}\geq\Vert K_\mX-K_{\widehat{\mX}_k} \Vert_{\mathrm{F}}$$
where $\Vert\cdot\Vert_{\mathrm{F}}$ is the Frobenius norm.
\end{theorem}

Theorem \ref{thm:optimalthm} will be suitably generalized to the case of an arbitrary mm-space in Theorem \ref{thm:optimalthmgen}. This generalization will be crucial for establishing the stability results in \S \ref{sec:stability}.

\subsection{The metric distortion incurred by cMDS}\label{sec:sec:metdist}
Now we know that for a given $k\in\mathbb{N}$ the cMDS method gives an optimal embedding  into $\R^k$ (with respect to the Frobenius norm on the induced), the next natural goal is to estimate the metric distortion incurred during this process. It will turn out that there is a distinct sense in which distortion is controlled. 

To explain this phenomenon  assume that $k=\pr(K_\mX).$  We consider two different types of distortion: the $L^\infty$-metric distortion is the number $$\max_{i,j}\left|\Vert\Phi_{\mX\!\!,\pr(K_\mX)}(x_i)-\Phi_{\mX\!\!,\pr(K_\mX)}(x_j)\Vert-d_X(x_i,x_j)\right|$$
whereas the $L^2$-metric distortion is the number
$$\dis(\mX):=\left(\frac{1}{n^2} \sum_{i,j=1}^n\left|\Vert\Phi_{\mX\!\!,\pr(K_\mX)}(x_i)-\Phi_{\mX\!\!,\pr(K_\mX)}(x_j)\Vert^2-d_X^2(x_i,x_j)\right|\right)^{\frac{1}{2}}.$$

In what follows, for a  symmetric square matrix $A$, let $\trng(A)$ be the sum of the absolute values of the negative eigenvalues of $A$. Suppose $\mX=(X=\{x_1,\dots,x_n\},d_X)\in\mathcal{M}_n$ is given. Let $$\lambda_1\geq\dots\geq\lambda_{\pr(K_\mX)}>0=\lambda_{\pr(K_\mX)+1}\geq\dots\geq\lambda_n$$ be the eigenvalues of $K_\mX$ and $\{v_m\}_{m=1}^n$ be the corresponding orthonormal eigenvectors. Then, by item (3) of Lemma \ref{lemma:simpletech} and item (1) of Remark \ref{rmk:cMDSfacts}, it is easy to check that for any $i,j\in\{1,\ldots,n\}$:

\begin{itemize}
    \item $d_X^2(x_i,x_j)=\sum_{m=1}^n\lambda_m\,(v_m(i)-v_m(j))^2$, and
    
    \item $\Vert \Phi_{\mX\!\!,\pr(K_\mX)}(x_i)-\Phi_{\mX\!\!,\pr(K_\mX)}(x_j) \Vert^2=\sum_{m=1}^{\pr(K_\mX)}\lambda_m\,(v_m(i)-v_m(j))^2$.
\end{itemize}

 Because of these observations and through a direct calculation, it holds that:
\begin{equation}\label{eq:l-inf-fin}
0\leq\Vert\Phi_{\mX\!\!,\pr(K_\mX)}(x_i)-\Phi_{\mX\!\!,\pr(K_\mX)}(x_j)\Vert-d_X(x_i,x_j)\leq\sqrt{2\,\trng(K_\mX)},\end{equation}
for every $i,j=1,\dots,n$. The leftmost inequality implies that the cMDS embedding $\Phi_{\mX\!\!,\pr(K_\mX)}$ is \emph{non-contractive}, and the rightmost inequality implies that the $L^\infty$-metric distortion of $\Phi_{\mX\!\!,\pr(K_\mX)}$ is upper bounded by $\sqrt{2\,\trng(K_\mX)}$. However, the   two examples below demonstrate that $\sqrt{2\,\trng(K_\mX)}$ does not scale well with respect to the cardinality of $X$.

In contrast, the  $L^2$-metric distortion satisfies the \emph{equality} below (which is also derived through direct calculation via the spectral expansion mentioned above and by the leftmost inequality in equation (\ref{eq:l-inf-fin})) and is well behaved (see examples below):
$$\dis(\mX)=\sqrt{\frac{2\,\trng(K_\mX)}{n}}.$$
Proposition \ref{prop:generrorbdd} will  generalize both the non-contractive property and the relation between the $L^2$-metric distortion and the negative trace to mm-spaces satisfying suitable conditions.

\begin{example}[Paley graphs]\label{ex:PaleyL2natural}
Fix a prime $q$ that is congruent to $1$ modulo $4$ and let $G_q$ be the \textbf{Paley graph} with  $q$ vertices: two vertices $i,j \in\{1,\ldots,q\}$ form an edge if and only if the difference $i-j$ is a nonzero square in the field $\mathbb{F}_q$. It is well-known that the Paley graph $G_q$ on $q$ vertices is strongly regular such that: (1) the degree of each vertex equals  $(q-1)/2$, (2) each pair of adjacent vertices has $(q-5)/4$ common neighbors, and (3) each pair of nonadjacent neighbors has $(q-1)/4$ common neighbors \cite[\S 9.1.2]{brouwer2011spectra}. In particular, all shortest path distances on $G_q$ are $0,1,$ or $2$; see Figure \ref{fig:paley}.

Now, let $\mathcal{G}_q:=(V(G_q):=\{x_1,\dots,x_q\},d_{G_q})\in\mathcal{M}_q$ be the metric space such that $V(G_q)$ is the set of vertices of $G_q$ and $d_{G_q}$ is the shortest path metric. Then, by  direct computation one has the following equality (cf. Proposition \ref{prop:homogeneousspace}):
$$(K_{\mathcal{G}_q})_{ij}=-\frac{1}{2}(d_{G_q}^2(x_i,x_j)-c_q)$$
for any $i,j=1,\dots,q$ where $c_q:=\frac{5(q-1)}{2q}$.

Another well-known fact about $G_q$ is that the adjacency matrix of $G_q$ has eigenvalues $\frac{q-1}{2}$ (with multiplicity $1$) and $\frac{-1\pm\sqrt{q}}{2}$ (with multiplicity $\frac{q-1}{2}$ each) \cite[Proposition 9.1.1]{brouwer2011spectra}. Moreover, the eigenvectors of its adjacency matrix are either constant functions or are orthogonal to constant functions. The constant eigenvectors correspond to the eigenvalue $\frac{q-1}{2}$, whereas  eigenvectors orthogonal to constant functions correspond to the eigevalue pair $\frac{-1\pm\sqrt{q}}{2}$. 

With these facts and the strong regularity of $G_q$, one can easily check that (1) the eigenvalues of $K_{\mathcal{G}_q}$ are $0$ (with multiplicity $1$) and $\frac{5\pm 3\sqrt{q}}{4}$ (each with multiplicity $\frac{q-1}{2}$) and (2)  eigenvectors of $K_{\mathcal{G}_q}$ are either constant functions or orthogonal to constant functions: Constant eigenvectors correspond to the eigenvalue $0$, and the eigenvectors orthogonal to constant functions correspond to the eigenvalue pair $\frac{5\pm 3\sqrt{q}}{4}$. At any rate, it follows that $$\trng(K_{\mathcal{G}_q})=\frac{(3\sqrt{q}-5)(q-1)}{8}.$$

Now, for each $k\geq 1$,  let $\mX_k:=(X_k,d_{X_k})\in\mathcal{M}_{kq}$ be the finite metric space defined as follows. First let $\mathcal{G}_q^i:=(V(G_q^i),d_{G_q}^i)$ for $i=1,\dots,k$ be  $k$ copies of $\mathcal{G}_q$. Also, we let

\begin{itemize}
    \item $X_k:=\sqcup_{i=1}^k V(G_q^i)$,
    
    \item $d_{X_k}|_{V(G_q^i)\times V(G_q^i)} = d_{G_q}^i$ for each $i$, whereas, for $i\neq j$, we let $d_{X_k}|_{V(G_q^i)\times V(G_q^j)}\equiv\sqrt{c_q}$. 
\end{itemize}
Note that $d_{X_k}$ indeed defines a metric on $X_k$: we only need to check the triangle inequality for $d_{X_k}$. Since $\sqrt{c_q}\geq 1$ and the value of $d_{G_q}$ is at most $2$, the triangle inequality of $d_{X_k}$  must hold.

A calculation shows that for every eigenvector $v$ of $K_{\mathcal{G}_q}$ with eigenvalue $\lambda$, and for each $i=1,\dots,k$, the vector $f_i:X_k\rightarrow \R$ defined by
\begin{align*}
    f_i(x):=\begin{cases}v(x)&\text{if }x\in V(G_q^i)\\ 0&\text{otherwise}\end{cases}
\end{align*}
is an eigenvector of $K_{\mX_k}$ associated to the eigenvalue $\lambda$. This implies that $0$ (with multiplicity at least $k$) and $\frac{5\pm 3\sqrt{q}}{4}$ (each with multiplicity at least $\frac{k(q-1)}{2}$) are eigenvalues of $K_{\mX_k}$. Furthermore, since the total number of eigenvalues of $K_{\mX_k}$ is exactly $kq$, it must be that both $0$ is an eigenvalue with multiplicity exactly $k$ and $\frac{k(q-1)}{2}$ is an eigenvalue with multiplicity exactly $\frac{k(q-1)}{2}$. Hence, $\trng(K_{\mX_k})=k\cdot\trng(K_{\mathcal{G}_q})$.\medskip

Thus, $\sqrt{2\,\trng(K_{\mX_k})}\rightarrow\infty$ as $k\rightarrow\infty$, whereas $\sqrt{\frac{2\,\trng(K_{\mX_k})}{kq}}=\sqrt{\frac{2\,\trng(K_{\mathcal{G}_q})}{q}}<\infty$ for  $k\geq 1$.
\end{example}

In \S\ref{sec:non-trace-class} Paley graphs will provide building blocks for constructing  a discrete but infinite metric (measure) space whose cMDS operator is ill behaved. 

\begin{figure}
\centering
\begin{tikzpicture}
\node[draw,circle,minimum size=10,inner sep=0] (0) at (0.0,2.23606797749979) {};
\node[draw,circle,minimum size=10,inner sep=0] (1) at (2.1266270208801,0.6909830056250527) {};
\node[draw,circle,minimum size=10,inner sep=0] (2) at (1.3143277802978344,-1.8090169943749472) {};
\node[draw,circle,minimum size=10,inner sep=0] (3) at (-1.314327780297834,-1.8090169943749477) {};
\node[draw,circle,minimum size=10,inner sep=0] (4) at (-2.1266270208801,0.6909830056250522) {};
\draw (0) -- (1);
\draw (0) -- (4);
\draw (1) -- (2);
\draw (2) -- (3);
\draw (3) -- (4);
\end{tikzpicture}
\begin{tikzpicture}[scale=0.6]
\node[draw,circle,minimum size=10,inner sep=0] (0) at (0.0,3.605551275463989) {};
\node[draw,circle,minimum size=10,inner sep=0] (1) at (1.6755832257000804,3.1925571026612056) {};
\node[draw,circle,minimum size=10,inner sep=0] (2) at (2.967310527359157,2.0481865721226473) {};
\node[draw,circle,minimum size=10,inner sep=0] (3) at (3.579262747168659,0.4346011812347757) {};
\node[draw,circle,minimum size=10,inner sep=0] (4) at (3.371249006393944,-1.2785461027619776) {};
\node[draw,circle,minimum size=10,inner sep=0] (5) at (2.390922746209171,-2.6987938827668176) {};
\node[draw,circle,minimum size=10,inner sep=0] (6) at (0.862864898610517,-3.5007805082218253) {};
\node[draw,circle,minimum size=10,inner sep=0] (7) at (-0.8628648986105145,-3.5007805082218257) {};
\node[draw,circle,minimum size=10,inner sep=0] (8) at (-2.3909227462091693,-2.698793882766819) {};
\node[draw,circle,minimum size=10,inner sep=0] (9) at (-3.3712490063939438,-1.278546102761979) {};
\node[draw,circle,minimum size=10,inner sep=0] (10) at (-3.5792627471686593,0.4346011812347724) {};
\node[draw,circle,minimum size=10,inner sep=0] (11) at (-2.967310527359158,2.048186572122646) {};
\node[draw,circle,minimum size=10,inner sep=0] (12) at (-1.675583225700083,3.1925571026612043) {};
\draw (0) -- (1);
\draw (0) -- (3);
\draw (0) -- (4);
\draw (0) -- (9);
\draw (0) -- (10);
\draw (0) -- (12);
\draw (1) -- (2);
\draw (1) -- (4);
\draw (1) -- (5);
\draw (1) -- (10);
\draw (1) -- (11);
\draw (2) -- (3);
\draw (2) -- (5);
\draw (2) -- (6);
\draw (2) -- (11);
\draw (2) -- (12);
\draw (3) -- (4);
\draw (3) -- (6);
\draw (3) -- (7);
\draw (3) -- (12);
\draw (4) -- (5);
\draw (4) -- (7);
\draw (4) -- (8);
\draw (5) -- (6);
\draw (5) -- (8);
\draw (5) -- (9);
\draw (6) -- (7);
\draw (6) -- (9);
\draw (6) -- (10);
\draw (7) -- (8);
\draw (7) -- (10);
\draw (7) -- (11);
\draw (8) -- (9);
\draw (8) -- (11);
\draw (8) -- (12);
\draw (9) -- (10);
\draw (9) -- (12);
\draw (10) -- (11);
\draw (11) -- (12);
\end{tikzpicture}
\begin{tikzpicture}[scale=0.4]
\node[draw,circle,minimum size=10,inner sep=0] (0) at (0.0,5.385164807134504) {};
\node[draw,circle,minimum size=10,inner sep=0] (1) at (1.1576512491986188,5.259262646534101) {};
\node[draw,circle,minimum size=10,inner sep=0] (2) at (2.2611720126216626,4.887443209832366) {};
\node[draw,circle,minimum size=10,inner sep=0] (3) at (3.258962885846707,4.287092360641849) {};
\node[draw,circle,minimum size=10,inner sep=0] (4) at (4.104368276606656,3.486281837428654) {};
\node[draw,circle,minimum size=10,inner sep=0] (5) at (4.757857968430178,2.522456650221259) {};
\node[draw,circle,minimum size=10,inner sep=0] (6) at (5.188875509629234,1.4406841935587253) {};
\node[draw,circle,minimum size=10,inner sep=0] (7) at (5.377266999018945,0.29154694521086294) {};
\node[draw,circle,minimum size=10,inner sep=0] (8) at (5.31422345993755,-0.8712227142639043) {};
\node[draw,circle,minimum size=10,inner sep=0] (9) at (5.002692738204636,-1.9932549679141913) {};
\node[draw,circle,minimum size=10,inner sep=0] (10) at (4.457241664126903,-3.0220848346085916) {};
\node[draw,circle,minimum size=10,inner sep=0] (11) at (3.703374923702898,-3.909605373242746) {};
\node[draw,circle,minimum size=10,inner sep=0] (12) at (2.776342487852145,-4.6143171098383515) {};
\node[draw,circle,minimum size=10,inner sep=0] (13) at (1.7194913629524755,-5.103268506823038) {};
\node[draw,circle,minimum size=10,inner sep=0] (14) at (0.582238732998539,-5.353596740304247) {};
\node[draw,circle,minimum size=10,inner sep=0] (15) at (-0.5822387329985378,-5.353596740304247) {};
\node[draw,circle,minimum size=10,inner sep=0] (16) at (-1.7194913629524764,-5.103268506823038) {};
\node[draw,circle,minimum size=10,inner sep=0] (17) at (-2.7763424878521463,-4.614317109838351) {};
\node[draw,circle,minimum size=10,inner sep=0] (18) at (-3.703374923702899,-3.9096053732427456) {};
\node[draw,circle,minimum size=10,inner sep=0] (19) at (-4.457241664126902,-3.022084834608593) {};
\node[draw,circle,minimum size=10,inner sep=0] (20) at (-5.002692738204636,-1.9932549679141913) {};
\node[draw,circle,minimum size=10,inner sep=0] (21) at (-5.31422345993755,-0.8712227142639044) {};
\node[draw,circle,minimum size=10,inner sep=0] (22) at (-5.377266999018945,0.29154694521086283) {};
\node[draw,circle,minimum size=10,inner sep=0] (23) at (-5.188875509629233,1.4406841935587265) {};
\node[draw,circle,minimum size=10,inner sep=0] (24) at (-4.757857968430177,2.52245665022126) {};
\node[draw,circle,minimum size=10,inner sep=0] (25) at (-4.104368276606654,3.4862818374286553) {};
\node[draw,circle,minimum size=10,inner sep=0] (26) at (-3.2589628858467057,4.28709236064185) {};
\node[draw,circle,minimum size=10,inner sep=0] (27) at (-2.2611720126216643,4.887443209832365) {};
\node[draw,circle,minimum size=10,inner sep=0] (28) at (-1.1576512491986204,5.259262646534101) {};
\draw (0) -- (1);
\draw (0) -- (4);
\draw (0) -- (5);
\draw (0) -- (6);
\draw (0) -- (7);
\draw (0) -- (9);
\draw (0) -- (13);
\draw (0) -- (16);
\draw (0) -- (20);
\draw (0) -- (22);
\draw (0) -- (23);
\draw (0) -- (24);
\draw (0) -- (25);
\draw (0) -- (28);
\draw (1) -- (2);
\draw (1) -- (5);
\draw (1) -- (6);
\draw (1) -- (7);
\draw (1) -- (8);
\draw (1) -- (10);
\draw (1) -- (14);
\draw (1) -- (17);
\draw (1) -- (21);
\draw (1) -- (23);
\draw (1) -- (24);
\draw (1) -- (25);
\draw (1) -- (26);
\draw (2) -- (3);
\draw (2) -- (6);
\draw (2) -- (7);
\draw (2) -- (8);
\draw (2) -- (9);
\draw (2) -- (11);
\draw (2) -- (15);
\draw (2) -- (18);
\draw (2) -- (22);
\draw (2) -- (24);
\draw (2) -- (25);
\draw (2) -- (26);
\draw (2) -- (27);
\draw (3) -- (4);
\draw (3) -- (7);
\draw (3) -- (8);
\draw (3) -- (9);
\draw (3) -- (10);
\draw (3) -- (12);
\draw (3) -- (16);
\draw (3) -- (19);
\draw (3) -- (23);
\draw (3) -- (25);
\draw (3) -- (26);
\draw (3) -- (27);
\draw (3) -- (28);
\draw (4) -- (5);
\draw (4) -- (8);
\draw (4) -- (9);
\draw (4) -- (10);
\draw (4) -- (11);
\draw (4) -- (13);
\draw (4) -- (17);
\draw (4) -- (20);
\draw (4) -- (24);
\draw (4) -- (26);
\draw (4) -- (27);
\draw (4) -- (28);
\draw (5) -- (6);
\draw (5) -- (9);
\draw (5) -- (10);
\draw (5) -- (11);
\draw (5) -- (12);
\draw (5) -- (14);
\draw (5) -- (18);
\draw (5) -- (21);
\draw (5) -- (25);
\draw (5) -- (27);
\draw (5) -- (28);
\draw (6) -- (7);
\draw (6) -- (10);
\draw (6) -- (11);
\draw (6) -- (12);
\draw (6) -- (13);
\draw (6) -- (15);
\draw (6) -- (19);
\draw (6) -- (22);
\draw (6) -- (26);
\draw (6) -- (28);
\draw (7) -- (8);
\draw (7) -- (11);
\draw (7) -- (12);
\draw (7) -- (13);
\draw (7) -- (14);
\draw (7) -- (16);
\draw (7) -- (20);
\draw (7) -- (23);
\draw (7) -- (27);
\draw (8) -- (9);
\draw (8) -- (12);
\draw (8) -- (13);
\draw (8) -- (14);
\draw (8) -- (15);
\draw (8) -- (17);
\draw (8) -- (21);
\draw (8) -- (24);
\draw (8) -- (28);
\draw (9) -- (10);
\draw (9) -- (13);
\draw (9) -- (14);
\draw (9) -- (15);
\draw (9) -- (16);
\draw (9) -- (18);
\draw (9) -- (22);
\draw (9) -- (25);
\draw (10) -- (11);
\draw (10) -- (14);
\draw (10) -- (15);
\draw (10) -- (16);
\draw (10) -- (17);
\draw (10) -- (19);
\draw (10) -- (23);
\draw (10) -- (26);
\draw (11) -- (12);
\draw (11) -- (15);
\draw (11) -- (16);
\draw (11) -- (17);
\draw (11) -- (18);
\draw (11) -- (20);
\draw (11) -- (24);
\draw (11) -- (27);
\draw (12) -- (13);
\draw (12) -- (16);
\draw (12) -- (17);
\draw (12) -- (18);
\draw (12) -- (19);
\draw (12) -- (21);
\draw (12) -- (25);
\draw (12) -- (28);
\draw (13) -- (14);
\draw (13) -- (17);
\draw (13) -- (18);
\draw (13) -- (19);
\draw (13) -- (20);
\draw (13) -- (22);
\draw (13) -- (26);
\draw (14) -- (15);
\draw (14) -- (18);
\draw (14) -- (19);
\draw (14) -- (20);
\draw (14) -- (21);
\draw (14) -- (23);
\draw (14) -- (27);
\draw (15) -- (16);
\draw (15) -- (19);
\draw (15) -- (20);
\draw (15) -- (21);
\draw (15) -- (22);
\draw (15) -- (24);
\draw (15) -- (28);
\draw (16) -- (17);
\draw (16) -- (20);
\draw (16) -- (21);
\draw (16) -- (22);
\draw (16) -- (23);
\draw (16) -- (25);
\draw (17) -- (18);
\draw (17) -- (21);
\draw (17) -- (22);
\draw (17) -- (23);
\draw (17) -- (24);
\draw (17) -- (26);
\draw (18) -- (19);
\draw (18) -- (22);
\draw (18) -- (23);
\draw (18) -- (24);
\draw (18) -- (25);
\draw (18) -- (27);
\draw (19) -- (20);
\draw (19) -- (23);
\draw (19) -- (24);
\draw (19) -- (25);
\draw (19) -- (26);
\draw (19) -- (28);
\draw (20) -- (21);
\draw (20) -- (24);
\draw (20) -- (25);
\draw (20) -- (26);
\draw (20) -- (27);
\draw (21) -- (22);
\draw (21) -- (25);
\draw (21) -- (26);
\draw (21) -- (27);
\draw (21) -- (28);
\draw (22) -- (23);
\draw (22) -- (26);
\draw (22) -- (27);
\draw (22) -- (28);
\draw (23) -- (24);
\draw (23) -- (27);
\draw (23) -- (28);
\draw (24) -- (25);
\draw (24) -- (28);
\draw (25) -- (26);
\draw (26) -- (27);
\draw (27) -- (28);
\end{tikzpicture}

\caption{$G_5$, $G_{17}$, and $G_{29}$.\label{fig:paley}}
\end{figure}
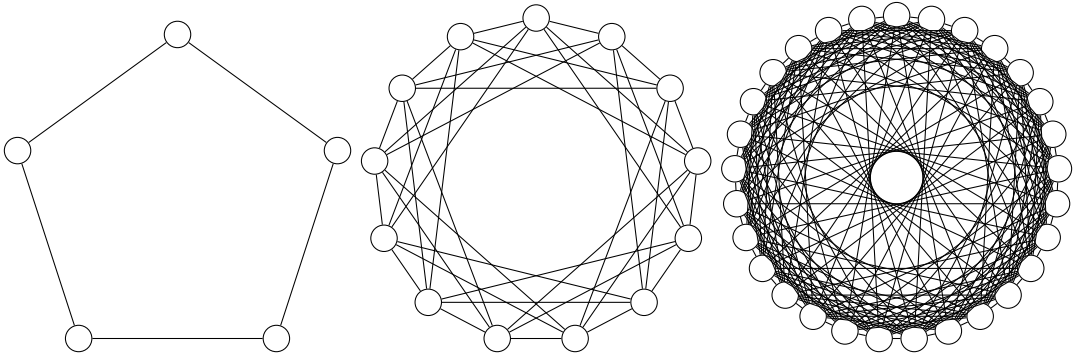

\begin{example}[vertex set of regular polygons]\label{ex:reg-polys}
Let $\mathcal{C}_n=(C_n,d_{C_n})\in\mathcal{M}_n$ be a finite metric space such that $C_n\subset\Sp^1$ is a set of $n$ evenly spaced points on $\Sp^1$, and $d_{C_n}$ is the metric inherited from the geodesic metric $d_{\Sp^1}$ of $\Sp^1$. Kassab already pointed out in \cite[Lemma 7.2.1]{kassab2019multidimensional} that, since   $d_{\mathcal{C}_n}$ is a circulant matrix, then $\lambda_k(n):=-\frac{1}{2}\sum_{j=0}^{n-1}d_{0j}^2\omega_{n}^{jk}$ for $k=1,\dots,n-1$ are all nonzero eigenvalues of $K_{\mathcal{C}_n}$ where $\frac{n}{2\pi}d_{0j}:=\begin{cases}j&\text{for }0\leq j\leq\lfloor\frac{n}{2}\rfloor\\ n-j&\text{for }\lceil\frac{n}{2}\rceil\leq j\leq n-1\end{cases}$ and $\omega_n:=e^{\frac{2\pi i}{n}}$. We now push this observation further to obtain an \emph{explicit formula} for all eigenvalues $\lambda_k(n)$ when $n=4m+2$ for some positive integer $m$.

From the above, when $n=4m+2$, one can verify that
\begin{equation}\label{eq:ncycleeigenvalue1}
    \left(\frac{4m+2}{2\pi}\right)^2\lambda_k(4m+2)=\frac{1}{2}(2m+1)^2(-1)^k-S_k(m),
\end{equation}
where $S_k(m) := \sum_{j=0}^{2m+1}j^2\cos\left(\frac{jk\pi}{2m+1}\right)$. Moreover, by  ``summation by parts" \cite[Theorem 3.41]{rudin1976principles} and the Lagrange's trigonometric identities \cite[pg.130]{jeffrey2008handbook},
$$S_k(m) = \sum_{j=0}^{2m}(2j+1)\sum_{l=j+1}^{2m+1}\cos\left(\frac{lk\pi}{2m+1}\right)=\frac{1}{2}(2m+1)^2(-1)^k-\frac{1}{2\sin\left(\frac{k\pi}{4m+2}\right)}\sum_{j=0}^{2m}(2j+1)\sin\left(\frac{(2j+1)k\pi}{4m+2}\right).$$

Hence, $\left(\frac{4m+2}{2\pi}\right)^2\lambda_k(n)=\frac{1}{2\sin\left(\frac{k\pi}{4m+2}\right)}\sum_{j=0}^{2m}(2j+1)\sin\left(\frac{(2j+1)k\pi}{4m+2}\right)$. By  elementary but tedious computations, we finally obtain the explicit formula:
\begin{equation}\label{eq:ncycleeigenvalue2}
    \lambda_k(4m+2)=\left(\frac{2\pi}{4m+2}\right)^2\cdot\frac{(-1)^{k+1}(4m+2)}{4\sin^2\left(\frac{k\pi}{4m+2}\right)}.
\end{equation}

By equation (\ref{eq:ncycleeigenvalue2}),  $\lambda_k(4m+2)<0$ if and only if $k=2,4,\dots,4m$. Hence, from equation (\ref{eq:ncycleeigenvalue1}):
\begin{align*} 
    \trng(K_{\mathcal{C}_{4m+2}})
    &=\left(\frac{2\pi}{4m+2}\right)^2\cdot\left(\sum_{j=0}^{2m+1}j^2\sum_{l=1}^{2m}\cos\left(\frac{2lj\pi}{2m+1}\right)-m(2m+1)^2\right)\\
    &=\left(\frac{\pi}{2m+1}\right)^2\cdot\frac{m(m+1)(4m+2)}{3}
\end{align*}

where the last equality holds because of  Lagrange's trigonometric identities. Therefore, whereas $\trng(K_{\mathcal{C}_{4m+2}})\rightarrow \infty$ as $m\rightarrow\infty$, $\frac{\trng(K_{\mathcal{C}_{4m+2}})}{4m+2}$ (and therefore the distortion $\dis(\mathcal{C}_{4m+2})$) not only remain bounded but in fact converge to $\frac{\pi^2}{12}$ (and $\frac{\pi}{\sqrt{6}}$, respectively). It will turn out that $\frac{\pi^2}{12}$  also coincides with the negative trace of the limiting space $\Sp^1$ which we will study in \S\ref{sec:gencMDS} (cf. Definition \ref{def:ngtr} and Corollary \ref{cor:propsofS1}).
\end{example}

\subsection{The stability of cMDS on finite metric spaces}\label{sec:sec:finstab}
Stability of cMDS is an obviously desirable property which guarantees that a small perturbation in the input data does not much affect  the output of the cMDS procedure.

There are some major existing contributions on this topic. 
In \cite{sibson1979studies}, Sibson studies the robustness of  cMDS when applied to  finite sets in Euclidean space in the following manner. Let $\mX$ be a finite set in Euclidean space where its true dimensionality is known and the positive eigenvalues of its Gram matrix $K_\mX$ are all simple. He first analyzes how the spectrum of $K_\mX$ is affected if $K_\mX$ is perturbed to $K_\mX+\varepsilon\cdot C+O(\varepsilon^2)$ where $C$ is symmetric and $\varepsilon>0$ is small enough (so that the smallest genuine eigenvalues are much larger than the largest spurious ones). Observe that the perturbed matrix $K_\mX+\varepsilon\cdot C+O(\varepsilon^2)$ is in general not the Gram matrix associated to a finite set in Euclidean space (i.e. some of its eigenvalues might be negative). Through this analysis, Sibson computes the first order error of the cMDS embedding (where the target dimension of the embedding is equal to the true dimension of $\mX$).

In \cite{de2004sparse}, de Silva and Tenenbaum slightly modify the argument in \cite{sibson1979studies} in order to show the robustness of \emph{landmark MDS}, a variant of cMDS which they introduce. Though their setting is quite similar to Sibson's, there are some improvements. Firstly, the positive eigenvalues of $K_\mX$, the Gram matrix associated to the input data, are not required to be simple. Secondly, it is not necessary to know the true dimension of the input dataset in order to apply their result. Instead, the target dimension of the embedding is required to satisfy certain conditions involving the eigenvalues of $K_\mX$.

More recently, in \cite{arias2020perturbation}, Arias-Castro et al. establish other perturbation bounds for cMDS. For example, in \cite[Corollary 2]{arias2020perturbation} the authors  consider a finite set $\mX$ in Euclidean space and its associated squared distance matrix $A_\mX$.\footnote{We have used slightly different notation.} Next, they assume another squared distance matrix $A_\mY$ is close enough to $A_\mX$ (with respect to a certain matrix norm) where $\mY$ is a finite metric space with the same cardinality as $\mX$. Note that in \cite[Corollary 2]{arias2020perturbation} $\mY$ is not required to be Euclidean. Finally, via an analysis of the stability of Procrustes-like methods, the authors show that a certain distance between $\mX$ and the result of the cMDS embedding on $\mY$ (where the target dimension is equal to the dimension of $\mX$) is stable with respect to the Frobenius norm of the difference between $A_\mX$ and $A_\mY$.

Proposition \ref{prop:finstability} below extends the existing stability results for cMDS of finite metric spaces in a certain way that we will explain below. Theorem   \ref{thm:stability} will eventually generalize Proposition \ref{prop:finstability} beyond the case of finite metric spaces and  will  be used to obtain convergence/consistency results for cMDS. 
 
Unlike the  results discussed above, the proof of Proposition \ref{prop:finstability}  heavily relies on both the convexity of the set of positive semi-definite matrices and the optimality of cMDS (cf. Theorem \ref{thm:optimalthm}). Some advantages of our result are the following: 
\begin{itemize}
    \item[(1)] We do not require the  input metric spaces to be Euclidean, 
    \item[(2)] The difference (measured by the Frobenius norm) between the  Gram matrices associated to the two input metric spaces is not required to be small for our results to be applicable (actually, the difference can be arbitrarily large), and 
    \item[(3)] Our results hold in the continuous setting as well as in the finite setting (cf. Theorem \ref{thm:stability}).
\end{itemize}
One disadvantage of our approach is that the target dimension $k$ of the embedding $\Phi_{\mathcal{X},k}:X\rightarrow\R^k$ into Euclidean space must be maximal (i.e. equal to the cardinality of $X$) in order to invoke the convexity of the set of positive semi-definite matrices. Note that, for $k < n$, the set of $n$ by $n$ positive semi-definite matrices with rank upper bounded by $k$, is not convex in general.

The setting of continuous metric (measure) spaces is considered in \S \ref{sec:stability} which deals with the formulation of  Theorem \ref{thm:stability} (which is applicable in the setting of mm-spaces with the Gromov-Wasserstein distance). Though the argument employed to prove Theorem \ref{thm:stability} is fairly complicated, its core idea is exemplified in the proof of the following proposition. More precisely, the nonexpansiveness of the nearest point projection to a convex set, which is invoked in the proof of Proposition \ref{prop:finstability}, will also be used in the proof of Lemma \ref{lemma:pullbackkernel}.

\begin{figure}
    \centering
    \includegraphics[width=0.5\linewidth]{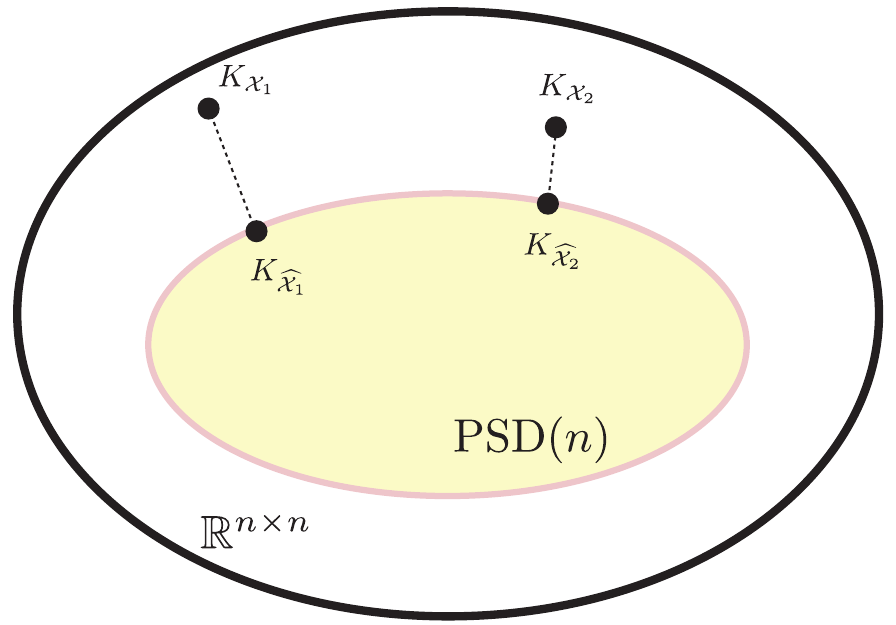}
    \caption{Illustration of Proposition \ref{prop:finstability}.}
    \label{fig:proj-same-X}
\end{figure}

\begin{proposition}\label{prop:finstability}
Let $n\geq 1$ be any positive integer and let $X$ be any finite set with cardinality $n$. Let $d_1$ and $d_2$ be any two metrics on $X$. Consider the finite metric spaces  $\mX_1=(X,d_1),\mX_2=(X,d_2)\in\mathcal{M}_n$. Then, we have the following inequality:
$$\Vert K_{\widehat{\mX_1}}-K_{\widehat{\mX_2}}\Vert_{\mathrm{F}}\leq\Vert K_{\mX_1}-K_{\mX_2}\Vert_{\mathrm{F}}.$$
\end{proposition}
\begin{proof} 
Let $\mathrm{PSD}(n)$ denote the set of positive semi-definite $n$ by $n$ matrices. Also, recall that $\R^{n\times n}$ equipped with the Frobenius norm is a Hilbert space and $\mathrm{PSD}(n)$ is a convex subset of $\R^{n\times n}$. Also, by Theorem \ref{thm:optimalthm}, $K_{\widehat{\mX_1}}$ is the (unique) nearest point of $K_{\mX_1}$ to $\mathrm{PSD}(n)$ (cf. \S \ref{sec:sec:sec:altcmds}). Similarly, $K_{\widehat{\mX_2}}$ is the (unique) nearest point of $K_{\mX_2}$ to $\mathrm{PSD}(n)$. We  obtain the claim via the nonexpansive property of the nearest point projection (cf. Theorem \ref{thm:projnonexpansive}). See Figure \ref{fig:proj-same-X}.
\end{proof}


\section{Generalized cMDS on metric measure spaces}\label{sec:gencMDS}
The set of square-summable sequences $\{(x_1,x_2,\dots):x_i\in\R\text{ for  }i\in\mathbb{Z}_{>0}\text{ and }\sum_{i=1}^\infty x_i^2<\infty\}$, together with the Euclidean inner product structure, will be denoted by $\ell^2$. Also, $m_k$ will denote the Lebesgue measure on $\R^k$.\medskip

From now we will model (possibly infinite) datasets as \textbf{mm-spaces}, that is triples $\mX=(X,d_X,\mu_X)$ where $(X,d_X)$ is a metric space and $\mu_X$ is a Borel probability measure on $X$ with full support.\footnote{The support of $\mu_X$ is defined as the smallest closed set $C$ satisfying $\mu_X(X\backslash C)=0$ and is denoted by $\supp[\mu_X]$. Note that if $X$ is discrete,
$\supp[\mu_X]=\{x\in X\colon\mu_X(x)>0\}.$} We will denote the collection  of all compact mm-spaces by $\mathcal{M}_w$. As we already mentioned in the introduction, as a slight abuse of terminology the term mm-space will always refer to a compact mm-space unless specified otherwise.

Given any mm-space $\mX$ and $p\in[1,\infty)$, the \emph{$p$-diameter} of $\mX$ is 
\begin{equation}\label{eq:diam-p}\diamna_p(\mX):=\left(\iint_{X\times X}d_X^p(x,x')\,\mu_X(dx)\,\mu_X(dx')\right)^{1/p}.\end{equation}

The most basic example of an mm-space is $\star$, the mm-space consisting of a single point. Another family of examples comes from compact Riemannian manifolds endowed with their geodesic distances and normalized volume measures.

\begin{example}
For any $d\geq 2$, we consider the standard unit $(d-1)$-dimensional sphere
$$\Sp^{d-1}:=\left\{(x_1,\dots,x_d)\in\R^d:\sum_{i=1}x_i^2=1\right\}$$
with the canonical Riemannian metric. Let $d_{\Sp^{d-1}}$ be the geodesic distance of $\Sp^{d-1}$ and let $\vol_{\Sp^{d-1}}$ be the volume measure of $\Sp^{d-1}$ induced by the canonical Riemannian metric. Since the volume $\vol_{\Sp^{d-1}}(\Sp^{d-1})$ of $\Sp^{d-1}$ will be used multiple times, we introduce the abbreviated notation $\vert\Sp^{d-1}\vert:=\vol_{\Sp^{d-1}}(\Sp^{d-1})$.\footnote{The following equality is well-known (\cite[(1.19)]{atkinson2012spherical}):
    $\vert\Sp^{d-1}\vert=\frac{2\pi^{\frac{d}{2}}}{\Gamma(\frac{d}{2})}.$} Consider the probability measure $\nvol_{\Sp^{d-1}}:=\frac{1}{\vert\Sp^{d-1}\vert}\vol_{\Sp^{d-1}}$ induced   in this way. Then, $\left(\Sp^{d-1},d_{\Sp^{d-1}},\nvol_{\Sp^{d-1}}\right)$ is a mm-space.
\end{example}

\subsection{cMDS for metric measure spaces}\label{sec:sec:cmdsmm}
The treatment of cMDS of finite metric spaces reviewed in \S\ref{sec:fms} can be recovered from the point of view in this section by noticing that any finite metric space $(X,d_X)$ can be regarded as a mm-space $\mX=(X,d_X,\mu_X)$ where $\mu_X$ is the \emph{uniform probability measure} on $X$, that is $\mu_X := \frac{1}{\mathrm{card}(X)}\sum_{x\in X} \delta_x.$\medskip

This section will employ a number of standard results from Functional Analysis and Spectral Theory of operators which can be consulted in Appendix \S\ref{sec:sec:functanal}.\medskip

Assume a mm-space $\mX\in\mathcal{M}_w$ is given. Partially motivated by \cite{zha2003isometric} and cMDS for finite metric spaces,  and following \cite{kassab2019multidimensional,adams2020multidimensional}, we now define a generalized cMDS procedure for the mm-space $\mX=(X,d_X,\mu_X)$ in the following way.

\medskip First, define the integral kernel $K_\mX:X\times X\rightarrow\mathbb{R}$, for each $x,x'\in X$,  by  
{\footnotesize{$$
K_\mX(x,x')
:= -\frac{1}{2} \left( d_X^2(x,x')-\int_X d_X^2(x,v)\,d\mu_X(v) -\int_X d_X^2(u,x')\,d\mu_X(u) +\iint_{X\times X} d_X^2(u,v)\, d\mu_X(u)\,d\mu_X(v) \right).$$}}

\begin{remark}
The last term in the above definition is $\big(\diamna_2(\mX)\big)^2$, cf. equation (\ref{eq:diam-p}).
\end{remark}

Now we are ready to define the generalized cMDS operator.

\begin{definition}[generalized cMDS operator]
For a given $\mX\in\mw$, we define the \emph{generalized cMDS operator} $\cmdsop_\mX:L^2(\mu_X)\rightarrow L^2(\mu_X)$ in the following way:
$$\phi\mapsto \cmdsop_\mX(\phi) := \int_X K_\mX(\cdot,x')\,\phi(x')\,d\mu_X(x').$$
\end{definition}

\begin{remark}[finite mm-spaces]\label{rmk:finmmcompare}
Notice that when $\mX=(X,d_X,\mu_X)$ is a finite mm-space with cardinality $n$ and $\mu_X$ is the uniform measure on $X$, then the integral kernel $K_\mX$ coincides with the expression of the matrix $K_\mX$ described in item (2) of Lemma \ref{lemma:simpletech}. Furthermore, observe that $\lambda$ is an eigenvalue of the matrix $K_\mX$ if and only if $\frac{\lambda}{n}$ is an eigenvalue of the operator $\cmdsop_\mX$.
\end{remark}

We now establish some basic properties of $K_\mX$ and $\cmdsop_\mX$.

\begin{lemma}\label{lemma:simpletechgen}
Let $\mathcal{X}=(X,d_X,\mu_X)\in\mathcal{M}_w$ be given and let $K_\mX$ and $\cmdsop_\mX$ be as above. Then,
\begin{enumerate}
    \item $K_\mX(x,x')=K_\mX(x',x)$ for any $x,x'\in X$. In particular, $\cmdsop_\mX$ is a self-adjoint operator.
    
    \item $\cmdsop_\mX$ is a Hilbert-Schmidt operator\footnote{See Definition \ref{def:HSop} for the general definition of Hilbert-Schmidt operators. For the case when the operator $\cmdsop$ is induced by an integral kernel $K$, it is known that $\cmdsop$ is Hilbert-schmidt if and only the $L^2$-norm of $K$ is finite (cf. Theorem \ref{thm:Hilbertschmidtintgkernell}). Also, every Hilbert-Schmidt operator is compact.}. In particular, it is a compact operator.
    
    \item $d_X^2(x,x')=K_\mX(x,x)+K_\mX(x',x')-2K_\mX(x,x')$ for any $x,x'\in X$.
    
    \item The constant function $\phi\equiv 1$ is an eigenfunction of $\cmdsop_\mX$ with zero eigenvalue. In particular, $\mathrm{dim}(\ker (\cmdsop_\mX))\geq 1.$
    
    \item If $\phi\in L^2(\mu_X)$ is an eigenfunction for a nonzero eigenvalue of $\cmdsop_\mX$, then $\int_X\phi(x)\,d\mu_X(x)=0.$
    
    \item If $\phi\in L^2(\mu_X)$ is an eigenfunction for a nonzero eigenvalue of $\cmdsop_\mX$, then there is a uniformly continuous function $\phi'$ such that $\phi=\phi'$ $\mu_X$-almost everywhere. 
\end{enumerate}
\end{lemma}
\begin{proof} 
    
    \smallskip
    \noindent (1) Obvious from the definition since $d_X$ is metric.
    
   \smallskip
    \noindent (2)  Note that, since $X$ is compact, $\mathrm{diam}(X)<\infty$ and $\sup_{x,x'\in X}\big|K_\mX(x,x')\big|\leq 2\,\big(\diamna(X)\big)^2$. Hence, $\int_X\int_X (K_\mX(x,x'))^2 \,d\mu_X(x')\,d\mu_X(x)\leq 4\,\big(\mathrm{diam}(X)\big)^4<\infty$. This implies that $\cmdsop_\mX$ is a Hilbert-Schmidt operator. In particular, this is a compact operator.
    
   \smallskip
    \noindent (3) Easy to show by using the definition and direct computation.
    
    \smallskip
    \noindent (4) Let $\phi\equiv 1$. Then,
    $\cmdsop_\mX(\phi)(x)=\int_X K_\mX(x,x')\,\phi(x')\,d\mu_X(x')=\int_X K_\mX(x,x')\,d\mu_X(x')=0$ for any $x\in X$.
    
   \smallskip
    \noindent (5) Obvious by item (3) of Theorem \ref{thm"selfadjtspectrmprpty} and item (4) of this lemma.

   \smallskip
    \noindent (6) Let $\phi$ be an eigenfunction of $\cmdsop_\mX$ with nonzero eigenvalue $\lambda$. This implies $\phi(x)=\frac{1}{\lambda}\cmdsop_\mX\phi(x)$ for $\mu_X$-a.e. $x\in X$. Hence, it is enough to show that $\phi':=\frac{1}{\lambda}\cmdsop_\mX\phi$ is a uniformly continuous function on $X$. Since $K_\mX$ is continuous on a compact space $X\times X$, it is uniformly continuous. Therefore, for arbitrary $\varepsilon>0$, there exists $\delta>0$ such that $\vert K_\mX(x,x'')-K_\mX(x',x'') \vert<\varepsilon$ for any $x,x',x''\in X$ whenever $d_X(x,x')<\delta$. Moreover, observe that $\int_X\vert\phi\vert\,d\mu_X=C_\phi>0$ since $\phi\neq 0$ $\mu_X$-almost everywhere. Also, $C_\phi$ cannot be infinity because $L^2(\mu_X)\subset L^1(\mu_X)$ (see \cite[Proposition 6.12]{folland2013real}). Now, observe that
     \begin{align*}
         \left\vert \frac{1}{\lambda}\cmdsop_\mX\phi(x)-\frac{1}{\lambda}\cmdsop_\mX\phi(x') \right\vert&=\frac{1}{\lambda}\left\vert \int_X (K_\mX(x,x'')-K_\mX(x',x''))\phi(x'')\,d\mu_X(x'')\right\vert\\
         &\leq\frac{1}{\lambda} \int_X \vert K_\mX(x,x'')-K_\mX(x',x'') \vert\,\vert\phi(x'')\vert\,d\mu_X(x'')\\
         &\leq\frac{\varepsilon}{\lambda}\int_X \vert\phi(x'')\vert\,d\mu_X(x'')=\frac{C_\phi\,\varepsilon}{\lambda}
     \end{align*}
     for  $x,x'\in X$ with $d_X(x,x')<\delta$. Since $\varepsilon>0$ is arbitrary, this concludes the proof.
\end{proof}

Even though the definition of the kernel $K_\mX$ appears to be rather complicated, finding the spectrum of $\cmdsop_\mX$ can be reduced to a simpler eigenproblem involving the kernel $-\frac{1}{2}d_X^2$ as made explicit in the following result.

\begin{lemma}\label{lemma:simplereigenproblem}
Suppose $\mX\in\mw$ is given. If $\lambda\in\R$ and $0\neq\phi\in L^2(\mu_X)$ satisfy the following two conditions:
\begin{enumerate}
    \item $-\frac{1}{2}\int_X d_X^2(x,x')\phi(x')\,d\mu_X(x')=\lambda\phi(x)$ for $\mu_X$-a.e $x\in X$, and
    \item $\int_X\phi(x)\,d\mu_X(x)=0$,
\end{enumerate}
then $\lambda$ is an eigenvalue of $\cmdsop_\mX$ and $\phi$ is a associated eigenfunction.
\end{lemma}
\begin{proof}
Assume $\phi$ and $\lambda$ satisfy the two conditions above. Then, for any $x\in X$, we have
\begin{align*}
&\int_X K_\mX(x,x')\phi(x')\,d\mu_X(x')\\
&= -\frac{1}{2} \bigg(\int_X d_X^2(x,x')\phi(x')\,d\mu_X(x')-\int_X d_X^2(x,v)\,d\mu_X(v)\cdot\int_X \phi(x')\,d\mu_X(x')\\
&\quad-\iint_{X\times X}d_X^2(u,x')\phi(x')\,d\mu_X(u)\,d\mu_X(x') +\iint_{X\times X} d_X^2(u,v)\, d\mu_X(u)\,d\mu_X(v)\cdot\int_X \phi(x')\,d\mu_X(x') \bigg)\\
&=\lambda\phi(x)-\lambda\int_X \phi(u)\,d\mu_X(u)=\lambda\phi(x).
\end{align*}

This completes the proof.
\end{proof}

Moreover, the reverse implication of Lemma \ref{lemma:simplereigenproblem} holds for nonzero eigenvalues in the setting of \textbf{two-point homogeneous spaces} (cf. Corollary \ref{cor:twohomoeigenproblem}).

\begin{definition}[two-point homogeneous space]
A mm-space $\mX=(X,d_X,\mu_X)\in\mw$ is said to be \textbf{two-point homogeneous} if it satisfies the following condition:

\begin{quote}For arbitrary $x,x'\in X$ there is an isometry $f:X\rightarrow X$ such that $f_{*}\mu_X=\mu_X$ and $f(x)=x'$.
\end{quote}
\end{definition}

\begin{proposition}\label{prop:homogeneousspace}
Suppose $\mX=(X,d_X,\mu_X)\in\mathcal{M}_w$ is two-point homogeneous. Then, 
$$K_\mX(x,x')=-\frac{1}{2}\Big(d_X^2(x,x')-\big(\diamna_2(\mX)\big)^2\Big)
\,\,\mbox{for every $x,x'\in X$.}$$
\end{proposition}
\begin{proof}
Fix arbitrary $x,x'\in X$. Since $\mX$ is two-point homogeneous, there exist an isometry $f$ such that $f(x)=x'$ and $f_*\mu_X=\mu_X$. Then, it is easy to check that $\int_X d_X^2(x,u)\,d\mu_X(u)=\int_X d_X^2(x',u)\,d\mu_X(u)$. This implies that the map $x\mapsto\int_X d_X^2(x,u)\,d\mu_X(u)$ is constant, and therefore the constant is equal to its average $\big(\mathrm{diam}_2(\mX)\big)^2$. Hence, the claim follows.
\end{proof}

\begin{corollary}\label{cor:twohomoeigenproblem}
Suppose $\mX=(X,d_X,\mu_X)\in\mathcal{M}_w$ is two-point homogeneous. If $\lambda$ is a nonzero eigenvalue of $\cmdsop_\mX$ with an associated eigenfunction $\phi\in L^2(\mu_X)$, then the following two equalities hold:

\begin{enumerate}
    \item $-\frac{1}{2}\int_X d_X^2(x,x')\phi(x')\,d\mu_X(x')=\lambda\phi(x)$ for $\mu_X$-a.e $x\in X$, and
    
   \item $\int_X\phi(x)\,d\mu_X(x)=0.$
\end{enumerate}
\end{corollary}
\begin{proof}
Apply Proposition \ref{prop:homogeneousspace} and item (5) of Lemma \ref{lemma:simpletechgen}.
\end{proof}

The above corollary will be useful when  studying the spectrum of $\cmdsop_{\Sp^{d-1}}$ (see \S\ref{sec:sec:S1} and \S\ref{sec:Spd-1}).

The following example describes the generalized cMDS kernel and operator for one of the simplest cases: the case when $\mX$ is a compact subset of a Euclidean space.

\begin{example}\label{ex:rankEucoper}
For a compact subset $X$ of $\R^k$ with $m_k(X)>0$, we consider $\mX=(X,\Vert\cdot\Vert,\overline{m}_k)\in\mathcal{M}_w$
where $\Vert\cdot\Vert$ indicates the metric on $X$ induced by the Euclidean norm, and $\overline{m}_k:=\frac{1}{m_k(X)}m_k$ is the normalized $k$-dimensional Lebesgue measure on $X$. Then, one can easily show that
$$K_\mX(x,x')=\left\langle x-\mathrm{cm}(\mX),x'-\mathrm{cm}(\mX)\right\rangle$$
for any $x,x'\in X$ where $\mathrm{cm}(\mX):=\int_X x\,d\overline{m}_k(x)\in\R^k$ is the center of mass of $\mX$. This means that $\cmdsop_\mX$ is positive semi-definite (cf. Definition \ref{def:psdop}), since
$$\langle\cmdsop_\mX\phi,\phi \rangle=\left\Vert\int_X (x-\mathrm{cm}(\mX))\phi(x)\,d\overline{m}_k(x)\right\Vert^2\geq 0$$
for any $\phi\in L^2(\overline{m}_k)$. Moreover, observe that $\cmdsop_\mX=\cmdsop_2\circ \cmdsop_1$ where
$$ \cmdsop_1:L^2(\overline{m}_k)\longrightarrow\R^k\text{  s.t.  } \phi\longmapsto\int_X (x-\mathrm{cm}(\mX))\,\phi(x)\,d\overline{m}_k(x), \text{ and}
$$ 
$$
    \cmdsop_2:\R^k\longrightarrow L^2(\overline{m}_k)\text{  s.t.  }
    x\longmapsto\langle \cdot-\mathrm{cm}(\mX),x \rangle.
$$

In particular, this implies that the rank of $\cmdsop_\mX$ is at most $k$ since the domain of $\cmdsop_2$ is $\R^k$. 
\end{example}

\subsection{The cMDS embedding for metric measure spaces}\label{sec:sec:cmdsembdmm}
Since by items (1) and (2) of Lemma \ref{lemma:simpletechgen} the generalized cMDS operator $\cmdsop_\mX$ is compact and self-adjoint, standard results from functional analysis guarantee that $\cmdsop_\mX$ has discrete spectrum consisting of real eigenvalues (cf. Theorem \ref{thm"selfadjtspectrmprpty}, Theorem \ref{thm:RieszSchauder} and Theorem \ref{thm:HilbertSchmidt}).

\begin{definition}
Let $\mathfrak{A}$ be a compact and self-adjoint operator on a Hilbert space $\mathcal{H}$. Then,
\begin{align*}
    &\pr(\mathfrak{A}):=\text{cardinality of the set of positive eigenvalues of }\mathfrak{A}\text{ counted with multiplicity,}\\
    &\nr(\mathfrak{A}):=\text{cardinality of the set of negative eigenvalues of }\mathfrak{A}\text{ counted with multiplicity.}
\end{align*}
\end{definition}

The cMDS procedure for finite metric spaces (cf. Algorithm \ref{alg:cMDSfin}) will be now  generalized to general mm-spaces in the following way: For a given $\mX\in\mw$, choose $k\leq \pr(\cmdsop_\mX)$. We build a generalized cMDS embedding $$\Phi_{\mX\!\!,k}:X\rightarrow \R^k$$ together with its induced mm-space
$$\widehat{\mX}_k:=(\Phi_{\mX\!\!,k}(X),\Vert\cdot\Vert,(\Phi_{\mX\!\!,k})_\ast\mu_X)\in\mathcal{M}_{w}$$
obtained by eliminating all negative eigenvalues and by keeping the top $k$ largest (positive) eigenvalues   of $\cmdsop_\mX$ and their corresponding orthonormal eigenfunctions.

\begin{algorithm}
\caption{cMDS embeding into $\R^k$ for mm-spaces}\label{alg:gencMDSRk}
\begin{algorithmic}[1]
    \State \textbf{Input}: $\mX=(X,d_X,\mu_X)\in\mathcal{M}_w$ and a positive integer $k\leq \pr(\cmdsop_\mX)$.
    
    \State Construct $K_\mX$ and $\cmdsop_\mX$ as before.
    
    \State Compute the positive eigenvalues $\lambda_1\geq\dots\geq\lambda_{\mathrm{pr}(\cmdsop_\mX)}>0$ ($\mathrm{pr}(\cmdsop_\mX)$ can be infinite) and corresponding orthonormal eigenfunctions $\{\phi_i\}_{i=1}^{\mathrm{pr}(\cmdsop_\mX)}$ of $\cmdsop_\mX$.
    
    \State $\Phi_{\mX\!\!,k}(x):=\big(\lambda_1^{\frac{1}{2}}\phi_1(x),\lambda_2^{\frac{1}{2}}\phi_2(x),\dots,\lambda_k^{\frac{1}{2}}\phi_k(x)\big)\in\R^k$ for each $x\in X$.
    
    \State \textbf{Output}: $\Phi_{\mX\!\!,k}:X\rightarrow \R^k$ and $\widehat{\mX}_k$.
\end{algorithmic}
\end{algorithm}

\begin{remark}
By item (6) of Lemma \ref{lemma:simpletechgen}, one can assume that $\Phi_{\mX\!\!,k}$ is continuous.
\end{remark}

\begin{definition}[generalized cMDS into $\R^k$]\label{def:cmdsrk}
Given a mm-space $\mX=(X,d_X,\mu_X)\in\mw$ and orthonormal eigenfunctions $a=\{\phi_i\}_{i=1}^{\pr(\cmdsop_\mX)}$ of $\cmdsop_\mX$ for the corresponding positive eigenvalues $\lambda_1\geq\lambda_2\geq\cdots>0$, one defines the following map:
\begin{align*}
    \Phi_{\mX\!\!,k}^a:X&\longrightarrow \R^k\\
    x&\longmapsto\big(\lambda_1^{\frac{1}{2}}\phi_1(x),\lambda_2^{\frac{1}{2}}\phi_2(x),\dots,\lambda_k^{\frac{1}{2}}\phi_k(x)\big).
\end{align*}
This map $\Phi_{\mX\!\!,k}^a$ is called \textbf{generalized cMDS embedding into $\R^k$} of $\mX$, and the resulting mm-space is denoted by
$$\widehat{\mX}_k:=(\Phi_{\mX\!\!,k}^a(X),\Vert\cdot\Vert,(\Phi_{\mX\!\!,k}^a)_\ast\mu_X)\in\mathcal{M}_{w}.$$
In practice, the superscript $a$ will be omitted in many cases for the simplicity, unless we have to specify the set of orthonormal eigenfunctions.
\end{definition}

\subsubsection{The cMDS embedding into $\ell^2$.}
One critical difference between the cMDS method for finite metric spaces (see \S\ref{sec:fms}) and the  generalized cMDS method is the possibility that there can exist infinitely many positive eigenvalues of $\cmdsop_\mX$. $\Sp^{d-1}$ ($d\geq 2$) and $\mathbb{T}^N$ ($N\geq 2$) are actual examples with such property, and they are studied in \S \ref{sec:sec:S1}, \S \ref{sec:sec:TN}, and \S \ref{sec:Spd-1}. To encompass such cases where $\pr(\cmdsop_\mX)=\infty$, we consider a variant of the cMDS method where the target space is infinite dimensional, namely $\ell^2$. As a result of applying the generalized cMDS method to a given mm-space $\mX=(X,d_X,\mu_X)$ into $\ell^2$, we formally expect to be able to build the following embedding: 
$$\Phi_\mX^a:X\longrightarrow\ell^2\text{  s.t.  }x\longmapsto \left(\lambda_1^{\frac{1}{2}}\phi_1(x),\lambda_2^{\frac{1}{2}}\phi_2(x),\dots\right)$$
for a fixed family of orthonormal eigenfunctions $a=\{\phi_i\}_{i=1}^{\pr(\cmdsop_\mX)}$.\medskip

However,  to guarantee that (as defined above) $\Phi_\mX^a(x)$ is a point in $\ell^2$, the square summability of the sequence $\left(\lambda_1^{\frac{1}{2}}\phi_1(x),\lambda_2^{\frac{1}{2}}\phi_2(x),\dots\right)$ is required. This can be ensured $\mu_X$-a.e (see Proposition \ref{prop:traceclasscmds}) by assuming that the cMDS operator $\cmdsop_\mX$ is \textbf{trace class} (cf. Definition \ref{def:trace-norm}).

\subsubsection{Traceable metric measure spaces.}\label{sec:sec:sec:trblmmspace}

\begin{framed}
\noindent
\textbf{Traceability of $\cmdsop_\mX$.} Here we provide an abridged explanation of what it means for the generalized cMDS operator to be trace class (or traceable). 

Recall that, for a given $\mX\in\mw$, since $\cmdsop_\mX$ is a compact and self-adjoint operator (cf. Lemma \ref{lemma:simpletechgen}),  $\cmdsop_\mX$ has a discrete set of real eigenvalues.

Under these conditions, we define the \emph{trace norm} of $\cmdsop_\mX$ (denoted by $\Vert \cmdsop_\mX \Vert_1$) to be the sum of absolute values of all eigenvalues of $\cmdsop_\mX$ (counted with their multiplicity). Now, we say $\cmdsop_\mX$ is trace class (equivalently, $\cmdsop_\mX$ is \emph{traceable}) if $\Vert \cmdsop_\mX \Vert_1<\infty$ (cf. Theorem \ref{thm:tracesingvaluecndtn}). Moreover, if $\cmdsop_\mX$ is trace class, then the trace of $\cmdsop_\mX$ (denoted by $\tr(\cmdsop_\mX)$) is well-defined and it coincides with the sum of all eigenvalues of $\cmdsop_\mX$ (counted with their multiplicity).

Readers may check Definition \ref{def:trace-norm} in Appendix \S\ref{sec:sec:functanal} to see the detailed general definition of \textbf{trace class} operators.
\end{framed}

We will define the above version of the cMDS  embedding into $\ell^2$ under the assumption that  $\cmdsop_\mX$ is trace class (see Algorithm \ref{algorithm:gencmdsl2} and Definition \ref{def:cmdsl2}).

\begin{definition}
The collection of all mm-spaces $\mX=(X,d_X,\mu_X)\in\mathcal{M}_w$ such that $\cmdsop_\mX:L^2(\mu_X)\rightarrow L^2(\mu_X)$ is trace class is said to be the collection of all \textbf{traceable mm-spaces} and is denoted by $\mathcal{M}_w^{tr}$.
\end{definition}

Naturally, establishing canonical examples of mm-spaces which belong to $\mwtr$ will be an important topic of the later part of this paper. However, there exist (non-compact) mm-spaces which \emph{do not} induce traceable cMDS operators --- a construction of one such pathological mm-space will be described in \S\ref{sec:non-trace-class}.

\medskip

Before we formally introduce a generalized cMDS formulation for mm-spaces in $\mwtr$ which produces embeddings into $\ell^2$, it is  worth  mentioning that there is a \emph{coordinate-free} interpretation of generalized cMDS which we will describe first. Even though in that scenario we lose precise coordinatizations , this alternative formulation is far more general in the sense that it is well defined even in the case when the cMDS operator is not trace class. Moreover, the stability of cMDS with respect to the Gromov-Wasserstein distance holds for this coordinate-free interpretation (Theorem \ref{thm:stability}) without having to assume the traceability of the cMDS operator.

\subsubsection{Coordinate-free interpretation of generalized cMDS.}\label{sec:sec:sec:altcmds}
Fix an arbitrary $\mX\in\mw$. Let $L^2_{\mathrm{sym}}(\mu_X\otimes\mu_X)$ denote the subset of symmetric kernels in $L^2(\mu_X\otimes\mu_X)$ (cf. Definition \ref{def:symker}). Observe that $K_\mX\in L^2_{\mathrm{sym}}(\mu_X\otimes\mu_X)$. Also, let $L^2_{\succeq 0,\mathrm{sym}}(\mu_X\otimes\mu_X)$ denote the set of symmetric kernels whose associated integral operator is positive semi-definite. Note that $L^2_{\succeq 0,\mathrm{sym}}(\mu_X\otimes\mu_X)$ is a convex subset of $L^2_{\mathrm{sym}}(\mu_X\otimes\mu_X)$.

Finally, let $P_{\mu_X}$ be the nearest point projection (with respect to the $L^2(\mu_X\otimes\mu_X)$-norm) from $L^2_{\mathrm{sym}}(\mu_X\otimes\mu_X)$ to the closure of $L^2_{\succeq 0,\mathrm{sym}}(\mu_X\otimes\mu_X)$. Let
$$K_\mX^+:=P_{\mu_X}(K_\mX)$$
be the nearest point projection of the cMDS operator $\cmdsop_\mX$ (cf. Theorem \ref{thm:optimalthmgen}). Then, the  map
$K_\mX\longmapsto K_\mX^+$
can be regarded as an alternative, coordinate free, interpretation of cMDS. In \S\ref{sec:sec:sec:coordinate-dependent} we will define for each  $\mX\in\mwtr$ the embedding $\Phi_\mX:X\rightarrow\ell^2$ together with its induced mm-space $\widehat{\mX}$. Moreover, we will prove that $K_{\widehat{\mX}}=K_\mX^+$ (cf. Proposition \ref{prop:twoviewsame}) when $\mX\in\mwtr$.


\subsubsection{Coordinate-dependent interpretation of generalized cMDS into $\ell^2$.}\label{sec:sec:sec:coordinate-dependent}

Now we introduce the coordinate-dependent version of cMDS mapping into $\ell^2$ that we promised. 

The following proposition shows that the traceability of $\cmdsop_\mX$ guarantees that the $\mu_X$-almost everywhere square summability of the sequence $\left(\lambda_1^{\frac{1}{2}}\phi_1(x),\lambda_2^{\frac{1}{2}}\phi_2(x),\dots\right)$.

\begin{proposition}\label{prop:traceclasscmds}
Let $\mX=(X,d_X,\mu_X)\in\mwtr$. Then, $\sum_{i=1}^{\mathrm{pr}(\cmdsop_\mX)} \lambda_i(\phi_i(x))^2<\infty$
for $\mu_X$-almost every $x\in X$ where $\lambda_1\geq\dots\geq\lambda_{\mathrm{pr}(\cmdsop_\mX)}>0$ are all positive eigenvalues and $\{\phi_i\}_{i=1}^{\mathrm{pr}(\cmdsop_\mX)}$ are corresponding orthonormal eigenfunctions of $\cmdsop_\mX$.
\end{proposition}
\begin{proof}
If $\mathrm{pr}(\cmdsop_\mX)<\infty$, the claim is obvious. Consider the $\mathrm{pr}(\cmdsop_\mX)=\infty$ case. We define the map $f:X\longrightarrow [0,\infty]$ by $x\longmapsto\sum_{i=1}^\infty \lambda_i(\phi_i(x))^2$. Note that $f$ is measurable since it is the pointwise limit of a sequence of continuous functions (cf. item (6) of Lemma \ref{lemma:simpletechgen}).

Therefore, by the monotone convergence theorem,
\begin{align*}
    \int_X f\,d\mu_X(x)&=\sum_{i=1}^\infty \lambda_i\int_X (\phi_i(x))^2\,d\mu_X(x)=\sum_{i=1}^\infty \lambda_i\leq\Vert\cmdsop_\mX\Vert_1<\infty.
\end{align*}
Hence, $\sum_{i=1}^\infty \lambda_i(\phi_i(x))^2<\infty$ for $\mu_X$-almost every $x\in X$.
\end{proof}

Hence, for any $\mX\in\mwtr$, we consider the following type of cMDS embedding where the target space $\ell^2$ is infinite dimensional: 

\begin{algorithm}
\caption{cMDS embeding into $\ell^2$ for traceable mm-spaces}\label{algorithm:gencmdsl2}
\begin{algorithmic}[1]
    \State \textbf{Input}: $\mX=(X,d_X,\mu_X)\in\mwtr$.
    
    \State Construct $K_\mX$ and $\cmdsop_\mX$ as before.
    
    \State Compute the positive eigenvalues $\lambda_1\geq\lambda_2\geq\dots>0$ and corresponding orthonormal eigenfunctions $\{\phi_i\}_{i=1}^{\pr(\cmdsop_\mX)}$ of $\cmdsop_\mX$.
    
    \State Define $\Phi_\mX(x):=\begin{cases}\big(\lambda_1^{\frac{1}{2}}\phi_1(x),\lambda_2^{\frac{1}{2}}\phi_2(x),\dots,\lambda_k^{\frac{1}{2}}\phi_{\pr(\cmdsop_\mX)}(x),0,0,\dots\big)&\text{if }\pr(\cmdsop_\mX)<\infty\\ \big(\lambda_1^{\frac{1}{2}}\phi_1(x),\lambda_2^{\frac{1}{2}}\phi_2(x),\dots\big)&\text{if }\pr(\cmdsop_\mX)=\infty \end{cases}$

    for each $x\in X$. Note that $\Phi_\mX(x)\in\ell^2$ $\mu_X$-a.e. by Proposition \ref{prop:traceclasscmds}.
    
    \State \textbf{Output}: $\Phi_\mX:X\rightarrow \ell^2$ and $\widehat{\mX}$.
\end{algorithmic}
\end{algorithm}

\begin{definition}[generalized cMDS into $\ell^2$]\label{def:cmdsl2}
Given a mm-space $\mX=(X,d_X,\mu_X)\in\mwtr$ and orthonormal eigenfunctions $a=\{\phi_i\}_{i=1}^{\pr(\cmdsop_\mX)}$ of $\cmdsop_\mX$ for the corresponding positive eigenvalues $\lambda_1\geq\lambda_2\geq\cdots>0$, one defines the following map: 
\begin{align*}
    \Phi_\mX^a:X&\longrightarrow \ell^2\\
    x&\longmapsto\begin{cases}\big(\lambda_1^{\frac{1}{2}}\phi_1(x),\lambda_2^{\frac{1}{2}}\phi_2(x),\dots,\lambda_k^{\frac{1}{2}}\phi_{\pr(\cmdsop_\mX)}(x),0,0,\dots\big)&\text{if }\pr(\cmdsop_\mX)<\infty\\ \big(\lambda_1^{\frac{1}{2}}\phi_1(x),\lambda_2^{\frac{1}{2}}\phi_2(x),\dots\big)&\text{if }\pr(\cmdsop_\mX)=\infty \end{cases}
\end{align*}
This map $\Phi_\mX^a$ is called \textbf{generalized cMDS embedding into $\ell^2$} of $\mX$, and the resulting mm-space is denoted by 
$$\widehat{\mX}^a:=(\Phi_\mX^a(X),\Vert\cdot\Vert,(\Phi_\mX^a)_\ast\mu_X).$$

Unless it is necessary to specify the set $a$ of orthonormal eigenfunctions, the superscript $a$ will be omitted, for simplicity, and in such cases we will write $\Phi_{\mX}$ and $\widehat{\mX}$.\footnote{Note that when $a$ and $a'$ are two different sets of orthonormal eigenfunctions, the mm-spaces $\widehat{\mX}^a$ and $\widehat{\mX}^{a'}$ are isomorphic.} See Algorithm \ref{algorithm:gencmdsl2}.
\end{definition}

The following proposition shows that the coordinate-free interpretation agrees with the coordinate-dependent interpretation  when the traceability of the cMDS operator is assumed.

\begin{proposition}\label{prop:twoviewsame}
Let $\mX\in\mwtr$. Then, $K_\mX^+=K_{\widehat{\mX}}$.
\end{proposition}
\begin{proof}
Theorem \ref{thm:optimalthmgen} below establishes that $K_{\widehat{\mX}}$ is the nearest point of $K_\mX$ to $L^2_{\succeq 0,\mathrm{sym}}(\mu_X\otimes\mu_X)$. Since there is the unique nearest point to a closed convex subset, we can conclude that  $K_\mX^+=K_{\widehat{\mX}}$ in this case.
\end{proof}

\begin{remark}
Suppose $\mX=(X,d_X,\mu_X)\in\mw$ and positive integer $k\leq\pr(\cmdsop_\mX)$ are given. Then,
$$\Vert \Phi_{\mX\!\!,k}(x)-\Phi_{\mX\!\!,k}(x')\Vert^2=\sum_{i=1}^k\lambda_i(\phi_i(x)-\phi_i(x'))^2$$
for any $x,x'\in X$ where $\lambda_1\geq\lambda_2\geq\dots>0$ are positive eigenvalues of $\cmdsop_\mX$, and $\{\phi_i\}_{i=1}^{\pr(\cmdsop_\mX)}$ are their corresponding orthonormal eigenfunctions. Then, by  definition, it follows that 
$$K_{\widehat{\mX}_k}(x,x')=\sum_{i=1}^k\lambda_i\,\phi_i(x)\,\phi_i(x')$$
for any $x,x'\in X$. Similarly, if $\mX\in\mwtr$, we have both
$$\Vert \Phi_{\mX}(x)-\Phi_{\mX}(x')\Vert^2=\sum_{i=1}^{\pr(\cmdsop_\mX)}\lambda_i(\phi_i(x)-\phi_i(x'))^2 \mbox{   and   } K_{\widehat{\mX}}(x,x')=\sum_{i=1}^{\pr(\cmdsop_\mX)}\lambda_i\,\phi_i(x)\,\phi_i(x').$$
\end{remark}

\begin{remark}\label{rmk:tracediam2}
If $\mX=(X,d_X,\mu_X)\in\mwtr$, by Theorem \ref{thm:trintegdiag} and the definition of $K_\mX$, we have
$$\tr(\cmdsop_\mX)=\int_X K_\mX(x,x)\,d\mu_X(x)=\frac{1}{2}\iint_{X\times X} d_X^2(u,v)\,d\mu_X(u)\,d\mu_X(v)=\frac{1}{2}\big(\diamna_2(\mX)\big)^2.$$

This implies that, unless $\diamna_2(\mX)=0$ (i.e. $\mX$ is the one point mm-space), there must be at least one positive eigenvalue of $\cmdsop_\mX$.
\end{remark}

Now, let us collect some simple examples of spaces which belong to $\mwtr$. More advanced examples of mm-space in $\mwtr$ can be found in \S \ref{sec:sec:S1}, \S \ref{sec:sec:TN}, and \S \ref{sec:Spd-1}.

\begin{remark}
If $\mX=(X,d_X,\mu_X)\in\mathcal{M}_w$ is a finite mm-space, then $\mX$ belongs to $\mathcal{M}_w^{tr}$ since there are only finitely many eigenvalues.
\end{remark}

\begin{remark}
For a compact subset $X$ of $\R^k$ with $m_k(X)>0$, we consider $$\mX=(X,\Vert\cdot\Vert,\overline{m}_k)\in\mathcal{M}_w$$
as in Example \ref{ex:rankEucoper}. Then, there are only finitely many eigenvalues of $\cmdsop_\mX$ since $\mathrm{rank}(\cmdsop_\mX)\leq k$. Hence, $\mX\in\mwtr$.
\end{remark}

\begin{example}
Consider  the interval $\mX=(X,\Vert\cdot\Vert,\overline{m}_1)\in\mathcal{M}_w$ where $X:=[-1,1]$ and $\overline{m}_1:=\frac{1}{2}m_1$. Then, by Example \ref{ex:rankEucoper}, $K_\mX(x,x')=xx'$ for any $x,x'\in X$ since the center of mass of $\mX$ is the origin. Then, observe that
\begin{align*}
    \cmdsop_\mX\mathrm{Id}_{[-1,1]}(x)=\int_{[-1,1]} x(x')^2\,d\overline{m}_1(x')=\frac{1}{2}\int_{[-1,1]} x(x')^2\,dm_1(x')
    =\frac{1}{3}x=\frac{1}{3}\mathrm{Id}_{[-1,1]}(x)
\end{align*}

for any $x\in X$. Therefore, the identity map is an eigenfunction of $\cmdsop_\mX$ for the eigenvalue $\frac{1}{3}$, which is the unique nonzero eigenvalue of $\cmdsop_\mX$.
\end{example}

\subsection{A non-compact metric measure space with non-traceable cMDS operator} \label{sec:non-trace-class}

We adapt a construction due to Lyons \cite{lyons2018errata,lyons2013distance} which arose in the context of the notion of distance covariance between metric spaces.

\emph{The construction.} For each prime $q$ that is congruent to $1$ modulo $4$, let $G_q$ be the Paley graph consisting of $q$ vertices (see Example \ref{ex:PaleyL2natural} to review the definition and properties of Paley graphs).

Now, let $\mathcal{G}_q:=(V(G_q),d_{G_q},\mu_{G_q})$ be the mm-space such that $V(G_q)$ is the set of vertices of $G_q$, $d_{G_q}$ is the shortest path metric, and $\mu_{G_q}$ is the uniform probability measure.

Taking into account the spectral properties of the adjacency matrix of $G_q$ explained in Example \ref{ex:PaleyL2natural}, Remark \ref{rmk:finmmcompare}, and Lemma \ref{lemma:simplereigenproblem}, one can easily check that (1) the eigenvalues of $\cmdsop_{\mathcal{G}_q}$ are $0$ (with multiplicity $1$) and $\lambda_q^\pm:=\frac{5\pm 3\sqrt{q}}{4q}$ (each with multiplicity $\frac{q-1}{2}$) and (2) eigenfunctions of $\cmdsop_{\mathcal{G}_q}$ are either constant functions or orthogonal to constant functions: Constant eigenfunctions correspond to the eigenvalue $0$, and the eigenfunctions orthogonal to constant functions correspond to the eigenvalue pair $\lambda_q^\pm$. At any rate, note that the trace norm (cf. Definition \ref{def:trace-norm}) $\Vert\cmdsop_{\mathcal{G}_{q}}\Vert_1$ is approximately $\frac{3\sqrt{q}}{4}$ as $q\rightarrow\infty$.

Now, choose a sequence $(q_n)_n$ of primes $q_n\equiv 1 (\mathrm{mod}\,4)$ so that $\sum_{n} \frac{1}{\sqrt{q_n}}=:\frac{1}{c}<\infty$. We  let $\mX:=(X,d_X,\mu_X)$ be the mm-space defined as follows:
\begin{itemize}
    \item $X:=\sqcup_{n}V(G_{q_n})$,
    
    \item $d_X|_{V(G_{q_n})\times V(G_{q_n})} = d_{G_{q_n}}$ for each $n$, whereas, for $n\neq m$, $d_X|_{V(G_{q_n})\times V(G_{q_m})}=1$. 
    
    \item  $\mu_X$ is the probability measure on $X$ which assigns mass $c\,q_n^{-3/2}$ to each point of $V(G_{q_n})$.
\end{itemize}
    We will prove that $\cmdsop_\mX$ is not trace class.\medskip

A calculation shows that for any eigenvector $v_n\perp\mathbf{1}$ of $\cmdsop_{\mathcal{G}_{q_n}}$ (necessarily) with eigenvalue $\lambda =\lambda_{q_n}^\pm$, the function $f_n:X\rightarrow \R$ defined by
\begin{align*}
    f_n(x):=\begin{cases}v_n(x)&\text{if }x\in G_{q_n}\\ 0&\text{otherwise}\end{cases}
\end{align*}
is an eigenfunction of $\cmdsop_\mX$ for the eigenvalue $\frac{\lambda\cdot c}{\sqrt{q_n}}$ by Lemma \ref{lemma:simplereigenproblem}. Therefore, for each $n$, $\lambda_{q_n}^\pm\cdot\frac{c}{\sqrt{q_n}}$ is an eigenvalue of $\cmdsop_\mX$ with multiplicity at least $\frac{q_n-1}{2}$. This implies that the trace norm of $\cmdsop_\mX$ satisfies
$$\Vert\cmdsop_\mX\Vert_1\geq \sum_n \Vert\cmdsop_{\mathcal{G}_{q_n}}\Vert_1\cdot\frac{c}{\sqrt{q_n}}.$$ Hence, $\Vert\cmdsop_\mX\Vert_1$ cannot be finite since $$\lim_{n\rightarrow\infty}\Vert\cmdsop_{\mathcal{G}_{q_n}}\Vert_1\cdot\frac{c}{\sqrt{q_n}}=\lim_{n\rightarrow\infty}\frac{3\sqrt{q_n}}{4}\cdot\frac{c}{\sqrt{q_n}}=\frac{3c}{4}> 0.$$

\medskip
Therefore, $\cmdsop_\mX$ is \emph{not} trace class.\medskip

We point out that the space $\mX$ constructed above is not compact. This motivates:

\begin{restatable}{question}{cmtr}
\label{q:cmmtr}
Is it true that every compact mm-space induces a cMDS operator which is trace class?
\end{restatable}
and also the more circumspect
\begin{restatable}{question}{cmtrman}
\label{q:cmmtrman}
Is it true that every compact Riemannian manifold induces a cMDS operator which is trace class?
\end{restatable}

\begin{framed}
\noindent\textbf{Note}: After this paper was accepted, we learned that this question was recently answered to the negative by Ma and Stepanov; see  \cite{ma2024eigenvalues} where the authors prove that for $n\geq 1$ the spaces $\mathbb{RP}^{2n+1}$ induce non-traceable cMDS operators.  These recent developments motivate
the following loosely defined challenge:\begin{restatable}{challenge}{cmtrmannew}
\label{c:cmmtrmannew}
Find as-general-as-possible geometric conditions guaranteeing that a compact Riemannian manifold induces a trace class cMDS operator.
\end{restatable}
\end{framed}

\subsection{Distortion incurred by the cMDS operator}\label{sec:sec:cmdsdistmm}

We define the following notions in order to measure the metric distortion incurred by the generalized cMDS method.

\begin{definition}[negative trace]\label{def:ngtr}
Let $\mathfrak{A}$ be a compact and self-adjoint operator on a Hilbert space $\mathcal{H}$. Then, the \emph{negative trace} of $\mathfrak{A}$ is defined in the following way:

$$\trng(\mathfrak{A}):=\sum_{i=1}^{\mathrm{nr}(\mathfrak{A})} \vert \zeta_i \vert$$
where $\zeta_1\leq\dots\leq\zeta_{\nr(\cmdsop_\mX)}<0$ are the negative eigenvalues of $\mathfrak{A}$.
\end{definition}

\begin{definition}
Let $\mX=(X,d_X,\mu_X)\in\mwtr$. For each $k\leq \pr(\cmdsop_\mX)$, we define:
$$\dis_k(\mathcal{X}):=\left(\int_X\int_X\big\vert\Vert \Phi_{\mX\!\!,k}(x)-\Phi_{\mX\!\!,k}(x')\Vert^2-d_X^2(x,x')\big\vert\,d\mu_X(x')\,d\mu_X(x)\right)^{\frac{1}{2}},$$
and
$$\dis(\mathcal{X}):=\left(\int_X\int_X\big\vert\Vert \Phi_\mX(x)-\Phi_\mX(x')\Vert^2-d_X^2(x,x')\big\vert\,d\mu_X(x')\,d\mu_X(x)\right)^{\frac{1}{2}}.$$
\end{definition}

\begin{proposition}[generalized cMDS bounds]\label{prop:generrorbdd}
Let $\mX=(X,d_X,\mu_X)\in{\mathcal M}_w^{tr}$. Then the following bounds hold: 
\begin{enumerate}
	\item (\textbf{cMDS is non-contracting}) For $\mu_X\otimes\mu_X$-almost every $(x,x')\in X\times X$, 
	
	$$d_X(x,x')\leq \Vert \Phi_\mX(x)-\Phi_\mX(x') \Vert$$
	and for any integer $k\leq\mathrm{pr}(\cmdsop_\mX)$ and $\mu_X\otimes\mu_X$-almost every $(x,x')\in X\times X$,
	$$\Vert \Phi_{\mX\!\!,k}(x)-\Phi_{\mX\!\!,k}(x') \Vert\leq\Vert \Phi_\mX(x)-\Phi_\mX(x') \Vert.$$ 
	
	\item (\textbf{$L^2$-distortion})
	$$\dis(\mathcal{X})=\sqrt{2\,\trng(\cmdsop_\mX)}.$$
\end{enumerate}
\end{proposition}

\begin{remark}
Note that item (1) of Proposition \ref{prop:generrorbdd} implies that
$$d_X(x,x')\leq \Vert \Phi_\mX(x)-\Phi_\mX(x') \Vert\geq\Vert \Phi_{\mX\!\!,k}(x)-\Phi_{\mX\!\!,k}(x') \Vert$$
for any integer $k\leq\mathrm{pr}(\cmdsop_\mX)$ and $\mu_X\otimes\mu_X$-almost every $(x,x')\in X\times X$. Also, in general it is not true that $d_X(x,x')\leq \Vert \Phi_{\mX\!\!,k}(x)-\Phi_{\mX\!\!,k}(x') \Vert$ for $\mu_X\otimes\mu_X$-almost every $(x,x')\in X\times X$. Indeed, one such example is given by the case when $\mX=\Sp^1$ and $k=1$ (see \S \ref{sec:sec:S1}). By Proposition \ref{prop:nzreigenvalueS1}, we can choose a cMDS embedding of $\Sp^1$ into $\R^1$ 
$$\Phi_{\Sp^1,1}:\Sp^1\longrightarrow\R\text{  s.t.  }\theta\longmapsto \sqrt{2}\cos\theta.$$
Then, $d_{\Sp^1}(0,\pi)=\pi>2\sqrt{2}=\vert\Phi_{\Sp^1,1}(0)-\Phi_{\Sp^1,1}(\pi)\vert$.
\end{remark}

\begin{proof}[Proof of Proposition \ref{prop:generrorbdd}]
\noindent
(1) Let $\lambda_1\geq\dots\geq\lambda_{\mathrm{pr}(\cmdsop_\mX)}>0$ be all the positive eigenvalues of, and $\{\phi_i\}_i$ be the corresponding orthonormal eigenfunctions of $\cmdsop_\mX$. Similarly, let $\zeta_1\leq\dots\leq\zeta_{\mathrm{nr}(\cmdsop_\mX)}<0$ be all the negative eigenvalues of, and let $\{\psi_i\}_i$ be the corresponding orthonormal eigenfunctions of $\cmdsop_\mX$.

Since $K_\mX$ is trace class, the series $L(x,x'):=\sum_{i=1}^{\pr(\cmdsop_\mX)}\lambda_i\,\phi_i(x)\,\phi_i(x')+\sum_{i=1}^{\nr(\cmdsop_\mX)}\zeta_i\,\psi_i(x)\,\psi_i(x')$ absolutely converges for $\mu_X\otimes\mu_X$-almost every $(x,x')\in X\times X$. Also,  $L(x,x)=\sum_{i=1}^{\pr(\cmdsop_\mX)}\lambda_i\,\phi_i^2(x)+\sum_{i=1}^{\nr(\cmdsop_\mX)}\zeta_i\,\psi_i^2(x)$ absolutely converges for $\mu_X$-almost every $x\in X$. Furthermore, by Corollary \ref{cor:HSuniquekernel}, $K_\mX(x,x')=L(x,x')$ for $\mu_X\otimes\mu_X$-almost every $(x,x')\in X\times X$.

Next, we will show that $K_\mX$ is $\mu_X$-almost everywhere equal to $L$ on the diagonal: $$K_\mX(x,x)=\sum_{i=1}^{\pr(\cmdsop_\mX)}\lambda_i\,\phi_i^2(x)+\sum_{i=1}^{\nr(\cmdsop_\mX)}\zeta_i\,\psi_i^2(x)=L(x,x)\,\,\,\,\mbox{for $\mu_X$-almost every $x\in X$.}$$ We will do this by resorting to an idea used by Brislawn in \cite{brislawn1991traceable}.

Let $\mathcal{U}=\{U_n\}_{n\in\N}$ be a countable base for the Borel $\sigma$-algebra of $X$. We will use $\mathcal{U}$ to define a nested sequence of partitions $\{\mathcal{P}_n\}_{n\in\N}$ of $X$. For $\mathcal{P}_0$ we take any countable partition of $X$ and we define $\mathcal{P}_n$ inductively as follows. For $n\geq 1$, let  $$\mathcal{P}_n:=\bigcup_{B\in \mathcal{P}_{n-1}} \{U_n \cap B,U_n^c\cap B\}.$$  This is an increasing sequence of (Borel measurable)  partitions in the sense that  $\mathcal{P}_n$ is finer than $\mathcal{P}_{n-1}$ for every $n\geq 0$. Let $C_n(x)$ be the unique block in the partition $\mathcal{P}_n$ containing the point $x\in X$. As in \cite[pg.232]{brislawn1991traceable}, one can verify that
$\mu_X(C_n(x))>0$ for $\mu_X$-almost every $x\in X$, for every $n\geq 0$. Then, by the  argument used in the proof of \cite[Theorem 3.1]{brislawn1991traceable} (which is based on Doob's martingale convergence theorem), one has
$$\lim_{n\rightarrow\infty}\frac{1}{\big(\mu_X(C_n(x))\big)^2}\iint_{C_n(x)\times C_n(x)}K_\mX(s,t)\,d\mu_X(s)\,d\mu_X(t)=L(x,x)$$
for $\mu_X$-almost every $x\in X$. Since, by continuity of $K_\mX$, $$\lim_{n\rightarrow\infty}\frac{1}{\big(\mu_X(C_n(x))\big)^2}\iint_{C_n(x)\times C_n(x)}K_\mX(s,t)\,d\mu_X(s)\,d\mu_X(t)=K_\mX(x,x)$$ for $\mu_X$-almost every $x\in X$, it follows that
$$K_\mX(x,x)=L(x,x) = \sum_{i=1}^{\pr(\cmdsop_\mX)}\lambda_i\,\phi_i^2(x)+\sum_{i=1}^{\nr(\cmdsop_\mX)}\zeta_i\,\psi_i^2(x)\text{ for }\mu_X\text{-almost every }x\in X.$$

Therefore, for $\mu_X\otimes\mu_X$-almost every $(x,x')\in X\times X$,
\begin{align*}
    &\Vert\Phi_\mX(x)-\Phi_\mX(x')\Vert^2-d_X^2(x,x')\\
    &=(K_{\widehat{\mX}}(x,x)-K_\mX(x,x))+(K_{\widehat{\mX}}(x',x')-K_\mX(x',x'))-2(K_{\widehat{\mX}}(x,x')-K_\mX(x,x'))\\
    &=\sum_{i=1}^{\mathrm{nr}(\cmdsop_\mX)}(-\zeta_i)\,\big(\psi_i(x)-\psi_i(x')\big)^2\geq 0.
\end{align*}
Thus,  $\Vert\Phi_\mX(x)-\Phi_\mX(x')\Vert\geq d_X(x,x')$ for $\mu_X\otimes\mu_X$-almost every $(x,x')$. Also,
$$ \Vert\Phi_\mX(x)-\Phi_\mX(x')\Vert^2-\Vert\Phi_{\mX\!\!,k}(x)-\Phi_{\mX\!\!,k}(x')\Vert^2=\sum_{i=k+1}^{\pr(\cmdsop_\mX)}\lambda_i\,((\phi_i(x)-\phi_i(x'))^2\geq 0 $$
   for $\mu_X\otimes\mu_X$-almost every $(x,x')$. Hence, we have $\Vert\Phi_\mX(x)-\Phi_\mX(x')\Vert\geq\Vert\Phi_{\mX\!\!,k}(x)-\Phi_{\mX\!\!,k}(x')\Vert$ for $\mu_X\otimes\mu_X$-almost every $(x,x')$.
   
\medskip
   \noindent
   (2) Since we already know $\Vert\Phi_\mX(x)-\Phi_\mX(x')\Vert\geq d_X(x,x')$ and $\Vert\Phi_\mX(x)-\Phi_\mX(x')\Vert^2-d_X^2(x,x')=\sum_{i=1}^{\mathrm{nr}(\cmdsop_\mX)}(-\zeta_i)\,\big(\psi_i(x)-\psi_i(x')\big)^2$ for $\mu_X\otimes\mu_X$-almost every $(x,x')$, we have:
   \begin{align*}
       \big(\dis(\mX)\big)^2&=\int_X\int_X\big(\Vert \Phi_\mX(x)-\Phi_\mX(x')\Vert^2-d_X^2(x,x')\big)\,d\mu_X(x')\,d\mu_X(x)\\
       &=\int_X\int_X\sum_{i=1}^{\nr(\cmdsop_\mX)}(-\zeta_i)\,((\psi_i(x)-\psi_i(x'))^2\,d\mu_X(x')\,d\mu_X(x)\\
       &=\sum_{i=1}^{\nr(\cmdsop_\mX)}(-\zeta_i)\int_X\int_X(\psi_i^2(x)+\psi_i^2(x')-2\psi_i(x)\psi_i(x'))\,d\mu_X(x')\,d\mu_X(x)\\
       &=2\,\sum_{i=1}^{\nr(\cmdsop_\mX)}(-\zeta_i)\,(\because\{\psi_i\}_i\text{ are orthonormal and item (5) of Lemma \ref{lemma:simpletechgen}})\\
       &=2\,\trng(\cmdsop_\mX).
   \end{align*}
   This completes the proof.
   \end{proof}

\subsection{Optimality of the generalized cMDS}\label{sec:sec:gencMDSoptimal}

We establish the following optimality result for generalized cMDS in a manner is analogous to Theorem \ref{thm:optimalthm} (which is only applicable to finite metric spaces).

\begin{theorem}\label{thm:optimalthmgen}
Suppose $\mX=(X,d_X,\mu_X)\in\mwtr$. Then,
\begin{enumerate}
    \item $\Vert K_\mX-K \Vert_{L^2(\mu_X\otimes\mu_X)}\geq\Vert K_\mX-K_{\widehat{\mX}_k} \Vert_{L^2(\mu_X\otimes\mu_X)}$ for any positive integer $k\leq\pr(\cmdsop_\mX)$, and $K\in L^2_{\succeq 0,{\mathrm{sym}}}(\mu_X\otimes\mu_X)$ with $\mathrm{rank}(\cmdsop)\leq k$ where $\cmdsop$ is the positive semi-definite operator induced by $K$.
    
    \item $\Vert K_\mX-K \Vert_{L^2(\mu_X\otimes\mu_X)}\geq\Vert K_\mX-K_{\widehat{\mX}} \Vert_{L^2(\mu_X\otimes\mu_X)}$ for any $K\in L^2_{\succeq 0,{\mathrm{sym}}}(\mu_X\otimes\mu_X)$.
\end{enumerate}
\end{theorem}

We provide the proof of Theorem \ref{thm:optimalthmgen} in Appendix \S\ref{sec:otherproofs}. This theorem is stated in \cite{kassab2019multidimensional}. We provide a proof of this result to make our paper self-contained and to fill in some gaps in the argument given therein\footnote{In particular, the arguments in \cite{kassab2019multidimensional} do not consider the possibility that the cMDS operator may fail to be traceable. }; See Appendix  \S\ref{sec:otherproofs} for details.

\subsection{Thickness and spectrum of $\cmdsop_\mX$ for Euclidean data}\label{sec:thickness}

In this section, we prove that if $\mX$ is a compact subset of Euclidean space, then the smallest positive eigenvalue of $\cmdsop_\mX$ is upper bounded by a certain notion of \emph{thickness} of $\mX$  (Proposition \ref{thm:thicknessbdd}). This result exemplifies the phenomenon  that the spectrum of $\cmdsop_\mX$ is controlled by the underlying geometry of $\mX$.

We begin with an example.

\begin{example}
For $a\geq b>0$ consider the following ellipse in $\R^2$: $\mX:=(X,\Vert\cdot\Vert,\overline{m}_2)$ where $\overline{m}_2=\frac{1}{ab\pi}m_2$ is the normalized Lebesgue measure and
$$X:=\left\{x=(x_1,x_2)\in\mathbb{R}^2:\left(\frac{x_1}{a}\right)^2+\left(\frac{x_2}{b}\right)^2\leq1\right\}.$$
By Example \ref{ex:rankEucoper}, $K_\mX(x,x')=\langle x,x' \rangle$ for any $x,x'\in X$, $\cmdsop_\mX$ is positive semi-definite, and the rank of $\cmdsop_\mX$ is at most $2$. Let $\lambda_1\geq\lambda_2\geq 0$ be the two largest eigenvalues of $\cmdsop_\mX$. We  show that $\pi_1:=\langle \cdot,e_1 \rangle$ and $\pi_2:=\langle \cdot,e_2 \rangle$ are their corresponding eigenfunctions.

First, consider $\pi_1$ and notice that:
 $\int_X\big(\pi_1(x)\big)^2\,d\overline{m}_2(x)=\frac{a^2}{4}.$ Hence, $\Vert \pi_1 \Vert_{L^2(\overline{m}_2)}=\frac{a}{2}$.  Also,
 \begin{align*}
    \cmdsop_\mX\pi_1(x)=\int_X\langle x,x' \rangle\,\pi_1(x')\,d\overline{m}_2(x')
    =\frac{1}{ab\pi}\int_X x_1(x_1')^2\,dm_2(x')+\frac{1}{ab\pi}\int_X x_2x_1'x_2'\,dm_2(x')
    =\frac{a^2x_1}{4}=\frac{a^2}{4}\pi_1(x)
\end{align*}

for any $x\in X$. In a similar way,  show that the $L^2$-norm of $\pi_2=\langle \cdot,e_2 \rangle$ is $\frac{b}{2}$ and that $ \cmdsop_\mX\pi_2=\frac{b^2}{4}\pi_2$. Therefore $\frac{a^2}{4}\geq \frac{b^2}{4}>0$ are the only positive eigenvalues, and $\frac{2\pi_1}{a},\frac{2\pi_2}{b}$ are the corresponding orthonormal eigenfunctions.\medskip

With the aid of the above calculations, it is easy to check $\Phi_{\mX\!\!,1}(x)=\frac{a}{2}\cdot\frac{2}{a}\pi_1(x)=x_1$ for any $x=(x_1,x_2)\in X$. Hence, if we only retain the 1st coordinate produced by cMDS, we incur the distortion:
\begin{align*}
   \big( \dis_1(\mX)\big)^2&= \int_X\int_X\big\vert\Vert\Phi_{\mX\!\!,1}(x)-\Phi_{\mX\!\!,1}(x') \Vert^2-\Vert x-x'\Vert^2\big\vert\,d\overline{m}_2(x')\,d\overline{m}_2(x)\\
    &=\int_X\int_X\big((x_1-x_1')^2+(x_2-x_2')^2-(x_1-x_1')^2\big)\,d\overline{m}_2(x')\,d\overline{m}_2(x)\\
    &=2\int_X x_2^2\,d\overline{m}_2(x)-2\left(\int_X x_2\,d\overline{m}_2(x)\right)^2=\frac{b^2}{2}.
\end{align*}
Hence, $\dis_1(\mX)=\frac{b}{\sqrt{2}}$ which is consistent with the fact that, for thin ellipses ($a\gg b$), the cMDS method with 1 coordinate causes small distortion.
\end{example}

We now introduce a notion of thickness which is applicable in general.

\begin{definition}
For a given mm-space $\mX=(X,\Vert\cdot\Vert,\overline{m}_k)\in\mathcal{M}_w$ where $X$ is a compact subset of $\R^k$ with $m_k(X)>0$ and $\overline{m}_k:=\frac{1}{m_k(X)}m_k$, the \emph{thickness} of $\mX$  is the number
$$\mathrm{Th}(\mX):=\inf_{u\in\Sp^{k-1}}\left(\int_X \big(\pi_u(x-\mathrm{cm}(\mX))\big)^2\,d\overline{m}_k(x)\right)^{\frac{1}{2}},$$
where $\pi_u(x):=\langle x,u \rangle$.
\end{definition}
In words, given a unit vector $u \in \R^k$, one first projects (via push-forward) the measure $\overline{m}_k$ under the map $f_u:\mathbb{R}^k\rightarrow \mathbb{R}$ defined by $x\mapsto \pi_u(x-\mathrm{cm}(X))$. Then,  $\mathrm{Th}(\mX)$ is calculated as the infimum of the standard deviation of $(f_u)_\#\overline{m}_k$ over all such $u$.

\begin{example}
For the ellipse considered above, its thickness is $\frac{b}{2}.$ 
\end{example}

The discussion above can be generalized as follows. 

\begin{proposition}\label{thm:thicknessbdd}
Suppose $\mX=(X,\Vert\cdot\Vert,\overline{m}_k)\in\mathcal{M}_w$ is given, where $X$ is a compact subset of $\R^k$ with $m_k(X)>0$ and $\overline{m}_k:=\frac{1}{m_k(X)}m_k$. Then, the $k$-th eigenvalue $\lambda_k$ of $K_\mX$ satisfies the following inequality:
$$\lambda_k\leq \big(\mathrm{Th}(\mX)\big)^2.$$
\end{proposition}

The proof of the proposition will make use of the optimality theorem (Theorem \ref{thm:optimalthmgen}).

\begin{proof}
Fix any $u\in\Sp^{k-1}$. Without loss of generality, we may assume $\mathrm{cm}(\mX)=0$. Choose an orthonormal basis $v_1,\cdots,v_{k-1},v_k=u$ of $\R^k$. Then, consider the following map
$$
\pi_{u^\perp}:\R^k\longrightarrow\R^{k-1}\text{  s.t.  }x\longmapsto (\pi_{v_1}(x),\cdots,\pi_{v_{k-1}}(x)),
$$

and the kernel 
$$
    K_{u^\perp}:X\times X\longrightarrow\mathbb{R}\,\text{    s.t.  }
    (x,x')\longmapsto\langle \pi_{u^\perp}(x),\pi_{u^\perp}(x') \rangle.
$$
Then, by Theorem \ref{thm:optimalthmgen}, we have the following inequality:
$$\Vert K_\mX-K_{u^\perp} \Vert_{L^2(\overline{m}_k\otimes\overline{m}_k)}\geq\Vert K_\mX-K_{\widehat{\mX}_{k-1}} \Vert_{L^2(\overline{m}_k\otimes\overline{m}_k)}.$$
It is easy to show that $\Vert K_\mX-K_{\widehat{\mX}_{k-1}} \Vert_{L^2(\overline{m}_k\otimes\overline{m}_k)}=\lambda_k$. Also, by Example \ref{ex:rankEucoper}, we have $K_\mX(x,x')=\langle x,x' \rangle$ for any $x,x'\in X$. Therefore, for any $x,x'\in X$.
\begin{align*}
    K_\mX(x,x')-K_{u^\perp}(x,x')&=\sum_{i=1}^k\pi_{v_i}(x)\pi_{v_i}(x')-\sum_{i=1}^{k-1}\pi_{v_i}(x)\pi_{v_i}(x')
    =\pi_{v_k}(x)\pi_{v_k}(x')=\pi_u(x)\pi_u(x').
\end{align*}
 Hence,
\begin{align*}
    \Vert K_\mX-K_{u^\perp} \Vert_{L^2(\overline{m}_k\otimes\overline{m}_k)}^2&=\int_X\int_X \big(\pi_u(x)\pi_u(x')\big)^2\,d\overline{m}_k(x')\,d\overline{m}_k(x)=\left(\int_X \big(\pi_u(x)\big)^2\,d\overline{m}_k(x)\right)^2
\end{align*}
so that we have $\Vert K_\mX-K_{u^\perp} \Vert_{L^2(\overline{m}_k\otimes\overline{m}_k)}=\int_X \big(\pi_{u}(x)\big)^2\,d\overline{m}_k(x)$. Finally, we achieve
$$\lambda_k\leq\int_X \big(\pi_{u}(x)\big)^2\,d\overline{m}_k(x).$$
Since $u$ is arbitrary, the conclusion follows.
\end{proof}


\section{Traceability of cMDS on the circle, tori,  metric graphs and fractals }\label{sec:ntrvalexmple}

In this section, we study the traceability of the generalized cMDS procedure when applied to exemplar metric measure spaces such as spheres, tori, and metric graphs. In particular Corollary \ref{coro:prod-mm-tr} establishes that if $\mX,\mY\in\mwtr$ then their product $\mX\times \mY$ is also in $\mwtr$.

\subsection{Generalized cMDS for product spaces}\label{sec:sec:cmdsproduct}
Let $\mX=(X,d_X,\mu_X)$ and $\mY=(Y,d_Y,\mu_Y)$ be two mm-spaces in $\mathcal{M}_w$. Now, we consider the following product mm-space:
$$\mathcal{X}\times\mathcal{Y}=\Big(X\times Y,d_{X\times Y}=\sqrt{d_X^2+d_Y^2},\mu_X\otimes\mu_Y\Big)\in\mathcal{M}_w.$$

In other words, the product mm-space consists of the $\ell^2$-product metric and the product measure.

\begin{lemma}\label{lemma:prodK}
Two mm-spaces $\mX=(X,d_X,\mu_X)$ and $\mY=(Y,d_Y,\mu_Y)$ are given. Then,
$$K_{\mathcal{X}\times \mathcal{Y}}((x,y),(x',y')) = K_\mathcal{X}(x,x') + K_\mathcal{Y}(y,y')$$ for any $x,x'\in X$ and $y,y'\in Y$.
\end{lemma}
The proof follows by direct calculation and we omit it.

\begin{remark}\label{rmk:manyprod}The situation described above can be obviously generalized as follows. Let $N\geq 2$ be a positive integer and let  $\mX_i=(X_i,d_{X_i},\mu_{X_i})\in\mathcal{M}_w$ for each $i=1,\cdots,N$. Let $\mX$ be the product mm-space
$$\mX = \bigg(X_1\times \cdots\times X_N,\sqrt{d_{X_1}^2+\cdots d_{X_N}^2},\mu_{X_1}\otimes\cdots\otimes\mu_{X_N}\bigg).$$
Then, by using  induction on $N$ together with Lemma \ref{lemma:prodK}, we have:
$$K_\mX((x_1,\dots,x_n),(x_1',\dots,x_n'))=\sum_{i=1}^N K_{\mX_i}(x_i,x_i')$$\
for any $(x_1,\dots,x_n),(x_1',\dots,x_n')\in\mX$.
\end{remark}

The following proposition clarifies what the spectrum of $\cmdsop_{\mX\times\mY}$ is for given $\mX,\mY\in\mw$.

\begin{proposition}\label{prop:spectrumofprod}
Two mm-spaces $\mX=(X,d_X,\mu_X),\mY=(Y,d_Y,\mu_Y)\in\mw$ are given. Then,
\begin{enumerate}
    \item Let $\phi$ be an eigenfunction of $\cmdsop_\mX$ for some nonzero eigenvalue, and $\psi$ be an eigenfunction of $\cmdsop_\mY$ for some nonzero eigenvalue. Then, $\phi\times\psi$ becomes an eigenfunction of $\cmdsop_{\mX\times\mY}$ for the zero eigenvalue.
    
    \item Let $\phi$ be an eigenfunction of $\cmdsop_\mX$ for some nonzero eigenvalue $\lambda$, and $\psi$ be an eigenfunction of $\cmdsop_\mY$ for the zero eigenvalue. Then, $\phi\times\psi$ becomes an eigenfunction of $\cmdsop_{\mX\times\mY}$ for the eigenvalue $\lambda\cdot\int_Y\psi\,d\mu_Y$.\footnote{Note that an analogous statement, for the (symmetric) case when $\phi$ is an eigenfunction of $\cmdsop_\mX$ for the zero eigenvalue and $\psi$ is an eigenfunction of $\cmdsop_\mY$ for some nonzero eigenvalue, also holds.}
    
    \item Let $\phi$ be an eigenfunction of $\cmdsop_\mX$ for the zero eigenvalue, and $\psi$ be an eigenfunction of $\cmdsop_\mY$ for the  zero eigenvalue. Then, $\phi\times\psi$ becomes an eigenfunction of $\cmdsop_{\mX\times\mY}$ for the zero eigenvalue.
\end{enumerate}
\end{proposition}

\begin{remark}
In item (2) above, note that $\psi\equiv 1$ is always an eigenfunction of $\cmdsop_\mY$ for the zero eigenvalue (cf. item (4) of Lemma \ref{lemma:simpletechgen}) so that $\int_Y\psi\,d\mu_Y=1$.
\end{remark}

\begin{proof}[Proof of Proposition \ref{prop:spectrumofprod}]
By Lemma \ref{lemma:prodK},
    \begin{align*}
    &\int_{X\times Y}K_{\mX\times\mY}((x,y),(x',y'))\phi(x')\psi(y')\,d\mu_X(x')\,d\mu_Y(y')\\
    &=\int_Y\int_X K_\mX(x,x')\phi(x')\psi(y')\,d\mu_X(x')\,d\mu_Y(y')+\int_Y\int_X K_\mY(y,y')\phi(x')\psi(y')\,d\mu_X(x')\,d\mu_Y(y')\\
    &=\int_X K_\mX(x,x')\phi(x')\,d\mu_X(x')\cdot\int_Y\psi(y')\,d\mu_Y(y')+\int_Y K_\mY(y,y')\,\psi(y')\,d\mu_Y(y')\cdot\int_X\phi(x')\,d\mu_X(x')
\end{align*}
for any $(x,y)\in X\times Y$. Hence,

    \smallskip
    \noindent (1) If $\phi$ be an eigenfunction of $\cmdsop_\mX$ for some nonzero eigenvalue, and $\psi$ be an eigenfunction of $\cmdsop_\mY$ for some nonzero eigenvalue: Then,
    $$\int_{X\times Y}K_{\mX\times\mY}((x,y),(x',y'))\phi(x')\psi(y')\,d\mu_X(x')\,d\mu_Y(y')=0$$
    since $\int_Y\psi_j(y')\,d\mu_Y(y')=\int_X\phi_i(x')\,d\mu_X(x')=0$ by item (5) of Lemma \ref{lemma:simpletechgen}. Hence, $\phi\times\psi$ is an eigenfunction for zero eigenvalue.

 \smallskip
    \noindent (2) If $\phi$ be an eigenfunction of $\cmdsop_\mX$ for some nonzero eigenvalue $\lambda$, and $\psi$ be an eigenfunction of $\cmdsop_\mY$ for zero eigenvalue: Then,
   $$\int_{X\times Y}K_{\mX\times\mY}((x,y),(x',y'))\phi(x')\psi(y')\,d\mu_X(x')\,d\mu_Y(y')=\left(\lambda\cdot\int_Y\psi\,d\mu_Y\right)\phi(x)$$
   since $\int_X\phi_i(x')\,d\mu_X(x')=0$ by item (5) of Lemma \ref{lemma:simpletechgen}. Hence, $\phi\times\psi$ is an eigenfunction for $\lambda\cdot\int_Y\psi\,d\mu_Y$.
   
  \smallskip
    \noindent (3) If $\phi$ be an eigenfunction of $\cmdsop_\mX$ for zero eigenvalue, and $\psi$ be an eigenfunction of $\cmdsop_\mY$ for zero eigenvalue: Then,
    $$\int_{X\times Y}K_{\mX\times\mY}((x,y),(x',y'))\phi(x')\psi(y')\,d\mu_X(x')\,d\mu_Y(y')=0$$
    by the assumption on $\phi$ and $\psi$. Hence, $\phi\times\psi$ is an eigenfunction for the zero eigenvalue.
\end{proof}

\begin{corollary}\label{coro:prod-mm-tr}
If $\mX=(X,d_X,\mu_X),\mY=(Y,d_Y,\mu_Y)\in\mwtr$, then $\mX\times\mY\in\mwtr$.
\end{corollary}
\begin{proof}[Proof of Corollary \ref{coro:prod-mm-tr}]
Let $\{\phi_i\}_{i=1}^{\mathrm{dim}(\ker \cmdsop_\mX)}$ be an orthonormal basis of  $\ker\cmdsop_\mX$ with $\phi_1\equiv 1$ (one can always find such basis by choosing $\{\phi_i\}_{i=2}^{\mathrm{dim}(\ker \cmdsop_\mX)}$ as an orthonormal basis of $\ker \cmdsop_\mX\cap\{\phi_1\}^{\perp}$). Similarly, let $\{\psi_i\}_{j=1}^{\mathrm{dim}(\ker\cmdsop_\mY)}$ be an orthonormal basis of $\ker\cmdsop_\mY$ with $\psi_1\equiv 1$. Then, it is easy to verify that $\Vert\cmdsop_{\mX\times\mY}\Vert_1=\Vert\cmdsop_\mX\Vert_1+\Vert\cmdsop_\mY\Vert_1<\infty$ where the first equality holds by Proposition \ref{prop:spectrumofprod}. This completes the proof.
\end{proof}

\begin{proposition}\label{prop:spectrumofNprod}
Suppose $\mX_i=(X_i,d_{X_i},\mu_{X_i})\in\mathcal{M}_w$ for $i=1,\cdots,N$ are given. Let $\mX:=\mX_1\times\cdots\times\mX_N$ be the product mm-space. Then, if $\lambda_i$ is an eigenvalue of $\cmdsop_{\mX_i}$ with an eigenfunction $\phi_i$ for each $i=1,\dots,N$, the function  $\phi_1\times\cdots\times\phi_N:X_i\times \cdots \times X_N\rightarrow \R$ is an eigenfunction of $\cmdsop_\mX$ with eigenvalue:
$$\begin{cases} 0 &\text{ if there are at least two of nonzero }\lambda_is.\\  \lambda_i\cdot\left(\prod_{j\neq i}\int_{X_j}\phi_j\,d\mu_{X_j}\right) &\text{ if }\lambda_i\neq 0\text{ and }\lambda_j=0\text{ for any }j\neq i.\\0 &\text{ if }\lambda_i=0\text{ for any }i=1,\dots,N. \end{cases}$$
\end{proposition}
\begin{proof}
Follow the proof of Proposition \ref{prop:spectrumofprod} and use the observation in Remark \ref{rmk:manyprod}.
\end{proof}

\begin{corollary}\label{cor:propsofprod}
Suppose $\mX=(X,d_X,\mu_X)\in\mwtr$ satisfies $\mathrm{dim}(\ker \cmdsop_\mX)=1$. Let $\mX^{\otimes N}=\mX\times\cdots\times\mX$ (the product of $N$ copies of $\mX$). Then,
\begin{enumerate}
    \item $\Vert\cmdsop_{\mX^{\otimes N}}\Vert_1=N\cdot\Vert\cmdsop_\mX\Vert_1$. In particular, this implies $\mX^{\otimes N}\in\mwtr$.
    
    \item $\tr(\cmdsop_{\mX^{\otimes N}})=N\cdot\tr(\cmdsop_\mX)$.
    
    \item $\trng(\cmdsop_{\mX^{\otimes N}})=N\cdot\trng(\cmdsop_\mX)$.
    
    \item For any $k\leq\pr(\cmdsop_{\mX})$ and $(x_1,\dots,x_N),(x_1',\dots,x_N')\in\mX^{\otimes N}$, we have
        \begin{align*}
        &\Vert \Phi_{\mX^{\otimes N},kN}(x_1,\dots,x_N)-\Phi_{\mX^{\otimes N},kN}(x_1',\dots,x_N') \Vert^2=\sum_{i=1}^N\Vert \Phi_{\mX\!\!,k}(x_i)-\Phi_{\mX\!\!,k}(x_i') \Vert^2\text{, and}\\
        &\Vert \Phi_{\mX^{\otimes N}}(x_1,\dots,x_N)-\Phi_{\mX^{\otimes N}}(x_1',\dots,x_N') \Vert^2=\sum_{i=1}^N\Vert \Phi_{\mX}(x_i)-\Phi_{\mX}(x_i') \Vert^2.
    \end{align*}
\end{enumerate}
\end{corollary}
\begin{proof}
Observe that the assumption $\mathrm{dim}(\ker \cmdsop_\mX)=1$ implies that constant functions are the only eigenfunctions for the zero eigenvalue. Then, this corollary is a direct result of Proposition \ref{prop:spectrumofNprod}.
\end{proof}

\subsection{The case of $\Sp^1$}\label{sec:sec:S1}
We will compute eigenvalues and corresponding eigenfunctions of the cMDS operator on $\Sp^1$ (with geodesic metric) and apply generalized cMDS. In this subsection, we will view $\Sp^1$ as $\R/2\pi$.

\medskip
The following formulas can be obtained through direct calculation.
\begin{lemma}\label{lemma:u2cossin}
For any $n\in\Z\backslash\{0\}$ we have:
\begin{enumerate}
    \item $\int_{-\pi}^\pi u^2\cos(nu)\,du=\frac{(-1)^n 4\pi}{n^2}$, and
    
    \item $\int_{-\pi}^\pi u^2\sin(nu)\,du=0$.
\end{enumerate}
\end{lemma}

The following result has been described in \cite[Proposition 5.1]{adams2020multidimensional}. However, we provide a proof here in order to give a concise standalone treatment which will be contrasted with the much harder case of $\Sp^{d-1}$ for $d\geq 3$  (which will be dealt with in \S\ref{sec:Spd-1}).

\begin{proposition}\label{prop:nzreigenvalueS1}
The multiset of  eigenvalues of $\cmdsop_{\Sp^1}$ consists of $0$ (with multiplicity $1$) and $\left\{\frac{(-1)^{n+1}}{n^2}\right\}_{n=1}^{\infty}$ (each with multiplicity $2$). Moreover, the eigenspace corresponding to $0$ is spanned by a nonzero constant function and the eigenspace corresponding $\frac{(-1)^{n+1}}{n^2}$ is spanned by the functions $\theta\mapsto\cos(n\theta)$ and $\theta\mapsto\sin(n\theta)$, for each $n\geq 1$. Therefore, $\left\{\sqrt{2}\cos(n\theta),\sqrt{2}\sin(n\theta)\right\}_{n=1}^\infty$ together with $\phi_0\equiv 1$ form an orthonormal basis of $L^2(\nvol_{\Sp^1})$ consisting of eigenfunctions of $\cmdsop_{\Sp^1}$.
\end{proposition}

\begin{remark}
For any $n\geq 3$, recall that $\mathcal{C}_n$ is a metric space consisting of the vertices of the regular $n$-gon inscribed in $\Sp^1$ equipped with the metric inherited from $\Sp^1$. In Example \ref{ex:reg-polys}, we studied $\lambda_k(n)$, the $k$-th eigenvalue of the matrix $K_{\mathcal{C}_n}$ for $0\leq k\leq n-1$, and $\trng(K_{\mathcal{C}_n})$. If we view $\mathcal{C}_n$ as a mm-space equipped with the uniform probability measure, then it is easy to verify that $\widetilde{\lambda}_k(n):=\frac{\lambda_k(n)}{n}$ is an eigenvalule of $\cmdsop_{\mathcal{C}_n}$ for each $0\leq k\leq n-1$. Moreover, in \cite[Corollary 7.2.2]{kassab2019multidimensional},\footnote{This also follows more explicitly from equation (\ref{eq:ncycleeigenvalue2}).} it is proved that $\widetilde{\lambda}_k(n)$ converges to $\frac{(-1)^{k+1}}{k^2}$ which, as we established above, coincides with the $k$-th eigenvalue of $\cmdsop_{\Sp^1}$ as $n\rightarrow\infty$ for any $k\geq 1$.
\end{remark}

\begin{proof}[Proof of Proposition \ref{prop:nzreigenvalueS1}].
We already know that any nonzero constant function is an eigenfunction of $\cmdsop_{\Sp^1}$ associated to the eigenvalue $0$ (cf. item (4) of Lemma \ref{lemma:simpletechgen}).

\begin{claim}\label{claim:cossineigenS1}
For each $n\geq 1$, $\theta\mapsto\cos(n\theta)$ and $\theta\mapsto\sin(n\theta)$ are eigenfunctions of $\cmdsop_{\Sp^1}$ associated to the eigenvalue $\frac{(-1)^{n+1}}{n^2}$.
\end{claim}
\begin{proof}[Proof of Claim \ref{claim:cossineigenS1}]
Observe that, for any $\theta\in\Sp^1$,
\begin{align*}
    -\frac{1}{2}\int_{\Sp^1}d_{\Sp^1}^2(\theta,\theta')\cos(n\theta')\,d\theta'&=-\frac{1}{2}\cdot\frac{1}{2\pi}\int_{\theta-\pi}^{\theta+\pi}(\theta'-\theta)^2\cos(n\theta')\,d\theta'\\
    &=-\frac{1}{4\pi}\int_{-\pi}^\pi u^2\cos(nu+n\theta)\,du\quad(\text{by substituting } u=\theta'-\theta)\\
    &=-\cos(n\theta)\cdot\left(\frac{1}{4\pi}\int_{-\pi}^\pi u^2\cos(nu)\,du\right)+\sin(n\theta)\cdot\left(\frac{1}{4\pi}\int_{-\pi}^\pi u^2\sin(nu)\,du\right)\\
    &=\frac{(-1)^{n+1}}{n^2}\cos(n\theta)\quad(\because\text{ Lemma \ref{lemma:u2cossin}}).
\end{align*}
Also,
\begin{align*}
    -\frac{1}{2}\int_{\Sp^1}d_{\Sp^1}^2(\theta,\theta')\sin(n\theta')\,d\theta'&=-\frac{1}{2}\cdot\frac{1}{2\pi}\int_{\theta-\pi}^{\theta+\pi}(\theta'-\theta)^2\sin(n\theta')\,d\theta'\\
    &=-\frac{1}{4\pi}\int_{-\pi}^\pi u^2\sin(nu+n\theta)\,du\quad(\text{by substituting } u=\theta'-\theta)\\
    &=-\cos(n\theta)\cdot\left(\frac{1}{4\pi}\int_{-\pi}^\pi u^2\sin(nu)\,du\right)-\sin(n\theta)\cdot\left(\frac{1}{4\pi}\int_{-\pi}^\pi u^2\cos(nu)\,du\right)\\
    &=\frac{(-1)^{n+1}}{n^2}\sin(n\theta)\quad(\because\text{ Lemma \ref{lemma:u2cossin}}).
\end{align*}
Moreover, since $\int_{-\pi}^{\pi}\cos(n\theta)\,d\theta=\int_{-\pi}^{\pi}\sin(n\theta)\,d\theta=0$, the proof of this claim is completed by invoking Lemma \ref{lemma:simplereigenproblem}.
\end{proof}

\begin{claim}\label{claim:S1eigenmultiplicity}
If $\lambda$ is an arbitrary eigenvalue of $\cmdsop_{\Sp^1}$ and $\phi$ is a corresponding eigenfunction, then $\lambda=0$ and $\phi$ is a nonzero constant function, or, $\lambda=\frac{(-1)^{n+1}}{n^2}$ and $\phi$ is spanned by $\theta\mapsto\cos(n\theta)$ and $\theta\mapsto\sin(n\theta)$ for some $n\geq 1$.
\end{claim}
\begin{proof}[Proof of Claim \ref{claim:S1eigenmultiplicity}]
It is well-known that $\left\{\cos(n\theta),\sin(n\theta)\right\}_{n=1}^\infty$ together with a constant function form an orthogonal basis of $L^2(\nvol_{\Sp^1})$. Hence, $\phi(\theta)=a_0+\sum_{m=1}^\infty (a_m\cos(m\theta)+b_m\sin(m\theta))$ for some $\{a_m\}_{m=0}^\infty,\{b_m\}_{m=1}^\infty\subset\R$. Therefore, we have

\begin{equation}\label{eq:S1basisepxpression}
    \cmdsop_{\Sp_1}\phi(\theta)=\sum_{m=1}^\infty \frac{(-1)^{m+1}}{m^2}(a_m\cos(m\theta)+b_m\sin(m\theta))=\lambda a_0+\sum_{m=1}^\infty \lambda(a_m\cos(m\theta)+b_m\sin(m\theta)).
\end{equation}

Since the coefficients of a vector with respect to a given basis are uniquely determined, we conclude that $\lambda=0$ and $\phi$ is a nonzero constant function, or, $\lambda=\frac{(-1)^{n+1}}{n^2}$ and $\phi=a_n\cos(n\theta)+b_n\sin(n\theta)$ for some $n\geq 1$.
\end{proof}

This completes the proof.
\end{proof}

\begin{corollary}\label{cor:propsofS1}
\begin{enumerate}
    \item $\pr(\cmdsop_{\Sp^1})=\infty$ and $\nr(\cmdsop_{\Sp^1})=\infty$.
    
    \item $\Vert\cmdsop_{\Sp^1}\Vert_1=2\cdot\sum_{n=1}^\infty\frac{1}{n^2}=\frac{\pi^2}{3}$. In particular, $\Sp^1\in\mwtr$.
    
    \item $\tr(\cmdsop_{\Sp^1})=2\cdot\left(\sum_{\text{odd }n\geq 1}\frac{1}{n^2}-\sum_{\text{even }n\geq 1}\frac{1}{n^2}\right)=\frac{\pi^2}{6}$.
    
    \item $\trng(\cmdsop_{\Sp^1})=2\cdot\sum_{\text{even }n\geq 1}\frac{1}{n^2}=\frac{1}{2}\sum_{m=1}^\infty\frac{1}{m^2}=\frac{\pi^2}{12}$. As a consequence, $\dis(\Sp^1) = \frac{\pi}{\sqrt{6}}.$
    
    \item For all $\theta,\theta'\in \Sp^1$,
$$\Vert \Phi_{\Sp^1}(\theta)-\Phi_{\Sp^1}(\theta')\Vert=\sqrt{\pi}\big(d_{\Sp^1}(\theta,\theta')\big)^{\frac{1}{2}}.$$
\end{enumerate}
\end{corollary}

Item (5) provides a conclusive answer to the discussion at the end of  \cite[Section 5]{adams2020multidimensional}. 

\begin{proof}[Proof of Corollary \ref{cor:propsofS1}]
Items (1),(2),(3), and (4) are achieved from the well-known fact $\sum_{n=1}^\infty\frac{1}{n^2}=\frac{\pi^2}{6}$ and Proposition \ref{prop:nzreigenvalueS1}. Let's prove item (5).

Without loss of generality, we may assume that $\theta\geq\theta'$ and $2\pi\geq \theta-\theta'\geq 0$. Also,
\begin{align*}
    \Vert \Phi_{\Sp^1}(\theta)-\Phi_{\Sp^1}(\theta')\Vert^2&=\sum_{\text{odd }n\geq 1}\frac{1}{n^2}\big(2(\cos(n\theta)-\cos(n\theta'))^2+2(\sin(n\theta)-\sin(n\theta'))^2\big)\\
    &=\sum_{\text{odd }n\geq 1}\frac{4}{n^2}\big(1-\cos(n(\theta'-\theta))\big)\\
    &=4\sum_{\text{odd }n\geq 1}\frac{1}{n^2}-4\sum_{\text{odd }n\geq 1}\frac{\cos(n(\theta'-\theta))}{n^2}
\end{align*}

Here, by \cite[Section 21.6]{bronshtein2013handbook}, we have
$$\sum_{\text{odd }n\geq 1}\frac{\cos(n(\theta'-\theta))}{n^2}=\begin{cases}\frac{\pi}{4}(\frac{\pi}{2}-(\theta'-\theta))&\text{ if }0\leq\theta'-\theta\leq\pi \\ \frac{\pi}{4}((\theta'-\theta)-\frac{3\pi}{2})&\text{ if }\pi<\theta'-\theta\leq2\pi. \end{cases}$$

Hence, combining the above with the fact $\sum_{\text{odd }n\geq 1}\frac{1}{n^2}=\frac{\pi^2}{8}$, one  concludes
$$
    \Vert \Phi_{\Sp^1}(\theta)-\Phi_{\Sp^1}(\theta')\Vert^2=4\left(\frac{\pi^2}{8}-\left(\frac{\pi^2}{8}-\frac{\pi}{4}\,d_{\Sp^1}(\theta,\theta')\right)\right)=\pi\,d_{\Sp^1}(\theta,\theta').
$$

\end{proof}

Item (5) of Corollary \ref{cor:propsofS1} motivates the following conjecture which we arrived at via numerical computations.

\begin{conjecture}\label{conj:sd}
For any integer $n\geq 1$, and for all $\theta,\theta'\in\Sp^n$ it holds that $$\Vert \Phi_{\Sp^n}(\theta)-\Phi_{\Sp^n}(\theta')\Vert=\sqrt{\pi}\big(d_{\Sp^n}(\theta,\theta')\big)^{\frac{1}{2}}.$$
\end{conjecture}

It turns out that, the above conjecture is true. See Theorem \ref{thm:cMDSSd-1dist}.

\begin{lemma}\label{lemma:S1ker}
The dimension of $\ker (\cmdsop_{\Sp^1})$ is exactly $1$. In other words, there are no eigenfunctions for zero eigenvalue other than  constant functions (see item (4) of Lemma \ref{lemma:simpletechgen}). 
\end{lemma}
\begin{proof}
Observe that $\phi_0\equiv 1$, $\phi_n(\theta)=\sqrt{2}\cos(n\theta)$, and $\psi_n(\theta)=\sqrt{2}\sin(n\theta)$ for $n\geq 1$ form a complete orthonormal basis of $L^2(\nvol_{\Sp^1})$. Also, $\phi_n$ and $\psi_n$ are the eigenfunctions for a nonzero eigenvalue $\frac{(-1)^{n+1}}{n^2}$ for each $n\geq 1$. This completes the proof.
\end{proof}

\subsection{The case of $\mathbb{T}^N =  (\Sp^1)^N$}\label{sec:sec:TN}

Applying the results about product spaces (see \S\ref{sec:sec:cmdsproduct}) we have the following result for tori.

\begin{corollary}\label{cor:NtoruscMDS}
Let $\mathbb{T}^N =  (\Sp^1)^N$ be the $N$-torus. Then,
\begin{enumerate}
    \item $\Vert\cmdsop_{\mathbb{T}^N}\Vert_1=\frac{N\pi^2}{3}$. In particular, $\mathbb{T}^N\in\mwtr$.
    
    \item $\tr(\cmdsop_{\mathbb{T}^N})=\frac{N\pi^2}{6}$.
    
    \item $\trng(\mathbb{T}^N)=\frac{N\pi^2}{12}$.
    
    \item For any positive integer $k$ and $(\alpha_1,\dots,\alpha_N),(\beta_1,\dots,\beta_N)\in \mathbb{T}^N$, we have
        \begin{align*}
        &\Vert \Phi_{\mathbb{T}^N,kN}(\alpha_1,\dots,\alpha_N)-\Phi_{\mathbb{T}^N,kN}(\beta_1,\dots,\beta_N) \Vert^2=\sum_{i=1}^N \Vert \Phi_{\Sp^1,k}(\alpha_i)-\Phi_{\Sp^1,k}(\beta_i) \Vert^2\text{, and}\\
        &\Vert \Phi_{\mathbb{T}^N}(\alpha_1,\dots,\alpha_N)-\Phi_{\mathbb{T}^N}(\beta_1,\dots,\beta_N) \Vert^2=\sum_{i=1}^N\Vert \Phi_{\Sp^1}(\alpha_i)-\Phi_{\Sp^1}(\beta_i) \Vert^2=\pi\sum_{i=1}^N d_{\Sp^1}(\alpha_i,\beta_i).
    \end{align*}
\end{enumerate}
\end{corollary}
\begin{proof}
Apply Corollary \ref{cor:propsofprod}, Corollary \ref{cor:propsofS1}, and Lemma \ref{lemma:S1ker}.
\end{proof}

\begin{remark}[cMDS distortion and sectional curvature]\label{rmk:curvdist}
Consider the case when $\mX\in\mwtr$ is a compact Riemannian manifold. A basic fact in Riemannian geometry is that the sectional curvature of $\mX$ measures the deviation of the  geometry of $\mX$ from being locally Euclidean. Also, as we've already seen in item (2) of Proposition \ref{prop:generrorbdd}, $\sqrt{2\,\trng(\mX)}$ equals the distortion  incurred by the cMDS embedding of  $\mX$  into the Euclidean space $\ell^2$. Therefore, it is natural to ask what is the precise relationship between the sectional curvature and the negative trace of $\mX$.

The first observation is that, even if $\mX$ has the zero sectional curvature everywhere, one cannot expect that in general $\trng(\mX)=0$. This is because the negative trace of $\mX$ is determined by the \emph{global} geometry of $\mX$, whereas  sectional curvature reflects  local information. 

Here is a specific example. Note that the mm-space $\mathbb{T}^N=\Sp^1\times\cdots\times \Sp^1$ ($N$ times) defined as in \S\ref{sec:sec:cmdsproduct} coincides with the $N$-flat torus, as a Riemannian manifold (as defined in \cite[Example 2.7]{do1992riemannian}) because the $L^2$-product of distances is generated by the product Riemannian metric.\footnote{This is a standard fact which can be easily checked.}  In other words, the sectional curvature of $\mathbb{T}^N$ is identically zero everywhere whereas $\trng(\mathbb{T}^N)=\frac{N\pi^2}{12}\rightarrow\infty$ as $N\rightarrow\infty$ by item (3) of Corollary \ref{cor:NtoruscMDS}.

On the contrary, if $\trng(\mX)=0$, then $\mX$ can be isometrically embedded in Euclidean space which  forces the sectional curvatures of $\mX$ to be zero everywhere. Hence, we think the following question is intriguing:

\begin{restatable}{question}{curvdist}
\label{q:curvdist}
When $\mX$ is a compact Riemannian manifold, is it possible to uniformly upper bound the absolute magnitude of its sectional curvature  with respect to $\trng(\mX)$?
\end{restatable}
\end{remark}

\subsection{Metric graphs and fractal spaces}\label{sec:sec:metgraphs}
In this section we consider a fairly large class of one-dimensional metric spaces called \emph{metric graphs} (see Figure \ref{fig:mg}) and prove that they all induce traceable cMDS operators.

The literature on metric graphs in the computational geometry and computational topology  keeps growing. In particular, metric graphs are one of the important tools for the shape comparison and the topological data analysis \cite{hilaga2001topology,shinagawa1991constructing,bauer2021reeb} and in the study of quantum graphs \cite{quantum}. Furthermore, the class of metric graphs also includes metric trees, which feature prominently in phylogenetics. Roughly speaking, a metric graph is a one-dimensional geodesic metric space. The precise definition of metric graph is as follows:

\begin{definition}[{\cite[Definition 3.2.11]{burago2001course}}]\label{def:mg}
A metric graph is the result of gluing of a disjoint collection of segments $\{E_i\}$ and points $\{v_j\}$ (regarded with the length metric on the disjoint union) along with an equivalence relation $R$ defined on the union of the set $\{v_j\}$ and the set of the endpoints of the segments.
\end{definition}

Note that any compact geodesic metric space can be approximated by finite graphs with respect to the Gromov-Hausdorff distance \cite[Theorem 7.5.5]{burago2001course}.\medskip

\begin{figure}[h]
\centering
\includegraphics[width=.7\textwidth]{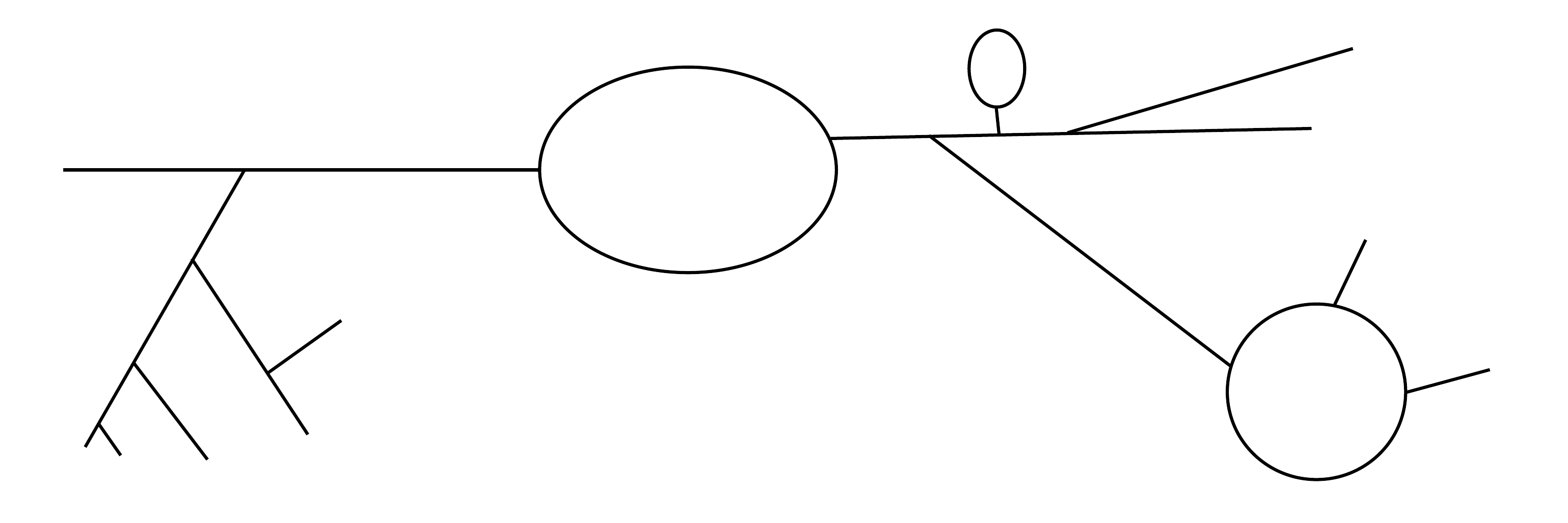}
\caption{A metric graph is, roughly speaking, a 1-dimensional geodesic metric space; see Definition \ref{def:mg}.}
\label{fig:mg}
\end{figure}

Recall that a subset $Y$ of a metric space $(X,d_X)$ is said to be a $\varepsilon$-net if $X=\cup_{y\in Y}B_\varepsilon(y,X)$. See the standard Metric Geometry reference \cite[Definition 1.6.1]{burago2001course}.

\begin{definition}[covering number]
Let $(X,d_X)$ be a totally bounded metric space. Then, given $\varepsilon>0$, the \textbf{$\varepsilon$-covering number} $N(\varepsilon,X)$ of $X$ is the smallest $n$ such that there exists a $\varepsilon$-net $\{x_1,\dots,x_n\}$ for $(X,d_X)$.
\end{definition}

Then, we have the following result.

\begin{theorem}\label{thm:metenttrace}
If $\mX=(X,d_X,\mu_X)\in\mathcal{M}_w$ satisfies $N(\varepsilon,X)\leq C\,\varepsilon^{-d}$ for some $d\in(0,2)$, $C>0$, and for all $\varepsilon\in(0,1)$, then $\cmdsop_\mX$ is trace class. Hence, $\mX\in\mathcal{M}_w^{tr}$.
\end{theorem}

We now apply this result to metric graphs. Similar consequences follow for certain fractal mm-spaces such as those considered in \cite{landry2021metric}.

\begin{corollary}
If $\mX=(X,d_X,\mu_X)\in\mw$ is a metric graph with  uniform probability measure, then $\mX$ belongs to $\mwtr$.
\end{corollary}
\begin{proof}
Since $\mX$ is a metric graph with the uniform measure, $N(\varepsilon,X)\leq C\,\varepsilon^{-1}$ for some $c>0$. Hence, if we apply Theorem \ref{thm:metenttrace}, we are able to completes the proof.
\end{proof}

In order to prove the above theorem, we need the following theorem from \cite{gonzalez1993metric} relating traceability with the covering number.

\begin{theorem}[{\cite[Theorem 1]{gonzalez1993metric}}]\label{thm:metricentropyTC}
Let $\mX=(X,d_X,\mu_X),\mY=(Y,d_Y,\mu_Y)\in\mathcal{M}_w$ be two mm-spaces, $K\in L^2(X\times Y,\mu_X\otimes\mu_Y)$ where $$\cmdsop:L^2(\mu_Y)\rightarrow L^2(\mu_X)$$ is the integral operator generated by $K$, and $K_Y:=\{K(\cdot,y):y\in Y\}\subseteq L^2(\mu_X)$. Suppose $K$ is such that for some $M<\infty$ and $r<2$, we have $N(\varepsilon,K_Y)\leq M\varepsilon^{-r}$ for any $0<\varepsilon<1$. Then $\cmdsop$ is trace class.
\end{theorem}

\begin{proof}[Proof of Theorem \ref{thm:metenttrace}]
First, observe that for any $x,x',x''\in X$
\begin{align*}
    K_\mX(x,x')-K_\mX(x,x'')
    =-\frac{1}{2}\left(d_X^2(x,x')-d_X^2(x,x'')-\int_X d_X^2(u,x')\,d\mu_X(u)+\int_X d_X^2(u,x'')\,d\mu_X(u)\right).
\end{align*}
 Hence, because of the triangle inequality and the factorization $a^2-b^2=(a+b)(a-b)$,
$$\vert K_\mX(x,x')-K_\mX(x,x'') \vert\leq 2\,\mathrm{diam}(X)\,d_X(x',x'')$$
for any $x,x',x''\in X$. Hence, $K_\mX$ is a $(2\,\mathrm{diam}(X))$-Lipschitz function with respect to the second variable. Now, fix arbitrary $\varepsilon\in(0,1)$. Then, from the assumption, we are able to choose $\{x_1,\dots,x_N\}\subseteq X$ which is an $\varepsilon$-net of $X$ and $N\leq C\,\varepsilon^{-d}$.

Now, choose arbitrary $x'\in X$. Then, there exist $x_i$ such that $d_X(x',x_i)\leq \varepsilon$. Therefore,
\begin{align*}
    \Vert K_\mX(\cdot,x')-K_\mX(\cdot,x_i) \Vert_{L^2(\mu_X)}^2&=\int_X \vert K_\mX(x,x')-K_\mX(x,x_i) \vert^2\,d\mu_X(x)
    \leq 4\,\mathrm{diam}(X)^2(d_X(x',x_i))^2.
\end{align*}
Hence,
$\Vert K_\mX(\cdot,x')-K_\mX(\cdot,x_i) \Vert_{L^2(\mu_X)}\leq 2\,\mathrm{diam}(X)\,d_X(x',x_i)\leq 2\,\mathrm{diam}(X)\,\varepsilon$
which implies that
$N(2\,\mathrm{diam}(X)\,\varepsilon,K_\mX(\cdot,X))\leq N(\varepsilon,X)\leq C\cdot\varepsilon^{-d}.$
Hence, finally,
$$N(\varepsilon,K_\mX(\cdot,X))\leq N\left(\frac{\varepsilon}{2\,\mathrm{diam}(X)},X\right)\leq C\cdot (2\,\mathrm{diam}(X))^d\cdot\varepsilon^{-d}.$$
So, by Theorem \ref{thm:metricentropyTC}, $\cmdsop_\mX$ is trace class.
\end{proof}


\section{Traceability of cMDS on higher dimensional spheres}\label{sec:Spd-1}

We will solve the  eigenvalue problem for $\cmdsop_{\Sp^{d-1}}$ when $d\geq 3$. For this, some standard results about spherical harmonics will be employed; these are collected in Appendix \S\ref{sec:sphehar}.\medskip

Since $\Sp^{d-1}$ is two-point homogeneous, by Lemma \ref{lemma:simplereigenproblem} and Corollary \ref{cor:twohomoeigenproblem}, the eigenvalue problem mentioned above is equivalent to solving the following simpler eigenvalue problem:

\begin{enumerate}
    \item[(A)] $-\frac{1}{2}\int_{\Sp^{d-1}} d_{\Sp^{d-1}}^2(u,v)\phi(v)\,d\nvol_{\Sp^{d-1}}(v)=\lambda\phi(u)$ for any $u\in\Sp^{d-1}$, and
    
    \item[(B)] $\int_{\Sp^{d-1}}\phi(u)\,d\nvol(u)=0$.
\end{enumerate}

This simplified problem can be completely solved via the Funk-Hecke theorem (cf. Theorem \ref{thm:FunkHecke}), as we describe below. We now introduce some terminology and concepts related to Harmonic analysis on spheres that will be heavily used in the rest of the paper: \label{page:multi}
\begin{itemize}
\item Below $\mathbb{Y}_n^d$ denotes the space of \emph{spherical harmonics} of order $n$ in $d$ dimensions. Any $Y_n^d\in\mathbb{Y}_n^d$ is a real valued function defined on the sphere $\Sp^{d-1}$; see Definition \ref{def:spheharmon}. Also, the dimension of $\mathbb{Y}_n^d$ is denoted by $N_{n,d}$ which equals $\frac{(2n+d-2)(n+d-3)!}{n!(d-2)!}=O(n^{d-2})$ (see \cite[(2.10),(2.12)]{atkinson2012spherical}). 
\item Below $P_{n,d}(t)$ denotes the \emph{Legendre polynomial} of degree $n$ in $d$ dimensions; see Definition \ref{def:Legendrepol}. $P_{n,d}(t)$ is characterized by the fact that the map $(x_1,\dots,x_d)\mapsto P_{n,d}(x_d)$ becomes a specific spherical harmonic of order $n$ often called \emph{Legendre harmonic} \cite[Section 2.1.2]{atkinson2012spherical}.
\end{itemize}

We are ready to state our main tool for studying the  output of the cMDS procedure on $\Sp^{d-1}$.

\begin{theorem}[Funk-Hecke theorem {\cite[Theorem 2.22]{atkinson2012spherical}}]\label{thm:FunkHecke}
Let $d\geq 3$ and $n\geq 0$ be given. If a real-valued measurable map $f$ on $[-1,1]$ satisfies
$\int_{-1}^1\vert f(t) \vert\,(1-t^2)^{\frac{d-3}{2}}\,dt<\infty$, then the following Funk-Hecke formula holds
$$\int_{S^{d-1}}f(\langle u, v\rangle)\,Y_n^d(v)\,d\vol_{S^{d-1}}(v)=\kappa_{n,d,f}\,Y_n^d(u)$$
 for any $u\in \Sp^{d-1}$ and $Y_n^d\in\mathbb{Y}_n^d$ where the real number $\kappa_{n,d,f}$ is defined by
$$\kappa_{n,d,f}:=\int_{\Sp^{d-1}} P_{n,d}(\langle u,v\rangle) f(\langle u,v \rangle)\,d\vol_{\Sp^{d-1}}(v)=\vert \Sp^{d-2}\vert \int_{-1}^1 P_{n,d}(t)f(t)(1-t^2)^{\frac{d-3}{2}}\,dt.$$
\end{theorem}

In other words, $\kappa_{n,d,f}$ is an eigenvalue and any $Y_n^d \in \mathbb{Y}_n^d$ is a corresponding eigenfunction for the kernel $(u,v)\mapsto f(\langle u,v\rangle)$. In particular, by choosing $f(t)=\arccos^p(t)$ for a given $p\in [1,\infty)$ we obtain the following result   (w.r.t the normalized volume measure):

\begin{proposition}\label{prop:eigenpgeodesic}
Let $d\geq 3$, $n\geq 0$, and $p\geq 1$ be given. Then, for any $u\in \Sp^{d-1}$ and $Y_n^d\in\mathbb{Y}_n^d$, the following equality holds
$$\int_{S^{d-1}}d_{\Sp^{d-1}}^p(u,v)\,Y_n^d(v)\,d\nvol_{S^{d-1}}(v)=\tau_{n,d,p}\,Y_n^d(u)$$
 where $\tau_{n,d,p}:=\frac{\vert \Sp^{d-2}\vert}{\vert \Sp^{d-1}\vert} \int_{-1}^1 P_{n,d}(t)\arccos^p(t)(1-t^2)^{\frac{d-3}{2}}\,dt.$
\end{proposition}

Note that $\tau_{n,d,p}=\frac{1}{\vert\Sp^{d-1}\vert}\kappa_{n,d,\arccos^p(\cdot)}$.
As a corollary to Proposition \ref{prop:eigenpgeodesic}, and because of  (A) and (B) above, we  obtain a formula for the nonzero eigenvalues of the cMDS operator $\cmdsop_{\Sp^{d-1}}$:

\begin{corollary}\label{cor:eigensphere}
Let $d\geq 3$ be given. Then, the multiset of the eigenvalues of $\cmdsop_{\Sp^{d-1}}$ consists of $0$ (with multiplicity $1$) and $\{\lambda_{n,d}\}_{n=1}^\infty$ (each with multiplicity $N_{n,d}$) where
$$\lambda_{n,d}:=-\frac{\tau_{n,d,2}}{2}=-\frac{\vert\Sp^{d-2}\vert}{2\vert\Sp^{d-1}\vert}\int_{-1}^1 P_{n,d}(t)\arccos^2(t) (1-t^2)^{\frac{d-3}{2}}\,dt\text{ for each }n\geq 1.$$

Moreover, any spherical harmonic $Y_n^d\in\mathbb{Y}_n^d$ is an eigenfunction of $\cmdsop_{\Sp^{d-1}}$ associated to the eigenvalue $\lambda_{n,d}$ for each $n\geq 1$.
\end{corollary}
\begin{proof}
Fix $n\geq 1$ and choose an arbitrary spherical harmonic $Y_n^d\in\mathbb{Y}_n^d$. 

Observe that $\int_{\Sp^{d-1}}Y_n^d(u)\,d\nvol_{\Sp^{d-1}}(u)=0$ (cf. Corollary \ref{cor:L1Ynzero}) and $$-\frac{1}{2}\int_{S^{d-1}}d_{\Sp^{d-1}}^2(\cdot,v)\,Y_n^d(v)\,d\nvol_{S^{d-1}}(v)=-\frac{\tau_{n,d,2}}{2}\,Y_n^d(\cdot)$$
holds by Proposition \ref{prop:eigenpgeodesic}. By Lemma \ref{lemma:simplereigenproblem} this concludes that $Y_n^d$ is an eigenfunction of $\cmdsop_{\Sp^{d-1}}$ associated to the eigenvalue $\lambda_{n,d}$.

It is well known that $\bigcup_{n\geq 1}\{Y_{n,j}^d\}_{j=1}^{N_{n,d}}$ together with $\phi_0\equiv 1$ form an orthonormal basis of $L^2(\nvol_{\Sp^{d-1}})$ where for each $n\geq 1$ $\{Y_{n,j}^d\}_{j=1}^{N_{n,d}}$ is an orthonormal basis of $\mathbb{Y}_n^d$. With this fact and an argument similar to the one used in the proof of Proposition \ref{prop:nzreigenvalueS1} (cf. equation (\ref{eq:S1basisepxpression})), one can easily verify that the multiplicity of the eigenvalue $0$ is exactly $1$, the multiplicity of the eigenvalue $\lambda_{n,d}$ is exactly $N_{n,d}$, and these are all possible eigenvalues of $\cmdsop_{\Sp^{d-1}}$.
\end{proof}

In the following example
we explicitly compute $\lambda_{n,3}$ for small $n$ via Corollary \ref{cor:eigensphere}.

\begin{example}\label{ex:compuS2eigen}
When $d=3$, by Corollary \ref{cor:eigensphere}, the eigenvalues $\lambda_{n,3}$ of $\cmdsop_{\Sp^2}$ have the  expression:
$$\lambda_{n,3}=-\frac{1}{4}\int_{-1}^1 P_{n,3}(t)\arccos^2(t)\,dt.$$
Since $P_{1,3}(t)=t$, $P_{2,3}(t) = \frac{3}{2}t^2-\frac{1}{2}$, $P_{3,3}(t) = \frac{5}{2}t^3-\frac{3}{2}t$, and $P_{4,3}(t) = \frac{35}{8}t^4-\frac{15}{4}t^2+\frac{3}{8}$, etc we have $$\mbox{$\lambda_{1,3} = \frac{\pi^2}{16}$, $\lambda_{2,3} = -\frac{1}{9}$, $\lambda_{3,3} = \frac{\pi^2}{256}$,  $\lambda_{4,3} = -\frac{4}{225},$ $\lambda_{5,3}=\frac{\pi^2}{1024}$ and $\lambda_{6,3}=-\frac{64}{11025}$}.$$

From this calculation, one might conjecture that $\lambda_{n,3}$ is positive for odd $n$ and negative for even $n$. This expectation turns out to be true, even for an arbitrary $d\geq 3$ (cf. Proposition \ref{prop:taylorargument}).
\end{example}

\subsection{Explicit formula for the eigenvalues of $\cmdsop_{\Sp^{d-1}}$ and consequences}\label{sec:sec:exformSd-1}

In this section, we provide two results for $\lambda_{n,d}$, the eigenvalues of $\cmdsop_{\Sp^{d-1}}$. These will help both establishing the traceability of $\cmdsop_{\Sp^{d-1}}$ and answering Conjecture \ref{conj:sd}.

First, Proposition \ref{prop:taylorargument} provides an expression of $\lambda_{n,d}$ any $n>0$ and any $d\geq 3$. Although this expression is rather complicated, since it has a form of series, from  it we are able to deduce the useful fact that $\lambda_{n,d}$ is positive for odd $n$ and negative for even $n$. This information about the sign of the eigenvalues is then crucially used to prove the traceability of $\cmdsop_{\Sp^{d-1}}$ (cf. Theorem \ref{thm:Sd-1tr}). 

\begin{proposition}\label{prop:taylorargument}
For any $n> 0$ and $d\geq 3$,
$$\lambda_{n,d}=\begin{cases}\frac{\pi\vert\Sp^{d-2}\vert}{2\vert\Sp^{d-1}\vert}\sum_{k=0}^\infty a_{\frac{n+2k-1}{2}}\cdot\frac{(n+2k)!}{2^n(2k)!}\frac{\Gamma(k+\frac{1}{2})\Gamma(\frac{d-1}{2})}{\Gamma(k+n+\frac{d}{2})}&\text{if }n\text{ is odd}\\ -\frac{\vert\Sp^{d-2}\vert}{2\vert\Sp^{d-1}\vert}\sum_{k=0}^\infty\left(\sum_{i,j:2(i+j+1)=n+2k}a_ia_j\right)\frac{(n+2k)!}{2^n(2k)!}\frac{\Gamma(k+\frac{1}{2})\Gamma(\frac{d-1}{2})}{\Gamma(k+n+\frac{d}{2})}&\text{if }n\text{ is even}
\end{cases}$$
where $a_i:=\frac{(2i)!}{2^{2i}(i!)^2(2i+1)}$ for integers $i\geq 0$. Thus, $\lambda_{n,d}>0$ if $n$ is odd and $\lambda_{n,d}<0$ if $n$ is even. 
\end{proposition}

Secondly, Proposition \ref{prop:lbdeta} establishes a precise relationship between $\lambda_{n,d}$ and $\eta_{n,d}$ for odd $n$ where $\{\eta_{n,d}\}_{n\geq 0}$ are the eigenvalues of the integral operator on $L^2(\nvol_{\Sp^{d-1}})$ generated by the kernel $(x,x')\mapsto d_{\Sp^{d-1}}(x,x')$. Thanks to this connection, we can deduce (a) the precise relationship between the geodesic metric $d_{\Sp^{d-1}}$ and the Euclidean metric obtained after applying the cMDS embedding of $\Sp^{d-1}$ into $\ell^2$ (cf. Theorem \ref{thm:cMDSSd-1dist}) and (b) a closed-form expression for $\lambda_{n,d}$ for odd $n$ (cf. Corollary \ref{cor:labdandodd}) by employing an explicit formula for $\eta_{n,d}$ for $n$ odd (cf. Lemma \ref{lemma:russianguys}).

\begin{proposition}\label{prop:lbdeta}
For any $d\geq 3$ and odd $n\geq 1$, 
$\lambda_{n,d}=-\frac{\pi}{2}\eta_{n,d}.$
\end{proposition}

Observe that $\eta_{n,d}$ is equal to the value $\tau_{n,d,1}$ given by Proposition \ref{prop:eigenpgeodesic}. Hence,

\begin{equation}\label{eq:etand}
    \eta_{n,d}=\frac{\vert \Sp^{d-2}\vert}{\vert\Sp^{d-1}\vert}\int_{-1}^1 P_{n,d}(t)\arccos(t)(1-t^2)^{\frac{d-3}{2}}\,dt
\end{equation}
and that any $Y_n^d\in\mathbb{Y}_n^d$ is an eigenfunction corresponding to $\eta_{n,d}$. Therefore, we have the following expression for $d_{\Sp^{d-1}}$:
\begin{equation}\label{eq:ds1fourier}
    d_{\Sp^{d-1}}(x,x')=\sum_{n\geq 0}\eta_{n,d}\left(\sum_{j=1}^{N_{n,d}}Y_{n,j}^d(x)\,Y_{n,j}^d(x')\right)\text{ for any }x,x'\in\Sp^{d-1}
\end{equation}
where $\{Y_{n,j}^d\}_{j=1}^{N_{n,d}}$ is any orthonormal basis of $\mathbb{Y}_n^d$. Actually, $\eta_{n,d}=0$ for even $n\geq 1$.

\begin{lemma}\label{lemma:etandeven0}
For any $d\geq 3$ and even $n\geq 1$, we have $\eta_{n,d}=0$
\end{lemma}

The proof of Lemma \ref{lemma:etandeven0} is relegated to Appendix \S\ref{sec:otherproofs}. Proposition \ref{prop:lbdeta} and Lemma \ref{lemma:etandeven0} help in settling Conjecture \ref{conj:sd}.

\begin{theorem}\label{thm:cMDSSd-1dist}
For any $d\geq 3$, and $u,v\in\Sp^{d-1}$,
$$\Vert\Phi_{\Sp^{d-1}}(u)-\Phi_{\Sp^{d-1}}(v)\Vert^2=\pi\cdot d_{\Sp^{d-1}}(u,v).$$
\end{theorem}
\begin{proof}
Observe that, by Lemma \ref{lemma:etandeven0} and equation (\ref{eq:ds1fourier}),
\begin{equation}\label{eq:ds1fourier2}
    d_{\Sp^{d-1}}(u,v)=\eta_{0,d}+\sum_{\text{odd }n\geq 1}\eta_{n,d}\left(\sum_{j=1}^{N_{n,d}}Y_{n,j}^d(u)\,Y_{n,j}^d(v)\right)\text{ for any }u,v\in \Sp^{d-1}
\end{equation}
where $\{Y_{n,j}^d\}_{j=1}^{N_{n,d}}$ is an orthonormal basis of $\mathbb{Y}_n^d$ and $N_{n,d} = \mathrm{dim}(\mathbb{Y}_n^d)$. In particular, if we choose $u=v$, we obtain
\begin{align}\label{eq:inter}
    0=\eta_{0,d}+\sum_{\text{odd }n\geq 1}\eta_{n,d}\left(\sum_{j=1}^{N_{n,d}}(Y_{n,j}^d(u))^2\right)=\eta_{0,d}+\sum_{\text{odd }n\geq 1}\eta_{n,d}N_{n,d}
\end{align}
since $\left(\sum_{j=1}^{N_{n,d}}(Y_{n,j}^d(u))^2\right)=N_{n,d}$ by the addition theorem (Theorem \ref{thm:addition}).

Also, recall that
$$\Vert\Phi_{\Sp^{d-1}}(u)-\Phi_{\Sp^{d-1}}(v)\Vert^2=K_{\widehat{\Sp}^{d-1}}(u,u)+K_{\widehat{\Sp}^{d-1}}(v,v)-2K_{\widehat{\Sp}^{d-1}}(u,v)$$
for any $u,v\in\Sp^{d-1}$. Moreover, since $\lambda_{n,d}>0$ only for odd $n$ by Proposition \ref{prop:taylorargument},
\begin{align*}
    K_{\widehat{\Sp}^{d-1}}(u,u)&=\sum_{\text{odd }n\geq 1}\lambda_{n,d}\left(\sum_{j=1}^{N_{n,d}}(Y_{n,j}^d(u))^2\right)=\sum_{\text{odd }n\geq 1}\lambda_{n,d}N_{n,d}\\
    &=\sum_{\text{odd }n\geq 1}\left(-\frac{\pi}{2}\eta_{n,d}\right)N_{n,d}=-\frac{\pi}{2}\sum_{\text{odd }n\geq 1}\eta_{n,d}N_{n,d}.
\end{align*}
where the the third equality holds by Proposition \ref{prop:lbdeta}.

Similarly, $K_{\widehat{\Sp}^{d-1}}(v,v)=-\frac{\pi}{2}\sum_{\text{odd }n\geq 1}\eta_{n,d}N_{n,d}$. Hence, via equation (\ref{eq:inter}) we find:
$$K_{\widehat{\Sp}^{d-1}}(u,u)+K_{\widehat{\Sp}^{d-1}}(v,v)=\pi\left(-\sum_{n\geq 1,\mathrm{odd}}\eta_{n,d}N_{n,d}\right)=\pi\cdot\eta_{0,d}.$$
Also,
\begin{align*}
    K_{\widehat{\Sp}^{d-1}}(u,v)=\sum_{\text{odd }n\geq 1}\lambda_{n,d}\left(\sum_{j=1}^{N_{n,d}}Y_{n,j}^d(u)Y_{n,j}^d(v)\right)
    =-\frac{\pi}{2}\sum_{\text{odd }n\geq 1}\eta_{n,d}\left(\sum_{j=1}^{N_{n,d}}Y_{n,j}^d(u)Y_{n,j}^d(v)\right).
\end{align*}
Finally, one  concludes that
\begin{align*}
    \Vert\Phi_{\Sp^{d-1}}(u)-\Phi_{\Sp^{d-1}}(v)\Vert^2&=\pi\left(\eta_{0,d}+\sum_{n\geq 1,\mathrm{odd}}\eta_{n,d}\left(\sum_{j=1}^{N_{n,d}}Y_{n,j}^d(u)Y_{n,j}^d(v)\right)\right)=\pi \, d_{\Sp^{d-1}}(u,v)
\end{align*}
where the last equality holds because of equation (\ref{eq:ds1fourier2}).
\end{proof}

\begin{remark}\label{rmk:Sd-1distortion}
Note that Theorem \ref{thm:cMDSSd-1dist} implies that
$$\delta(d-1):=\dis(\Sp^{d-1}) = \left(\pi\diamna_1(\Sp^{d-1})-\big(\diamna_2(\Sp^{d-1})\big)^2\right)^{\frac{1}{2}}.$$
Also, observe that $\diamna_1(\Sp^{d-1})=\frac{\pi}{2}$ (cf. \cite[Example 5.7]{memoli2011gromov}) and $\diamna_2(\Sp^{d-1})\rightarrow \diamna_1(\Sp^{d-1})=\frac{\pi}{2}$ as $d\uparrow \infty$ via concentration of measure considerations. This suggests that $\delta(\infty)=\frac{\pi}{2}$.
\end{remark}

\subsubsection{Proof of Proposition \ref{prop:taylorargument}.} 

In order to prove Proposition \ref{prop:taylorargument}, we need following lemmas.

\begin{lemma}\label{lemma:taylorarccos}
For any $t\in [-1,1]$,
$\arccos(t)=\frac{\pi}{2}-\sum_{n=0}^\infty a_nx^{2n+1}$
where $a_n:=\frac{(2n)!}{2^{2n}(n!)^2(2n+1)}$. This implies,
$\arccos^2(t)=\frac{\pi^2}{4}-\pi\sum_{n=0}^\infty a_nx^{2n+1}+\sum_{n,m=0}^\infty a_n a_mx^{2(n+m+1)}.$
\end{lemma}

The proof of Lemma \ref{lemma:taylorarccos} is elementary, so we omit it.

\begin{lemma}\label{lemma:tylrparity}
For any $n,m>0$, $k\geq 0$, and $d\geq 3$,
$$\int_{-1}^1 P_{n,d}(t)\,t^m(1-t^2)^{\frac{d-3}{2}}\,dt=\begin{cases}0&\text{if the parities of }n,m\text{ are different}\\
0&\text{if the parities of }n,m\text{ are the same, and }n>m\\
\frac{(n+2k)!}{2^n(2k)!}\frac{\Gamma(k+\frac{1}{2})\Gamma(\frac{d-1}{2})}{\Gamma(k+n+\frac{d}{2})}&\text{if the parity of }n,m\text{ are the same, and }m=n+2k.\end{cases}$$
\end{lemma}

The proof of Lemma \ref{lemma:tylrparity} is relegated to Appendix \S\ref{sec:otherproofs}. We are now ready to prove Proposition \ref{prop:taylorargument}.

\begin{proof}[Proof of Proposition \ref{prop:taylorargument}]
By Corollary \ref{cor:eigensphere} and Lemma \ref{lemma:taylorarccos},
\begin{align*}
    \lambda_{n,d}&=-\frac{\vert\Sp^{d-2}\vert}{2\vert\Sp^{d-1}\vert}\int_{-1}^1 P_{n,d}(t)\arccos^2(t) (1-t^2)^{\frac{d-3}{2}}\,dt\\
    &=-\frac{\vert\Sp^{d-2}\vert}{2\vert\Sp^{d-1}\vert}\Big(\frac{\pi^2}{4}\int_{-1}^1 P_{n,d}(t)(1-t^2)^{\frac{d-3}{2}}\,dt-\pi\sum_{i=0}^\infty a_i\int_{-1}^1 P_{n,d}(t) t^{2i+1}(1-t^2)^{\frac{d-3}{2}}\,dt\\
    &\quad+\sum_{i,j=0}^\infty a_i a_j \int_{-1}^1 P_{n,d}(t)t^{2(i+j+1)}(1-t^2)^{\frac{d-3}{2}}\,dt\Big).
\end{align*}

Also, $\int_{-1}^1 P_{n,d}(t)(1-t^2)^{\frac{d-3}{2}}\,dt=0$ for any $n>0$ by item (5) of Lemma \ref{lemma:Legdrprop}, which implies that the first term vanishes.

If $n$ is odd, $\int_{-1}^1 P_{n,d}(t)t^{2(i+j+1)}(1-t^2)^{\frac{d-3}{2}}\,dt=0$ and $\int_{-1}^1 P_{n,d}(t) t^{2i+1}(1-t^2)^{\frac{d-3}{2}}\,dt=0$ if $2i+1<n$ by Lemma \ref{lemma:tylrparity}. Also, for any $k\geq 0$,
$$\int_{-1}^1 P_{n,d}(t) t^{n+2k}(1-t^2)^{\frac{d-3}{2}}\,dt=\frac{(n+2k)!}{2^n(2k)!}\frac{\Gamma(k+\frac{1}{2})\Gamma(\frac{d-1}{2})}{\Gamma(k+n+\frac{d}{2})}$$
by Lemma \ref{lemma:tylrparity}. Therefore,
$$
    \lambda_{n,d}=\frac{\pi\vert\Sp^{d-2}\vert}{2\vert\Sp^{d-1}\vert}\sum_{k=0}^\infty a_{\frac{n+2k-1}{2}}\cdot\frac{(n+2k)!}{2^n(2k)!}\frac{\Gamma(k+\frac{1}{2})\Gamma(\frac{d-1}{2})}{\Gamma(k+n+\frac{d}{2})}>0.$$

If $n$ is even, $\int_{-1}^1 P_{n,d}(t) t^{2i+1}(1-t^2)^{\frac{d-3}{2}}\,dt=0$ and $\int_{-1}^1 P_{n,d}(t)t^{2(i+j+1)}(1-t^2)^{\frac{d-3}{2}}\,dt=0$ if $2(i+j+1)<n$ by Lemma \ref{lemma:tylrparity}. Also, for any $k\geq 0$,
$$\int_{-1}^1 P_{n,d}(t) t^{n+2k}(1-t^2)^{\frac{d-3}{2}}\,dt=\frac{(n+2k)!}{2^n(2k)!}\frac{\Gamma(k+\frac{1}{2})\Gamma(\frac{d-1}{2})}{\Gamma(k+n+\frac{d}{2})}$$
by Lemma \ref{lemma:tylrparity}. Therefore,
$$
    \lambda_{n,d}=-\frac{\vert\Sp^{d-2}\vert}{2\vert\Sp^{d-1}\vert}\sum_{k=0}^\infty\left(\sum_{i,j:2(i+j+1)=n+2k}a_ia_j\right)\frac{(n+2k)!}{2^n(2k)!}\frac{\Gamma(k+\frac{1}{2})\Gamma(\frac{d-1}{2})}{\Gamma(k+n+\frac{d}{2})}<0.
$$
\end{proof}

\subsubsection{Proof of Proposition \ref{prop:lbdeta}.}

\begin{proof}[Proof of Proposition \ref{prop:lbdeta}]
Recall that, by Corollary \ref{cor:eigensphere} and equation (\ref{eq:etand}),
$$\lambda_{n,d}=-\frac{\vert\Sp^{d-2}\vert}{2\vert\Sp^{d-1}\vert}\int_{-1}^1 P_{n,d}(t)\arccos^2(t)(1-t^2)^{\frac{d-3}{2}}\,dt\text{ and }\eta_{n,d}=\frac{\vert \Sp^{d-2}\vert}{\vert\Sp^{d-1}\vert}\int_{-1}^1 P_{n,d}(t)\arccos(t)(1-t^2)^{\frac{d-3}{2}}\,dt.$$

Since $n$ is odd, $P_{n,d}(-t)=-P_{n,d}(t)$ (cf. item (4) of Lemma \ref{lemma:Legdrprop}) for  $t\in[-1,1]$. Hence,
\begin{align*}
    &\int_{-1}^1 P_{n,d}(t)\arccos^2(t)(1-t^2)^{\frac{d-3}{2}}\,dt\\
    &=\left(\int_0^1 P_{n,d}(t)\arccos^2(t)(1-t^2)^{\frac{d-3}{2}}\,dt-\int_0^1 P_{n,d}(t)\arccos^2(-t)(1-t^2)^{\frac{d-3}{2}}\,dt\right)\\
    &=\left(\int_0^1 P_{n,d}(t)\arccos^2(t)(1-t^2)^{\frac{d-3}{2}}\,dt-\int_0^1 P_{n,d}(t)(\pi-\arccos(t))^2(1-t^2)^{\frac{d-3}{2}}\,dt\right)\\
    &=-\pi^2\int_0^1 P_{n,d}(t)(1-t^2)^{\frac{d-3}{2}}\,dt+2\pi\int_0^1 P_{n,d}(t)\arccos(t)(1-t^2)^{\frac{d-3}{2}}\,dt.
\end{align*}

Also,
\begin{align*}
    &\int_{-1}^1 P_{n,d}(t)\arccos(t)(1-t^2)^{\frac{d-3}{2}}\,dt\\
    &\quad=\int_0^1 P_{n,d}(t)\arccos(t)(1-t^2)^{\frac{d-3}{2}}\,dt-\int_0^1 P_{n,d}(t)\arccos(-t)(1-t^2)^{\frac{d-3}{2}}\,dt\\
    &\quad=\int_0^1 P_{n,d}(t)\arccos(t)(1-t^2)^{\frac{d-3}{2}}\,dt-\int_0^1 P_{n,d}(t)(\pi-\arccos(-t))(1-t^2)^{\frac{d-3}{2}}\,dt\\
    &\quad =2\int_0^1 P_{n,d}(t)\arccos(t)(1-t^2)^{\frac{d-3}{2}}\,dt-\pi\int_0^1 P_{n,d}(t)(1-t^2)^{\frac{d-3}{2}}\,dt.
\end{align*}

Hence,
$$\int_0^1 P_{n,d}(t)\arccos(t)(1-t^2)^{\frac{d-3}{2}}\,dt=\frac{1}{2}\left(\int_{-1}^1 P_{n,d}(t)\arccos(t)(1-t^2)^{\frac{d-3}{2}}\,dt+\pi\int_0^1 P_{n,d}(t)(1-t^2)^{\frac{d-3}{2}}\,dt\right).$$
Therefore,
$$
    \int_{-1}^1 P_{n,d}(t)\arccos^2(t)(1-t^2)^{\frac{d-3}{2}}\,dt=\pi\int_{-1}^1 P_{n,d}(t)\arccos(t)(1-t^2)^{\frac{d-3}{2}}\,dt.
$$

This implies that $\lambda_{n,d}=-\frac{\pi}{2}\eta_{n,d}$, as we required.
\end{proof}

\subsubsection{A closed-form expression for the postive eigenvalues of $\cmdsop_{\Sp^{d-1}}$.}
As another application of Proposition \ref{prop:lbdeta}, we obtain a closed-form expression for $\lambda_{n,d}$ for odd $n\geq 1$. We need the following lemma whose proof is given in Appendix \S\ref{sec:otherproofs}.

\begin{lemma}\label{lemma:russianguys}
For any $d\geq 3$ and odd $n\geq 1$, we have
$\eta_{n,d}=-\frac{\pi^{\frac{d-1}{2}}n!!\Gamma(\frac{d}{2})\Gamma(\frac{n}{2})}{n2^{\frac{n-1}{2}}(\Gamma(\frac{n+d}{2}))^2\vert\Sp^{d-1}\vert}.$
\end{lemma}

\begin{corollary}[formula for positive eigenvalues of $\cmdsop_{\Sp^{d-1}}$]\label{cor:labdandodd}
For any $d\geq 3$ and odd $n\geq 1$, we have 
\begin{equation}\label{eq:lambda-formula}
\lambda_{n,d}=\frac{\pi^{\frac{d+1}{2}}n!!\Gamma(\frac{d}{2})\Gamma(\frac{n}{2})}{n2^{\frac{n+1}{2}}(\Gamma(\frac{n+d}{2}))^2\vert\Sp^{d-1}\vert}.\end{equation}
\end{corollary}

\begin{remark}\label{rmk:evaldifficult}
Observe that although the formula for $\lambda_{n,d}$ given by Proposition \ref{prop:taylorargument} is very intricate, Corollary \ref{cor:labdandodd} provides an explicit expression $\lambda_{n,d}$ for odd $n$ without evaluating the formula.
\end{remark}

\begin{remark}
Observe that, by item (2) of Proposition \ref{prop:generrorbdd}, Remark \ref{rmk:Sd-1distortion}, and Theorem \ref{thm:cMDSSd-1dist},
$$\trng(\cmdsop_{\Sp^{d-1}})=\frac{\big(\dis({\Sp^{d-1}})\big)^2}{2}=\frac{\pi}{2}\diamna_1(\Sp^{d-1})-\frac{1}{2}\big(\diamna_2(\Sp^{d-1})\big)^2=\frac{\pi^2}{4}-\frac{1}{2}\big(\diamna_2(\Sp^{d-1})\big)^2.$$

As an application of the above fact, we can verify that, despite the apparent complexity of equation (\ref{eq:lambda-formula}), evaluating the series of $\lambda_{n,d}$ for odd $n\geq 1$ given in Corollary \ref{cor:labdandodd} becomes very simple; note that $\tr(\cmdsop_{\Sp^{d-1}})=\sum_{\text{odd }n\geq 1}\lambda_{n,d}-\trng(\cmdsop_{\Sp^{d-1}})=\frac{1}{2}\big(\diamna_2(\mX)\big)^2$ (cf. Remark \ref{rmk:tracediam2}). Hence, the sum of the positive eigenvalues of $\cmdsop_{\Sp^{d-1}}$ is,
$$\sum_{\text{odd }n\geq 1}\lambda_{n,d}=\frac{1}{2}\big(\diamna_2(\mX)\big)^2+\trng(\cmdsop_{\Sp^{d-1}})=\frac{\pi^2}{4}.$$
Moreover, with this we can also establish that $\diamna_2(\widehat{\Sp}^{d-1})$ does not depend on $d$:
$$\diamna_2(\widehat{\Sp}^{d-1})=\sqrt{2\,\tr(\cmdsop_{\widehat{\Sp}^{d-1}})}=\sqrt{2\,\sum_{\text{odd }n\geq 1}\lambda_{n,d}}=\frac{\pi}{\sqrt{2}}.$$
\end{remark}

\subsection{Traceability of $\cmdsop_{\Sp^{d-1}}$}\label{sec:sec:Sd-1traceable}

We already know that $\cmdsop_{\Sp^1}$ is trace class. In this section, we prove the following theorem which establishes the traceability of $\cmdsop_{\Sp^{d-1}}$ for any $d\geq 3$.

\begin{theorem}\label{thm:Sd-1tr}
For any $d\geq 3$, $\cmdsop_{\Sp^{d-1}}$ is trace class.
\end{theorem}

Recall from page \pageref{sec:sec:sec:trblmmspace} that $\cmdsop_{\Sp^{d-1}}$ being trace class means that $\|\cmdsop_{\Sp^{d-1}}\|_1$, the sum of absolute values of all eigenvalues of $\cmdsop_{\Sp^{d-1}}$ (with multiplicity), is finite. Via Corollary \ref{cor:eigensphere}, this  leads to the condition $\sum_{n=1}^\infty \vert\lambda_{n,d}\vert N_{n,d}<\infty$. Therefore, in order to prove Theorem \ref{thm:Sd-1tr}, perhaps the most natural strategy would be to directly estimate the rate of decay of $\vert\lambda_{n,d}\vert$ to zero and combine this with the well known fact that the multiplicities satisfy $N_{n,d}=O(n^{d-2})$ (see page \pageref{page:multi}). More precisely,  in order to guarantee that the aforementioned series converges, it would be enough to show that $\vert\lambda_{n,d}\vert=O(n^{-d+1-\delta})$ for some $\delta>0$. Whereas this approach is successful for the case when $d=3$  (even when the metric on $\Sp^2$ is more general than the Euclidean or geodesic one; see \S\ref{sec:sec:Sd-1diffmet}), it turns out to be   difficult to implement for higher values of $d$.

However, a strategy exploiting the so called Poisson identity for Legendre polynomials (cf. item (7) of Lemma \ref{lemma:Legdrprop}) leads to a successful general approach. By following this strategy, in the proof of Theorem \ref{thm:Sd-1tr}, we will require the integrability of a certain function in order to apply the dominated convergence theorem. The following technical lemma will be useful for that; its proof is given in Appendix \S\ref{sec:otherproofs}.

\begin{lemma}\label{lemma:tlfrtrc}
Suppose a positive integer $d\geq 3$ is given. Then,
$$\int_{-1}^1\frac{\arccos^2(t)(1-t^2)^{\frac{d-3}{2}}}{(1+r^2+2rt)^{\frac{d}{2}}}\,dt\in \left[0,\frac{2\pi^2}{1-r^2}\right]
\,\,\mbox{for each $r\in (0,1)$.}$$
\end{lemma}

With this we can prove the following theorem.

\begin{proof}[Proof of Theorem \ref{thm:Sd-1tr}]
Proving that $\cmdsop_{\Sp^{d-1}}$ is trace class is equivalent to showing that  $\|\cmdsop_{\Sp^{d-1}}\|_1 = \sum_{n=1}^\infty \vert\lambda_{n,d}\vert N_{n,d}<\infty$ (see Theorem \ref{thm:tracesingvaluecndtn}). Fix arbitrary $r\in (0,1)$. Then,
\begin{align*}
    &\sum_{n=1}^\infty \vert\lambda_{n,d}\vert N_{n,d}r^n\\
    &=-\sum_{n=1}^\infty \lambda_{n,d}(-r)^n N_{n,d}\quad(\because\, \lambda_{n,d}>0\text{ for odd }n\text{ and }\lambda_{n,d}<0\text{ for even }n \text{ by Proposition \ref{prop:taylorargument}})\\
    &=\frac{\vert\Sp^{d-2}\vert}{2\vert\Sp^{d-1}\vert}\sum_{n=1}^\infty\int_{-1}^1 P_{n,d}(t)\arccos^2(t)(1-t^2)^{\frac{d-3}{2}}(-r)^n N_{n,d}\,dt\\
    &=\frac{\vert\Sp^{d-2}\vert}{2\vert\Sp^{d-1}\vert}\left(\sum_{n=0}^\infty\int_{-1}^1 P_{n,d}(t)\arccos^2(t)(1-t^2)^{\frac{d-3}{2}}(-r)^n N_{n,d}\,dt-\int_{-1}^1\arccos^2(t)(1-t^2)^{\frac{d-3}{2}}\,dt\right)\\
    &=\frac{\vert\Sp^{d-2}\vert}{2\vert\Sp^{d-1}\vert}\sum_{n=0}^\infty\int_{-1}^1 P_{n,d}(t)\arccos^2(t)(1-t^2)^{\frac{d-3}{2}}(-r)^n N_{n,d}\,dt-\frac{1}{2}\big(\diamna_2(\Sp^{d-1})\big)^2\\
    &=\frac{\vert\Sp^{d-2}\vert}{2\vert\Sp^{d-1}\vert}\int_{-1}^1\sum_{n=0}^\infty P_{n,d}(t)\arccos^2(t)(1-t^2)^{\frac{d-3}{2}}(-r)^n N_{n,d}\,dt-\frac{1}{2}\big(\diamna_2(\Sp^{d-1})\big)^2.
\end{align*}

Here, note that the last equality holds by the dominated convergence theorem since, by item (7) of Lemma \ref{lemma:Legdrprop} and Lemma \ref{lemma:tlfrtrc}, the series
$$\sum_{n=0}^k P_{n,d}(t)\arccos^2(t)(1-t^2)^{\frac{d-3}{2}}(-r)^n N_{n,d}\,\,\stackrel{k\rightarrow\infty}{\longrightarrow}\,\,\frac{(1-r^2)\arccos^2(t)(1-t^2)^{\frac{d-3}{2}}}{(1+r^2+2rt)^{\frac{d}{2}}}\,\,\,\,\mbox{uniformly}$$
where the limit function is non-negative and integrable. Then,
\begin{align*}
    &\frac{\vert\Sp^{d-2}\vert}{2\vert\Sp^{d-1}\vert}\int_{-1}^1\sum_{n=0}^\infty P_{n,d}(t)\arccos^2(t)(1-t^2)^{\frac{d-3}{2}}(-r)^n N_{n,d}\,dt-\frac{1}{2}\big(\diamna_2(\Sp^{d-1})\big)^2\\
    &=\frac{\vert\Sp^{d-2}\vert}{2\vert\Sp^{d-1}\vert}\int_{-1}^1\frac{(1-r^2)\arccos^2(t)(1-t^2)^{\frac{d-3}{2}}}{(1+r^2+2rt)^{\frac{d}{2}}}\,dt-\frac{1}{2}\big(\diamna_2(\Sp^{d-1})\big)^2\\
    &\leq\frac{\vert\Sp^{d-2}\vert\pi^2}{\vert\Sp^{d-1}\vert}-\frac{1}{2}\big(\diamna_2(\Sp^{d-1})\big)^2\quad(\because\text{ Lemma \ref{lemma:tlfrtrc}}).
\end{align*}
Let $C_d:=\frac{\vert\Sp^{d-2}\vert\pi^2}{\vert\Sp^{d-1}\vert}-\frac{1}{2}\big(\diamna_2(\Sp^{d-1})\big)^2>0$. We have  proved $\sum_{n=1}^\infty \vert\lambda_{n,d}\vert N_{n,d}r^n< C_d$ for every $r\in(0,1)$.
\medskip

Now, to obtain a contradiction, suppose $\sum_{n=1}^\infty \vert\lambda_{n,d}\vert N_{n,d}=\infty$ and  choose a positive integer $m>0$ such that
$\sum_{n=1}^m \vert\lambda_{n,d}\vert N_{n,d}>2C_d.$
Also, choose $r\in (0,1)$ so that $r>2^{-\frac{1}{m}}$. Then,
$$\sum_{n=1}^\infty \vert\lambda_{n,d}\vert N_{n,d}r^n\geq\sum_{n=1}^m \vert\lambda_{n,d}\vert N_{n,d}r^n\geq \left(\sum_{n=1}^m \vert\lambda_{n,d}\vert N_{n,d}\right)r^m>2C_dr^m>C_d,$$
which is a contradiction. Hence, $\sum_{n=1}^\infty \vert\lambda_{n,d}\vert N_{n,d}<\infty$
as we required.
\end{proof}

\subsection{Metric transforms of $\Sp^2$}\label{sec:sec:Sd-1diffmet}

In this section, we introduce a family of metrics on $\Sp^2$ all of which ensure the traceability of their respective cMDS operators.

Suppose $f:[-1,1]\rightarrow\R$ is an almost everywhere (with respect to the Lebesgue measure) smooth function such that
$$d_f:\Sp^2\times\Sp^2\longrightarrow\R\text{  s.t.  }(u,v)\longmapsto f(\langle u,v\rangle)$$
becomes a metric on $\Sp^2$ compatible with the standard topology of $\Sp^2$. Then, this leads to considering the mm-space $\Sp^2_f:=(\Sp^2,d_f,\nvol_{\Sp^2})$ where the metric on the sphere is now given by $d_f$ instead of the usual Euclidean or geodesic metrics. Let $M(\Sp^2)$ denote the set of all maps $f:[-1,1]\rightarrow\R$ satisfying the above properties. Some examples of $f$ in $M(\Sp^2)$ are given snowflake transforms of the Euclidean or geodesic distance on $\Sp^2$:
\begin{itemize}
    \item $f(t)=\big(\arccos (t)\big)^{\frac{1}{p}}$ for $p\geq 1$ (if $p=1$, $d_f$ is the geodesic metric on $\Sp^2$).
    
    \item $f(t)=\big(\sqrt{2-2t}\big)^{\frac{1}{p}}$ for $p\geq 1$ (if $p=1$, $d_f$ is the Euclidean metric on $\Sp^2$ inherited from $\R^d$ for any $d\geq 2$).
\end{itemize}

Moreover, as in Corollary \ref{cor:eigensphere}, we can verify that any spherical harmonic $Y_n^3\in\mathbb{Y}_n^3$ is an eigenfunction of $\cmdsop_{\Sp^2_f}$ for the eigenvalue
\begin{equation}\label{eq:lambdan2f}
    \lambda_{n,3}^f:=-\frac{\vert\Sp^1\vert}{2\vert\Sp^2\vert}\int_{-1}^1 P_{n,3}(t)f^2(t)\,dt.
\end{equation}

Theorem  \ref{thm:S2ftraceable} below follows by estimating the rate of convergence of $\vert\lambda_{n,3}^f\vert$ to $0$ with respect to $n$ (cf. Lemma \ref{lemma:lambdaf3dbdd}). For each $f\in M(\Sp^2)$, we define the following map:
$$ D_f:\Sp^2\longrightarrow\R\text{  s.t.  }u\longmapsto d_f^2(e_3,u)\text{ where }e_3:=(0,0,1).$$

Below, $\Delta$ denotes the Laplace-Beltrami operator on $\Sp^2$. We now reduce the question of traceability of $\cmdsop_{\Sp^2_f}$ to a certain integrability property of the Laplacian of the function $D_f$.

\begin{theorem}\label{thm:S2ftraceable}
Let $f\in M(\Sp^2)$ and $\delta\in (0,1)$ be s.t.  $\Vert\Delta D_f\Vert_{L^{1+\delta}(\vol_{\Sp^2})}<\infty$. Then $\Sp^2_f\in\mwtr$.
\end{theorem}

\begin{example}\label{ex:applmetransform}
As an application of Theorem \ref{thm:S2ftraceable}, we verify that $\Sp^2_f\in\mwtr$ when:
\begin{itemize}
\item $f(t)=\arccos(t)$ (so that $d_f=d_{\Sp^2}$ and we recover Theorem \ref{thm:Sd-1tr} for the case when $d=3$) and 
\item $f(t)=(2-2t)^{\frac{1}{4}}$ (so that $d_f$ is the square root of the Euclidean metric). 
\end{itemize}

Note that, in the latter case, $\Sp^2_f$ is a metric space of \emph{negative type} \cite[Section 6.1]{deza-laurent}. Indeed, this follows form the facts that (1) Any  metric space  $(X,d_X)$ admitting an isometric embedding into Euclidean space is of negative type \cite[Theorem 6.2.2]{deza-laurent} and (2) If $(Z,d_Z)$ is a metric space of negative type, then $(Z,d_Z^{1/2})$ is also of negative type \cite[Example 9.1.6]{deza-laurent}. As a consequence, any finite subset of $\Sp^2_f$ can be isometrically embedded into a finite dimensional Euclidean space and $\Sp^2_f$ itself can be isometrically embedded into Hilbert space (see \cite[Theorem 6.2.2]{deza-laurent}). Therefore, this suggests trying to find one such embedding via cMDS, which in turn requires verifying that $\cmdsop_{\Sp_f^2}$ is traceable.\footnote{Note that one does not, however, expect that the entire $\Sp^2_f$ admits a finite dimensional isometric embedding into $\ell^2$. In fact, one can prove that all (the infinitely many) eigenvalues of the cMDS operator associated to $\Sp^2_f$ are positive in this case.} See \S\ref{sec:sec:metransformexdet} below for details about this example. 
\end{example}

Note that the definitions of $d_f$, $\Sp^{d-1}_f$, and $M(\Sp^{d-1})$ can be extended to the case of an arbitrary $d\geq 3$. Unfortunately, though, for $d\geq 4$ the argument given in the proof of Lemma \ref{lemma:lambdaf3dbdd} does not provide a fast enough rate of convergence to $0$ which could help establish the traceability of the cMDS operator of  $\Sp^{d-1}_f$ for an arbitrary $f\in M(\Sp^{d-1})$. Moreover, the argument via the Poisson identity used in  Theorem \ref{thm:Sd-1tr} does not seem applicable to the study $\Sp^{d-1}_f$ for an arbitrary $f\in M(\Sp^{d-1})$ for $d\geq 4$. The general classification of $f$ such that $\Sp^{d-1}_f\in\mwtr$ is currently not known. Related questions can be traced back to some classical papers by Von Neumann and Schoenberg such as \cite{von1941fourier}.

\subsubsection{Details about Example \ref{ex:applmetransform}.}\label{sec:sec:metransformexdet}

The example hinges on the following two claims.

\begin{claim}\label{claim:farccos}
If $f(t)=\arccos(t)$, then $\Sp^2_f\in\mwtr$.
\end{claim}

\begin{claim}\label{claim:fsqrtEuclid}
If $f(t)=(2-2t)^{\frac{1}{4}}$, then $\Sp^2_f\in\mwtr$.
\end{claim}

\begin{lemma}\label{lemma:Laplaciandistance}
For any $d\geq 3$,
$\Delta\,d_{\Sp^{d-1}}(e_d,\cdot)(u)=-(d-2)\cdot\cot d_{\Sp^{d-1}}(e_d,u)$
for any $u=(u_1,\cdots,u_d)\in\Sp^{d-1}\backslash\{e_d,-e_d\}$ where $e_d=(0,\cdots,0,1)\in\Sp^{d-1}$ is the north pole of $\Sp^{d-1}$.
\end{lemma}

The proof of Lemma \ref{lemma:Laplaciandistance} is relegated to Appendix \S\ref{sec:otherproofs}.

\begin{proof}[Proof of Claim \ref{claim:farccos}]
Fix an arbitrary $\delta\in(0,1)$. Since $d_f=d_{\Sp^2}$, observe that
\begin{align*}
    \Delta D_f&=\Delta d_{\Sp^2}^2(e_3,\cdot)=\nabla\cdot\nabla d_{\Sp^2}^2(e_3,\cdot)=2\nabla\cdot(d_{\Sp^2}(e_3,\cdot)\nabla d_{\Sp^2}(e_3,\cdot))\\
    &=2(\Vert \nabla d_{\Sp^2}(e_3,\cdot) \Vert^2+d_{\Sp^2}(e_3,\cdot)\Delta d_{\Sp^2}(e_3,\cdot)).
\end{align*}

Also, since $d_{\Sp^2}(e_3,\cdot)$ is $1$-Lipschitz, $\Vert \nabla d_{\Sp^2}(e_3,\cdot) \Vert\leq 1$ a.e.. Moreover,
\begin{align*}
    \int_{\Sp^2}\vert d_{\Sp^2}(e_3,u)\Delta d_{\Sp^2}(e_3,u)\vert^{(1+\delta)}\,d\,\vol_{\Sp^2}(u)&=\int_{\Sp^2}\vert d_{\Sp^2}(e_3,u)\cot d_{\Sp^2}(e_3,u)\vert^{(1+\delta)}\,d\,\vol_{\Sp^2}(u)\quad(\because \text{Lemma \ref{lemma:Laplaciandistance}})\\
    &=\vert\Sp^1\vert\int_0^\pi \vert\theta\cot\theta \vert^{1+\delta}\sin\theta\,d\theta\\
    &=\vert\Sp^1\vert\int_0^\pi\frac{\theta^{1+\delta}\vert\cos\theta\vert^{1+\delta}}{(\sin\theta)^\delta}\,d\theta.
\end{align*}
Let's bound $\int_0^\pi\frac{\theta^{1+\delta}\vert\cos\theta\vert^{1+\delta}}{(\sin\theta)^\delta}\,d\theta$. Observe that,
\begin{align*}
    \int_0^{\frac{\pi}{2}}\frac{\theta^{1+\delta}\vert\cos\theta\vert^{1+\delta}}{(\sin\theta)^\delta}\,d\theta&\leq \int_0^{\frac{\pi}{2}}\frac{\theta^{1+\delta}}{(\sin\theta)^\delta}\,d\theta\quad(\because\cos\theta\in[-1,1])\\
    &\leq\int_0^{\frac{\pi}{2}}\theta\left(\frac{\pi}{2}\right)^\delta\,d\theta\quad\left(\because\frac{\theta}{\sin\theta}\leq\frac{\pi}{2}\right)<\infty,
\end{align*}
and
\begin{align*}
    \int_{\frac{\pi}{2}}^\pi \frac{\theta^{1+\delta}\vert\cos\theta\vert^{1+\delta}}{(\sin\theta)^\delta}\,d\theta&\leq\pi^{1+\delta}\int_{\frac{\pi}{2}}^\pi\frac{\vert\cos\theta\vert}{(\sin\theta)^\delta}\,d\theta\quad(\because\cos\theta\in[-1,1])\\
    &=-\pi^{1+\delta}\int_{\frac{\pi}{2}}^\pi\frac{\cos\theta}{(\sin\theta)^\delta}\,d\theta=\pi^{1+\delta}\int_0^1\frac{1}{t^\delta}\,dt=\frac{\pi^{1+\delta}}{1-\delta}<\infty.
\end{align*}

So, we conclude that $\Vert\Delta D_f \Vert_{L^{1+\delta}(\vol_{\Sp^2})}<\infty$. Hence, $\Sp^2_f\in\mwtr$ by Theorem \ref{thm:S2ftraceable}.
\end{proof}

\begin{proof}[Proof of Claim \ref{claim:fsqrtEuclid}]
In this case, we verify that $\Vert\Delta D_f\Vert_{L^{1+\delta}(\vol_{\Sp^2})}<\infty$ when $\delta=\frac{1}{2}$.

Indeed, it is easy to check that  in the coordinate system where $\Sp^2\ni u= (\sqrt{1-t^2}\cos\theta,\sqrt{1-t^2}\sin\theta,t)$,  $t\in[-1,1]$ and $\theta\in[0,2\pi]$ we have:
$    \Vert \nabla d_f(e_3,\cdot)(\sqrt{1-t^2}\cos\theta,\sqrt{1-t^2}\sin\theta,t) \Vert=\frac{\sqrt{1+t}}{2^{\frac{3}{2}}(2-2t)^{\frac{1}{4}}},
$ and $\Delta d_f(e_3,\cdot)(\sqrt{1-t^2}\cos\theta,\sqrt{1-t^2}\sin\theta,t)=-\frac{(2-2t)^{\frac{1}{4}}}{4}+\frac{(1+t)}{8(2-2t)^{\frac{3}{4}}}.$
Therefore,

\begin{align*}
    \Delta D_f(\sqrt{1-t^2}\cos\theta,\sqrt{1-t^2}\sin\theta,t)&=2(\Vert \nabla d_f(e_3,\cdot) \Vert^2+d_f(e_3,\cdot)\Delta d_f(e_3,\cdot))
    &=\frac{(1+t)}{4\sqrt{2-2t}}-\frac{\sqrt{2-2t}}{4}
\end{align*}

for any $\theta\in \R$ and $t\in (-1,1)$. The $L^{\frac{3}{2}}(\vol_{\Sp^2})$-norm of the second term is obviously finite since $\vert\frac{\sqrt{2-2t}}{4}\vert\leq\frac{1}{2}$. For the first term, observe that:
 \begin{align*}
    \int_{\Sp^2} \left\vert\frac{(1+t)}{4\sqrt{2-2t}}\right\vert^{\frac{3}{2}}\,d\vol_{\Sp^2}(\sqrt{1-t^2}\cos\theta,\sqrt{1-t^2}\sin\theta,t)&=  2\pi\int_{-1}^1 \left(\frac{(1+t)}{4\sqrt{2-2t}}\right)^{\frac{3}{2}}\,dt\\ 
    &\leq \frac{\pi}{\sqrt{2}}\int_{-1}^1 \frac{1}{(2-2t)^{\frac{3}{4}}}\,dt=2\pi<\infty.
\end{align*}

Therefore $\Vert\Delta D_f\Vert_{L^{\frac{3}{2}}(\vol_{\Sp^2})}<\infty$ as we required. Hence, $\Sp^2_f\in\mwtr$ by Theorem \ref{thm:S2ftraceable}.
\end{proof}

\subsubsection{Proof of Theorem \ref{thm:S2ftraceable}.}
In order to establish Theorem \ref{thm:S2ftraceable}, we need the following lemma which provides estimates for the rate of decay of all eigenvalues of $\cmdsop_{\Sp^2_f}$. The proof of Lemma \ref{lemma:lambdaf3dbdd} is relegated to Appendix \S\ref{sec:otherproofs}.

\begin{lemma}\label{lemma:lambdaf3dbdd}
Suppose $\Sp^2_f\in\mw$ is given for a map $f:[-1,1]\rightarrow\R$. Then, for any $\delta\in (0,1)$, there exists a constant $C_\delta'>0$ such that
$$\vert \lambda_{n,3}^f\vert\leq\frac{C_\delta' \Vert\Delta F(u) \Vert_{L^{1+\delta}(\vol_{\Sp^2})}}{n(n+1)}N_{n,3}^{-\frac{1}{1+\frac{1}{\delta}}}\,\,\mbox{
for all $n\geq 1$}$$
where
$$F:\Sp^2\longrightarrow\R\text{  s.t.  }u\longmapsto d_f^2(e_3,u).$$
\end{lemma}

\begin{proof}[Proof of Theorem \ref{thm:S2ftraceable}]
Choose $\delta\in (0,1)$ such that $\Vert\Delta F\Vert_{L^{1+\delta}(\vol_{\Sp^2})}<\infty$. Since by Lemma \ref{lemma:lambdaf3dbdd}, we know $\vert \lambda_{n,3}^f\vert\leq\frac{C_\delta' \Vert\Delta F(u) \Vert_{L^{1+\delta}(\vol_{\Sp^2})}}{n(n+1)}N_{n,3}^{-\frac{1}{1+\frac{1}{\delta}}}$, we have
\begin{align*}
    \sum_{n=1}^\infty N_{n,3}\vert \lambda_{n,3}^f\vert\leq\sum_{n=1}^\infty \frac{C_\delta' \Vert\Delta F(u) \Vert_{L^{1+\delta}(\vol_{\Sp^2})}}{n(n+1)}N_{n,3}^{1-\frac{1}{1+\frac{1}{\delta}}}.
\end{align*}
But, since $N_{n,3}=2n+1$,
$\frac{N_{n,3}^{1-\frac{1}{1+1/\delta}}}{n(n+1)}=O(n^{\frac{1}{1+\delta}-2}).$
Since $\frac{1}{1+\delta}-2<-1$, we  conclude  that
$\sum_{n=1}^\infty N_{n,3}\vert \lambda_{n,3}^f\vert$ converges so that $\cmdsop_{\Sp^2_f}$ is trace class.
\end{proof}


\section{Stability of Generalized cMDS}\label{sec:stability}
In this section we establish the stability of the coordinate-free generalized cMDS procedure (cf. \S\ref{sec:sec:sec:altcmds}) in the sense of the Gromov-Wasserstein distance. See \S \ref{sec:sec:finstab} for a brief overview of  existing results about the stability of cMDS on finite metric spaces.

\subsection{The Gromov-Wasserstein distance}\label{sec:sec:dGW} 

In \cite{memoli2007use,memoli2011gromov}, the author defines the Gromov-Wasserstein distances combining ideas related to the Wasserstein distance and the Gromov-Hausdorff distance. Let $(X,d_X)$ be a compact metric space and let $\mathcal{P}(X)$ be the set of all Borel probability measures on $X$. Let us recall the definition of a coupling measure.

\begin{definition}
Let $(X,\Sigma_X,\mu_X)$ and $(Y,\Sigma_Y,\mu_Y)$ be two measure spaces. Then, we call a measure $\mu$ on $X\times Y$ a \emph{coupling measure} between $\mu_X$ and $\mu_Y$ if $\mu(A\times Y)=\mu_X(A)$ and $\mu(X\times B)=\mu_Y(B)$ whenever $A$ and $B$ are any measurable subsets of $X$ and $Y$, respectively. Also, we denote with $\Gamma(\mu_X,\mu_Y)$  the set of all coupling measures.
\end{definition}

\begin{remark}
 $\Gamma(\mu_X,\mu_Y)$ is  nonempty since this set always contains the product measure $\mu_X\otimes\mu_Y$.
\end{remark}

The Wasserstein distance  on $\mathcal{P}(X)\times\mathcal{P}(X)$ is defined as follows.

\begin{definition}[Wasserstein distance \cite{villani2003topics}]
Assume $(X,d_X)$ is a compact metric space. For each $p\in[1,\infty)$ and $\alpha,\beta\in\mathcal{P}(X)$,
$$d_{\mathrm{W},p}^X(\alpha,\beta):=\inf\limits_{\mu\in\Gamma(\alpha,\beta)}\left(\int_{X}d_X^p(x,x')d\mu(x,x')\right)^{\frac{1}{p}}.$$
\end{definition}

Now, consider the collection of all compact metric measure spaces (=mm-spaces), that is triples $\mX=(X,d_X,\mu_X)$ where $(X,d_X)$ is a compact metric space and $\mu_X$ is a fully supported Borel probability measure on $X$. This collection is denoted by $\mathcal{M}_w$.

\begin{definition}[Gromov-Wasserstein distances \cite{memoli2011gromov}]\label{GWdef}
For two mm-spaces $\mX=(X,d_X,\mu_X)$ and $\mY=(Y,d_Y,\mu_Y)$ and each $p\in[1,\infty)$,
$$d_{\mathrm{GW},p}(\mX,\mY):=\frac{1}{2}\inf\limits_{\mu\in\Gamma(\mu_X,\mu_Y)}\left(\int_{X\times Y\times X\times Y}\vert d_X(x,x')-d_Y(y,y')\vert^p d(\mu\otimes\mu)(x,y,x',y')\right)^{\frac{1}{p}}.$$
\end{definition}

The Gromov-Wasserstein distance $d_{\mathrm{GW},p}$ defines a pseudometric on $\mw$ and it can be extended to a proper metric on the collection of isomorphism classes of mm-spaces \cite{memoli2011gromov}. We say two mm-spaces $\mX=(X,d_X,\mu_X)$ and $\mY=(Y,d_Y,\mu_Y)$ are isomorphic if there is a bijective map $f\colon X\rightarrow Y$ satisfying $d_X(x,x')=d_Y(f(x),f(x'))$ for all $x,x'\in X$ and $\mu_Y(B)=\mu_X(f^{-1}(B))$ for any Borel set $B\subseteq Y$. i.e., $\mu_Y=f_{*}\mu_X$. We call this $\mu_Y$ the pushforward measure induced by $f$.

\begin{remark}[Sturm's $d_{\mathrm{GW},p,q}$]
In \cite{sturm2012space}, Sturm defined the $(p,q)$-Gromov-Wasserstein distance $d_{\mathrm{GW},p,q}$ for any $p,q\in [1,\infty)$ by substituting $d_X$ and $d_Y$ for $d_X^q$ and $d_Y^q$ in Definition \ref{GWdef}.  Obviously, we have $d_{\mathrm{GW},p,1}=d_{\mathrm{GW},p}$ for any $p\in [1,\infty)$. Sturm proved that $d_{\mathrm{GW},p,q}$  defines a pseudometric on $\mw$ for any $p,q\in [1,\infty)$. We point out that the stability result given in Theorem \ref{thm:stability} (which we will prove in the next section) still holds if we use $d_{\mathrm{GW},2,2}$ instead of $d_{\mathrm{GW},2}$.
\end{remark}

\begin{remark}\label{rmk:dgwpdiam}
Suppose $\mX,\mY\in\mw$ and $p\in[1,\infty)$ are given. The following properties of $d_{\mathrm{GW},p}$ and $\diamna_p$ are proved in \cite{memoli2011gromov}: 
\begin{enumerate}
    \item $\diamna_p(\mX)\leq\diamna(X)$.
    
    \item $d_{\mathrm{GW},p}(\mX,\star)=\frac{1}{2}\diamna_p(\mX)$.
    
    \item $\frac{1}{2}\vert \diamna_p(\mX)-\diamna_p(\mY)\vert\leq d_{\mathrm{GW},p}(\mX,\mY)\leq\frac{1}{2}(\diamna_p(\mX)+\diamna_p(\mY))$.
\end{enumerate}
\end{remark}

\subsection{Stability}\label{sec:sec:stability}
It is reasonable to expect that, the cMDS method is stable with respect to ``perturbations" of mm-spaces in the Gromov-Wasserstein sense. This expectation indeed turns out to be true for mm-spaces satisfying the ``Lebesgue differentiability property". 

\begin{definition}
A mm-space $\mX=(X,d_X,\mu_X)$ is said to satisfy the  \textbf{Lebesgue differentiability property} if for any $f\in L^1(\mu_X)$
$$\lim_{r\rightarrow 0}\frac{1}{\mu_X(B(x,r))}\int_{B(x,r)}f(x')\,d\mu_X(x')=f(x)\quad\quad\mu_X\text{-a.e.}$$
\end{definition}

\begin{remark}
Many types of spaces satisfy the Lebesgue differentiability property: it is known that $\mX\in\mw$ satisfies the Lebesgue differentiability property when $\mX$ is
\begin{enumerate}
    \item a finite mm-space.
    
    \item an ultrametric space with any Borel probability measure \cite[p.141-169]{federer2014geometric}, \cite[Theorem 9.1]{simmons2012conditional}.
    
    \item a Riemannian manifold with the normalized Riemannian volume measure \cite[p.141-169]{federer2014geometric}, \cite[Theorem 9.1]{simmons2012conditional}.
    
    \item a doubling mm-space \cite[Theorem 1.8]{heinonen2001lectures}.
\end{enumerate}
\end{remark}

\begin{theorem}[stability of generalized cMDS]\label{thm:stability}
For any $\mX=(X,d_X,\mu_X),\mY=(Y,d_Y,\mu_Y)\in\mw$ satisfying the Lebesgue differentiability property, we have 
$$\inf_{\mu\in\Gamma(\mu_X,\mu_Y)}\Vert K_\mX^+-K_\mY^+ \Vert_{L^2(\mu\otimes\mu)}\leq 8\max\{\diamna(X),\diamna(Y)\}\cdot d_{\mathrm{GW},2}(\mX,\mY).$$
\end{theorem}

\begin{remark}[high level overview of the proof of Theorem \ref{thm:stability}]\label{rmk:stbpfovvw}
\begin{figure}
    \centering
    \includegraphics[width=0.6\linewidth]{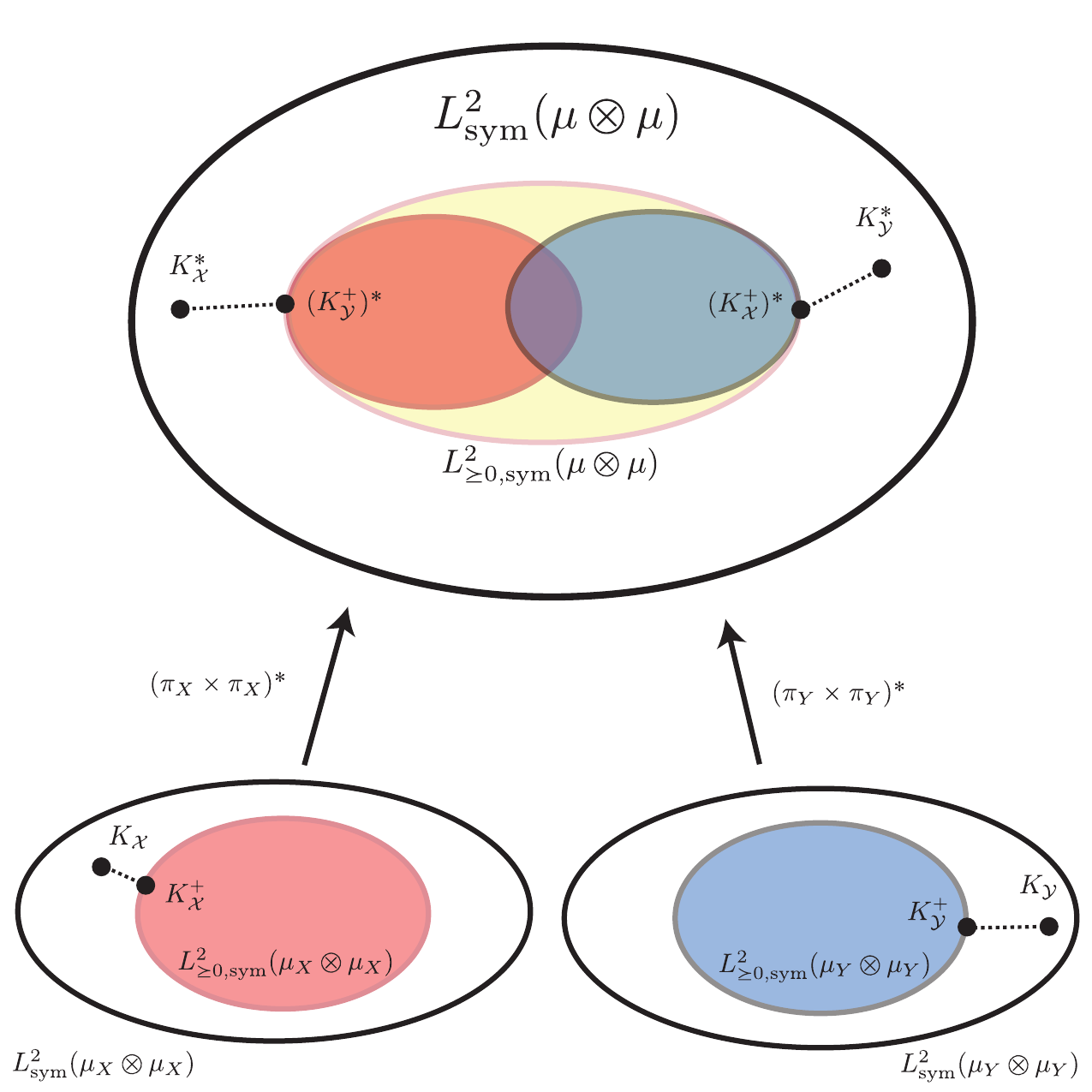}
    \caption{Description of how to upper bound $\Vert K_\mX^+-K_\mY^+\Vert_{L^2(\mu\otimes\mu)}$ by $\Vert K_\mX-K_\mY\Vert_{L^2(\mu\otimes\mu)}$ for a given coupling measure $\mu\in\Gamma(\mu_X,\mu_Y)$. Cf. Lemma \ref{lemma:pullbackkernel}.}
    \label{fig:stability}
\end{figure}

We emphasize that Theorem \ref{thm:stability} does not require the traceability of $\cmdsop_\mX$ and $\cmdsop_\mY$. The proof takes place purely at the level of kernel spaces, in other words: the cMDS embeddings into $\ell^2$ are not necessary. 

\smallskip
Let us overview the proof of Theorem \ref{thm:stability} at a high level; see Figure \ref{fig:stability}.

\smallskip
Recall that $K_\mX^+$ (resp. $K_\mY^+$) is the nearest point of $K_\mX$ (resp. $K_\mY$) onto the closure of the convex set consisting of all symmetric kernels associated to a positive semi-definite operator (cf. \S \ref{sec:sec:sec:altcmds}).

Roughly, there are two major steps in the proof: for a given coupling measure $\mu\in\Gamma(\mu_X,\mu_Y)$, 
\begin{enumerate}
    \item[(1)] showing $\Vert K_\mX^+-K_\mY^+\Vert_{L^2(\mu\otimes\mu)}$ is stable with respect to $\Vert K_\mX-K_\mY\Vert_{L^2(\mu\otimes\mu)}$ (cf.  item (3) of Lemma \ref{lemma:pullbackkernel}), and 
    \item[(2)] showing $\Vert K_\mX-K_\mY\Vert_{L^2(\mu\otimes\mu)}$ is stable with respect to $\Vert d_X-d_X\Vert_{L^2(\mu\otimes\mu)}$ and $d_{\mathrm{GW},2}(\mX,\mY)$ (cf. Lemma \ref{lemma:KXL2bdd}).
\end{enumerate}
The second step is relatively straightforward. So, we concentrate on explaining the idea behind the first step, which can be seen as a generalization of Proposition \ref{prop:finstability}. Note that $K_\mX,K_\mX^+$ and $K_\mY,K_\mY^+$ sit in different spaces ($L^2_{\mathrm{sym}}(\mu_X\otimes\mu_X)$ and $L^2_{\mathrm{sym}}(\mu_Y\otimes\mu_Y)$ respectively). Hence, in order to measure the dissimilarity between them, we apply the pullback maps $(\pi_X\times\pi_X)^*$ and $(\pi_Y\times\pi_Y)^*$ to them so that the resulting kernels $K_\mX^*,(K_\mX^+)^*,K_\mY^*,(K_\mY^+)^*$  all sit in the common space $L^2_{\mathrm{sym}}(\mu\otimes\mu)$. Furthermore, under the condition that $\mu$ is absolutely continuous with respect to $\mu_X\otimes\mu_Y$, we are able to prove that indeed $(K_\mX^+)^*$ (resp. $(K_\mY^+)^*$) is the nearest point of $K_\mX^*$ (resp. $K_\mY^*$) onto the closure of a convex subset $L^2_{\succeq 0,\mathrm{sym}}(\mu\otimes\mu)$. Hence, by the nonexpansive property of the nearest point projection in a Hilbert space (cf. Theorem \ref{thm:projnonexpansive}), we conclude
$$\Vert K_\mX^+-K_\mY^+\Vert_{L^2(\mu\otimes\mu)}=\Vert (K_\mX^+)^*-(K_\mY^+)^*\Vert_{L^2(\mu\otimes\mu)}\leq\Vert K_\mX^*-K_\mY^*\Vert_{L^2(\mu\otimes\mu)}=\Vert K_\mX-K_\mY\Vert_{L^2(\mu\otimes\mu)}.$$

Of course, if we assume that the spaces involved are traceable, that is  $\mX,\mY\in\mwtr$, then we obtain the following corollary since $K_\mX^+=K_{\widehat{\mX}}$ and $K_\mY^+=K_{\widehat{\mY}}$ in that case.
\end{remark}

\begin{corollary}\label{cor:cmdsl2stab}
For any spaces $\mX=(X,d_X,\mu_X)$ and $\mY=(Y,d_Y,\mu_Y)\in\mwtr$ satisfying the Lebesgue differentiability property, we have
$$\inf_{\mu\in\Gamma(\mu_X,\mu_Y)}\Vert K_{\widehat{\mX}}-K_{\widehat{\mY}} \Vert_{L^2(\mu\otimes\mu)}\leq 8\max\{\diamna(X),\diamna(Y)\}\cdot d_{\mathrm{GW},2}(\mX,\mY).$$
\end{corollary}
\begin{proof}
Apply Proposition \ref{prop:twoviewsame} and Theorem \ref{thm:stability}.
\end{proof}

\begin{remark}\label{rem:stab-prob}
Corollary \ref{cor:cmdsl2stab} suggests asking whether one would be able to control the difference between $\widehat{\mX}$ and $\widehat{\mY}$ as subsets of $\ell^2$ by $\inf_{\mu\in\Gamma(\mu_X,\mu_Y)}\Vert K_{\widehat{\mX}}-K_{\widehat{\mY}} \Vert_{L^2(\mu\otimes\mu)}$. Though this question is not easy to answer in general, note that there exist results in the literature  which are  useful  when $\mX$ and $\mY$ are finite spaces (cf. \S \ref{sec:sec:finstab}). For example, consider the case when $\mX_1=(X,d_1,\mu_X)$ and $\mX_2=(X,d_2,\mu_X)$ (i.e., two mm-spaces with the same finite underlying set $X$, the same probability measure $\mu_X$, but different metrics). Then, one can directly apply \cite[Theorem 1]{arias2020perturbation} in order to bound the difference between $\widehat{\mX_1}$ and $\widehat{\mX_2}$ (measured via Procrustes-like ideas) by $f(\Vert K_{\widehat{\mX_1}}-K_{\widehat{\mX_2}} \Vert_{L^2(\mu_X\otimes\mu_X)})$ where $f$ is a certain specific decreasing function such that $f(t)\rightarrow 0$ as $t\rightarrow 0^+$.

In terms of a possible generalization, one goal would attempt to establish  the following inequality:
$$\inf_{\substack{\phi:\ell^2\rightarrow\ell^2\\\text{ isometry}}}d_{\mathrm{W},2}^{\ell^2}\big((\Phi_\mX)_*\mu_X,(\phi\circ\Phi_\mY)_*\mu_Y\big)\leq f\left(\inf_{\mu\in\Gamma(\mu_X,\mu_Y)}\Vert K_{\widehat{\mX}}-K_{\widehat{\mY}} \Vert_{L^2(\mu\otimes\mu)}\right)$$
where $d_{\mathrm{W},2}^{\ell^2}$ is the $2$-Wasserstein distance on the space of probability measures on $\ell^2$ (see \cite[\S 7.1]{villani2003topics} for the precise definition). Recall that $\widehat{\mX}$ (resp. $\widehat{\mY}$) is the image of the cMDS embedding $\phi_\mX$ (resp. $\phi_\mY$). This could be regarded as an extension of \cite[Theorem 1]{arias2020perturbation} to  general mm-spaces.
\end{remark}

\begin{remark}[stability of cMDS into $\R^k$] 
Note that our proof of stability of the coordinate free version of the cMDS procedure bypasses the cMDS embedding $\Phi_{\mX\!\!,k}$  and it crucially relies on a projection argument, which itself hinges on the convexity of the set of all symmetric kernels on a given set associated to positive semi-definite operators. If we were to specify a maximal embedding dimension parameter $k$ as in Algorithm \ref{alg:gencMDSRk} and Definition \ref{def:cmdsrk}, we would have to contend with the collection of all symmetric kernels associated to positive semi-definite operators with rank at most $k$, which is \emph{non-convex}. Alternatively, one could attempt to establish stability at the level of the embeddings of $\mX_k$ and $\mY_k$ into  $\R^k$, which would require carefully expressing the stability of eigenvalues and eigenfunctions of the respective cMDS operators; see \cite[Proposition 5.4]{cmds-stepanov} for a restricted version of this.
\end{remark}

\subsubsection{The proof of Theorem \ref{thm:stability}.}\label{sec:sec:sec:pfthmstab}
We are going to follow the strategy described in Remark \ref{rmk:stbpfovvw}. Recall that there are two major steps in the proof of Theorem \ref{thm:stability}. The first step, which is more complicated than the second step, is showing that $\Vert K_\mX^+-K_\mY^+\Vert_{L^2(\mu\otimes\mu)}$ is stable with respect to $\Vert K_\mX-K_\mY\Vert_{L^2(\mu\otimes\mu)}$. The second step, which is rather straightforward, is showing that $\Vert K_\mX-K_\mY\Vert_{L^2(\mu\otimes\mu)}$ is stable with respect to $\Vert d_X-d_X\Vert_{L^2(\mu\otimes\mu)}$ and $d_{\mathrm{GW},2}(\mX,\mY)$. Below, the first step will be dealt with by item (3) of Lemma \ref{lemma:pullbackkernel}, and the second step will be dealt with by Lemma \ref{lemma:KXL2bdd}.

Suppose $\mX,\mY\in\mw$ and $\mu\in\Gamma(\mu_X,\mu_Y)$ are given. Then, for any $K\in L^2(\mu_X\otimes\mu_X)$, we have the \emph{pullback} kernel $$(\pi_X\times\pi_X)^*(K):=K\circ(\pi_X\times\pi_X)\in L^2(\mu\otimes\mu)$$ where $\pi_X:X\times Y\rightarrow X$ is the canonical projection. It is clear that the map $$(\pi_X\times\pi_X)^*:L^2(\mu_X\otimes\mu_X)\hookrightarrow L^2(\mu\otimes\mu)$$ thus defined is an isometric embedding. Also, we define $(\pi_Y\times\pi_Y)^*:L^2(\mu_Y\otimes\mu_Y)\hookrightarrow L^2(\mu\otimes\mu)$ in a similar manner where $\pi_Y:X\times Y\rightarrow Y$ is the canonical projection.

Recall that, in Proposition \ref{prop:finstability}, we established the stability of finite cMDS where two metric spaces of interest have the same underlying set but different metrics. The main idea behind the proof is using the convexity of the set of positive semi-definite matrices and the non-expansiveness of the nearest point projection to a convex set. Item (3) of the following lemma is a generalization of Proposition \ref{prop:finstability} to the setting of general metric measure spaces. The main obstacle of this generalization is the fact that the kernels of interest $K_\mX$, $K_\mX^+$ and $K_\mY$, $K_\mY^+$ belong to different function spaces. In order to overcome this, we employ two technical concepts. One is that of a coupling measure between $\mu_X$ and $\mu_Y$ which specifies a relationship between $X$ and $Y$, and the other is the aforementioned notion of pullback kernel. With these, one is able to apply a  procedure similar to the one employed in the proof of Proposition \ref{prop:finstability}.

\begin{lemma}\label{lemma:pullbackkernel}
Suppose $\mX,\mY\in\mw$ satisfy the Lebesgue differentiability property. If a coupling measure $\mu\in\Gamma(\mu_X,\mu_Y)$ satisfies $\mu\ll\mu_X\otimes\mu_Y$ (i.e., $\mu$ is absolutely continuous with respect to  $\mu_X\otimes\mu_Y$), then the following holds, where $f:=\frac{d\mu}{d(\mu_X\otimes\mu_Y)}$ is the density of $\mu$ relative to $\mu_X\otimes \mu_Y$: 

\begin{enumerate}
    \item For $\mu_X$-almost everywhere $x\in X$, $f(x,\cdot)\,d\mu_Y$ is a probability measure on $Y$.

    \item $(\pi_X\times\pi_X)^*(K_\mX^+)$ and $(\pi_Y\times\pi_Y)^*(K_\mY^+)$ are the nearest points of $(\pi_X\times\pi_X)^*(K_\mX)$ and $(\pi_Y\times\pi_Y)^*(K_\mY)$, respectively, to the closure of $L^2_{\succeq 0,\mathrm{sym}}(\mu\otimes\mu)$.
    
    \item $\Vert K_\mX^+-K_\mY^+\Vert_{L^2(\mu\otimes\mu)}\leq\Vert K_\mX-K_\mY\Vert_{L^2(\mu\otimes\mu)}$. 
\end{enumerate}
\end{lemma}

Also, we need to recall the following classical result.

\begin{theorem}[{\cite[Theorem 3.6]{GR85}}]\label{thm:projnonexpansive}
The (unique) nearest point projection onto a closed convex subset of a Hilbert space is nonexpansive.
\end{theorem}

\begin{lemma}\label{lemma:KXL2bdd}
Suppose $\mX=(X,d_X,\mu_X),\mY=(Y,d_Y,\mu_Y)\in\mw$ and a coupling measure $\mu\in\Gamma(\mu_X,\mu_Y)$ are given. Then,
$$ \Vert K_\mX-K_\mY\Vert_{L^2(\mu\otimes\mu)}
    \quad\leq \max\{\diamna(X),\diamna(Y)\}\cdot \big(3\,\Vert d_X-d_Y\Vert_{L^2(\mu\otimes\mu)}+2 \,d_{\mathrm{GW},2}(\mX,\mY)\big).$$
\end{lemma}

Now we are ready to prove the stability theorem.

\begin{proof}[Proof of Theorem \ref{thm:stability}]
Fix arbitrary $\varepsilon>d_{\mathrm{GW},2}(\mX,\mY)$. Then, there is a coupling measure $\mu_\varepsilon\in\Gamma(\mu_X,\mu_Y)$ such that
$\Vert d_X-d_Y\Vert_{L^2(\mu_\varepsilon\otimes\mu_\varepsilon)}<2\varepsilon.$
Moreover, by \cite[Proposition 2.8]{eckstein2020robust}, we can assume that $\mu_\varepsilon\ll\mu_X\otimes\mu_Y$. Therefore,
\begin{align*}
    \inf_{\mu\in\Gamma(\mu_X,\mu_Y)}\Vert K_\mX^+-K_\mY^+ \Vert_{L^2(\mu\otimes\mu)}&\leq\Vert K_\mX^+-K_\mY^+\Vert_{L^2(\mu_\varepsilon\otimes\mu_\varepsilon)}\\
    &\leq\Vert K_\mX-K_\mY\Vert_{L^2(\mu_\varepsilon\otimes\mu_\varepsilon)}\,(\because\text{ item (3) of Lemma \ref{lemma:pullbackkernel}})\\
    &<\max\{\diamna(X),\diamna(Y)\}(6\varepsilon+2d_{\mathrm{GW},2}(\mX,\mY))
\end{align*}
where the last inequality holds by Lemma \ref{lemma:KXL2bdd} and the assumption on $\Vert d_X-d_Y\Vert_{L^2(\mu_\varepsilon\otimes\mu_\varepsilon)}$. Since $\varepsilon>d_{\mathrm{GW},2}(\mX,\mY)$ is arbitrary, this completes the proof.
\end{proof}

The proofs of Lemma \ref{lemma:pullbackkernel} and Lemma \ref{lemma:KXL2bdd} are relegated to Appendix \S\ref{sec:otherproofs}.

\subsection{Consistency}\label{sec:cons}

A basic consequence of our stability bounds is that cMDS is \emph{consistent} (as expressed by Corollary \ref{cor:consistency} below) for traceable mm-spaces satisfying the Lebesgue differentiability property.

It is useful to define the following notion of convergence of $4$-tuples $\mX:=(X,\Sigma_X,P_X,\mu_X)$ where $(X,\Sigma_X,\mu_X)$ is a measure space endowed with a $\sigma$-algebra $\Sigma_X$, a probability measure $\mu_X$, and $P_X\in L^2(\mu_X\otimes\mu_X)$ is a kernel on $X$ whose associated integral operator is positive semi-definite (see Appendix \S\ref{sec:sec:functanal}). We say that a sequence $(\mX_n = (X_n,\Sigma_n,P_{X_n},\mu_{X_n}))_n$ of such $4$-tuples \emph{converges weakly} to $\mX = (X,\Sigma_X,P_X,\mu_X)$ if for each $n$ there exist a couplings $\mu_n\in\Gamma(\mu_X,\mu_{X_n})$ such that 
$\left\|P_X-P_{X_n}\right\|_{L^2(\mu_n\otimes \mu_n)}\rightarrow 0\,\,\mbox{as $n\rightarrow\infty$}.$ We denote this by $P_{\mX_n}\stackrel{n}{\Longrightarrow}P_\mX.$

Together with Theorem \ref{thm:stability}, Theorem 5.1 item (e) of \cite{memoli2011gromov} imply:

\begin{corollary}[consistency]\label{cor:consistency}
Let $\mX=(X,d_X,\mu_X)\in\mwtr$ satisfies the Lebesgue differentiability property and let $X_n = X_n(\omega) =  \{\mathbf{x}_i(\omega),\,i=1,\ldots,n\}$ be any collection of i.i.d. random variables chosen from $X$ w.r.t. the  measure $\mu_X$ for $\omega\in\Omega$, for some probability space $\Omega$. Let $\mX_n=\mX_n(\omega)=(X_n,d_X|_{X_n\times X_n},\mu_{X_n})$ be the random metric measure where $\mu_{X_n}$ is the uniform measure on $X_n$. Then, for $\mu_X$-almost every $\omega\in \Omega$:
$$K_{\widehat{\mX}_n(\omega)}\stackrel{n}{\Longrightarrow} K_{\widehat{\mX}}.$$
\end{corollary}

This corollary provides a partial answer to \cite[Conjecture 6.1]{adams2020multidimensional}.


\section{Discussion}\label{sec:discussion}
In this section we first compare the results in our paper with those in  \cite{kroshnin2022infinite} and then also formulate a number of  open questions.
\subsection{Comparison to Kroshnin et al. \cite{kroshnin2022infinite}}
As  already mentioned in the introduction, a recent paper \cite{kroshnin2022infinite} by Kroshnin, Stepanov, and Trevisan also studies a generalization of cMDS for metric measure spaces. Our paper and \cite{kroshnin2022infinite} were developed simultaneously. To ensure that future readers can fully benefit from reading both papers, let us clarify the relationships between the principal findings of each paper. \\
\\
\medskip
The following results have no counterparts in \cite{kroshnin2022infinite}:
\begin{itemize}
    \item Examples \ref{ex:PaleyL2natural} and \ref{ex:reg-polys} which show that the $L^2$-metric distortion is a natural choice for studying the cMDS embedding. Also, the relationship between the $L^2$-distortion and the negative trace of the cMDS operator that is described in Item (2) of Proposition \ref{prop:generrorbdd}.
\item[]
    \item  The  coordinate-free interpretation of generalized cMDS  given by $K_\mX^+$ in \S\ref{sec:sec:sec:altcmds}, which does not require the traceability of $K_\mX$. Also, Proposition \ref{prop:twoviewsame} which, under the traceability assumption, proves that $K_\mX^+=K_{\widehat{X}}$.
\item[]
    \item The explicit construction of a (non-compact)  metric measure space with non-traceable cMDS operator based on Paley graphs that is given in \S\ref{sec:non-trace-class}.
\item[]
    \item The relationship between the smallest positive eigenvalue of the cMDS operator of a Euclidean mm-space with  a certain notion of ``thickness" of the input mm-space (cf. \S\ref{sec:thickness}).
\item[]
    \item The traceability of the cMDS operator when the (Minkowski) dimension of the underlying mm-space is less than $2$ (cf.  \S\ref{sec:sec:metgraphs} where we establish Theorem \ref{thm:metenttrace}). In particular, this implies that every compact metric graph a is traceable mm-space. This result is achieved with the aid of \cite{gonzalez1993metric} which connects the traceability of the integral operator with the growth rate of the covering number for the domain of the operator.\footnote{ There might be a way to strengthen our result by employing a more careful analysis on the arguments used in \cite{gonzalez1993metric}.} 
\item[]
    \item The construction of a family of  metrics on $2$-sphere $\Sp^2$ which induce  traceable cMDS operators (cf. \S\ref{sec:sec:Sd-1diffmet}). These metrics are more general than the geodesic distance.
\end{itemize}
The following results have similar counterparts in \cite{kroshnin2022infinite}. Some of these are stronger and others are weaker than ours (as we discuss next):
\begin{itemize}
    \item \textbf{Generalized cMDS for homogeneous spaces.} Combining Lemma \ref{lemma:simplereigenproblem} and Corollary \ref{cor:twohomoeigenproblem}, one achieves the same result as \cite[Proposition 3.4]{kroshnin2022infinite}.
\item[]
    \item \textbf{cMDS is non-contracting.} Item (1) of Proposition \ref{prop:generrorbdd} and \cite[Corollary 4.6]{kroshnin2022infinite} make the same claim but the assumptions differ. Whereas Item (1) of of Proposition \ref{prop:generrorbdd} only assumes $\mX$ to be traceable, \cite[Corollary 4.6]{kroshnin2022infinite} makes the  additional assumption that almost every point in $\mX$ is a Lebesgue point.   
\item[]
    \item \textbf{Product spaces.} Both of \S\ref{sec:sec:cmdsproduct} and \cite[Section 6.2]{kroshnin2022infinite} study the generalized cMDS for product spaces. In particular, Corollary \ref{cor:propsofprod} is a generalization of \cite[Proposition 6.2]{kroshnin2022infinite} and Corollary \ref{cor:NtoruscMDS} is a generalization of \cite[Proposition 6.3]{kroshnin2022infinite} in the sense that our version can be applied to the product of $N$ spaces, for any integer $N\geq 2$ (not only for the case $N=2$). 
\item[]
    \item \textbf{The cMDS operator for spheres.} Both \S\ref{sec:Spd-1} in our paper and \cite[Section 6.1]{kroshnin2022infinite} study the spectrum of $\cmdsop_{\Sp^{d-1}}$, the cMDS operator for higher dimensional spheres, and prove that $\cmdsop_{\Sp^{d-1}}$ is traceable. In particular, Theorem \ref{thm:cMDSSd-1dist} agrees with \cite[Proposition 6.1]{kroshnin2022infinite}. However, we give a few more interesting auxiliary results such as the closed form expression of positive eigenvalues (cf. Corollary \ref{cor:labdandodd}) and the construction of an infinity family of non-euclidean metrics on $2$-spheres $\Sp^2$ all of 
 which induce traceable cMDS operators (cf. \S\ref{sec:sec:Sd-1diffmet}). 
\item[]
    \item \textbf{Stability and consistency.} Both of \S\ref{sec:stability} and \cite[Section 5.1]{kroshnin2022infinite} study the stability of cMDS with respect to the Gromov-Wasserstein distance. Although these sections are similar at the philosophical level, there are technical differences as follows: 
    \begin{enumerate}
        \item  Theorem \ref{thm:stability} and \cite[Lemma 5.1]{kroshnin2022infinite} are similar in the sense that Theorem \ref{thm:stability} expresses the stability of the cMDS kernel of the resulting spaces with respect to the $d_{\mathrm{GW},2}$ and \cite[Lemma 5.1]{kroshnin2022infinite} establishes the stability of the cMDS operator with respect to $d_{\mathrm{GW},4}$ (and applicable beyond the cases of bounded spaces $\mX$ as long as $\diamna_4(\mX)<\infty$; cf. equation (\ref{eq:diam-p})). 
        \item[]
        \item Also, with  additional assumptions on the spectral gap of the cMDS operator, \cite[Theorem 5.4, Corollary 5.6]{kroshnin2022infinite} proves the consistency of the cMDS embedding into a finite dimensional Euclidean spaces. Our result, Corollary \ref{cor:consistency},  proves consistency of cMDS at the kernel level in an $L^2$ sense, while \cite[Theorem 5.4, Corollary 5.6]{kroshnin2022infinite} proves consistency of cMDS embeddings (instead of kernels) into finite dimensional Euclidean space with respect to  $d_{\mathrm{GW},2}$.
    \end{enumerate}
\end{itemize}
\subsection{Some open questions}
We collect a few open questions about the  generalized cMDS:
\begin{itemize}
    \item In \S\ref{sec:thickness}, we prove that, when $\mX$ is a compact subset of Euclidean space, the smallest (positive) eigenvalue of the cMDS operator $\cmdsop_\mX$ is upper bounded by the square of $\mathrm{Th}(\mX)$, a certain notion of thickness of $\mX$. Next, we want to ask similar questions in the same spirit. Can we prove analogous results for non-Euclidean $\mX$? Can we bound positive eigenvalues other than the smallest one in a similar manner? What kind of geometric properties of $\mX$ other than  thickness can be used in order to control the spectrum of $\cmdsop_\mX$? 
\item[]
    \item In this paper, we argue that, in order to understand the strength and limits of the cMDS procedure, traceability is a natural condition on the cMDS operator which enables the study of the embedding into $\ell^2$. For this reason,  it is important to characterize metric measure spaces inducing traceable cMDS operators. We provide many nontrivial examples of  traceable spaces in \S\ref{sec:ntrvalexmple} and \S\ref{sec:Spd-1} and  in Section \S\ref{sec:non-trace-class}  construct  a non-compact metric measure space with 
 non-traceable cMDS operator. Moreover, recall the following questions: 
    \cmtr* 
    \cmtrman*

\begin{framed}
\noindent\textbf{Note}: After this paper was accepted, we learned that this question was recently answered to the negative by Ma and Stepanov; see  \cite{ma2024eigenvalues} where the authors prove that for $n\geq 1$ the spaces $\mathbb{RP}^{2n+1}$ induce non-traceable cMDS operators.  These recent developments motivate
the following loosely defined challenge:

\cmtrmannew*

\end{framed}
    
    \item As explained in Remark \ref{rmk:curvdist}, we think it is intriguing to establish connections between the $L^2$-distortion $\dis(\mX)$ of 
a compact Riemannian manifold $\mX$ (which is equal to $\sqrt{2\,\trng(\mX)}$ by item (2) of Proposition \ref{prop:generrorbdd}) to the sectional curvature of $\mX$. This is because both concepts quantify the deviation from flatness of the input space, with the former being a global measure of flatness and the latter being a local measure. More precisely, taking into account that $\trng(\mX)=0$ implies that the sectional curvature of $\mX$ is zero everywhere, we pose the following question:
    \curvdist*
    \item In \S\ref{sec:stability}, we establish the stability of the cMDS procedure at the kernel level in the following way: $$\inf_{\mu\in\Gamma(\mu_X,\mu_Y)}\Vert K_\mX^+-K_\mY^+ \Vert_{L^2(\mu\otimes\mu)}\leq 8\max\{\diamna(X),\diamna(Y)\}\cdot d_{\mathrm{GW},2}(\mX,\mY).$$
    In particular, as a result, there we establish the statistical consistency of the cMDS procedure.  Now, can we prove similar results but for the output spaces $\widehat{\mX}$, $\widehat{\mY}$ themselves (in the sense of Remark \ref{rem:stab-prob}) and not just for their associated kernels $K_\mX^+$, $K_\mY^+$?
\item[]
    \item As Corollary \ref{cor:propsofS1} and Theorem \ref{thm:cMDSSd-1dist} imply, $\widehat{\Sp}^{d-1}$ is homeomorphic to $\Sp^{d-1}$. Motivated by this observation, can we prove that $\widehat{M}$ is homeomorphic  (or even homotopy equivalent) to $M$ when $M$ is an arbitrary  compact manifold? In other words, does the cMDS procedure preserve the global topology of the input space?
\end{itemize}


\section*{Acknowledgments}
We are grateful to Henry Adams, Boris Mityagin, and Ery Arias-Castro for helpful conversations related to this work. We also thank Eugene Stepanov for pointing out to imprecisions in the proofs of Propositions \ref{prop:traceclasscmds} and \ref{prop:generrorbdd}.

\section*{Funding}
This work was supported by NSF grants DMS 1547357, CCF-1526513, IIS-1422400, and  CCF-1740761.


\bibliographystyle{acm}

\appendix

\section{Relegated proofs}\label{sec:otherproofs}
\subsection{Proof of Theorem \ref{thm:optimalthmgen}}
Our proof strategy is structurally similar to the proof of Theorem \ref{thm:optimalthm} (the optimality of the cMDS for finite metric spaces) given in \cite[Theorem 14.4.2]{mardiamultivariate}. Also, although we developed our proof separately, it shares the common strategy used in the proof of \cite[Theorem 6.4.3]{kassab2019multidimensional} since Kassab's proof also motivated by the proof of \cite[Theorem 14.4.2]{mardiamultivariate}. However,   Kassab's proof does not contemplate that some infinite sums therein considered may fail to converge (as they would if applied to the pathological mm-space constructed in \S\ref{sec:non-trace-class})  --  a problem which we circumvent by imposing the traceability condition \S\ref{sec:sec:sec:trblmmspace}.\medskip

We need the following preliminary results.

\begin{adefinition}[doubly (sub)stochastic matrix]
A square matrix $A=(a_{ij})_{i,j=1}^n$ of size $n$ is called \textbf{doubly stochastic} (resp. \textbf{doubly substochastic}) if it satisfies:
\begin{enumerate}
    \item $a_{ij}\geq 0$ for any $i,j\in\{1,\dots,n\}$.
    \item $\sum_{j=1}^n a_{ij}= 1$ for any $i\in\{1,\dots,n\}$ (resp. $\sum_{j=1}^n a_{ij}\leq 1$ for any $i\in\{1,\dots,n\}$).
    \item $\sum_{i=1}^n a_{ij}= 1$ for any $j\in\{1,\dots,n\}$ (resp. $\sum_{i=1}^n a_{ij}\leq 1$ for any $j\in\{1,\dots,n\}$).
\end{enumerate}
\end{adefinition}

\begin{atheorem}[{rearrangement inequality \cite[Theorem 368]{hardy1967inequalities}}]\label{thm:rearrangement}
Suppose we have two sequences of real numbers $x_1\leq\dots\leq x_n$ and $y_1\leq\dots\leq y_n$ for some positive integer $n$. Then,
$$x_ny_1+x_{n-1}y_2+\cdots+x_1y_n\leq x_{\sigma(1)}y_1+x_{\sigma(2)}y_2+\cdots+x_{\sigma(n)}y_n\leq x_1y_1+x_2y_2+\cdots+x_ny_n$$
for any permutation $\sigma$ on $[n]=\{1,\dots,n\}$.
\end{atheorem}

\begin{atheorem}[{Birkhoff--von Neumann theorem \cite[Theorem 1.7.1]{brualdi2006combinatorial}}]\label{thm:Birkhoff}
Any doubly stochastic matrix is a convex combination of permutation matrices.
\end{atheorem}

\begin{atheorem}[{\cite[Theorem 3.2.6]{horn1994topics}}]\label{thm:substochasticuppbdd}
A square matrix $A=(a_{ij})_{i,j=1}^n$ of size $n$ is doubly substochastic matrix if and only if there is a doubly stochastic matrix $B=(b_{ij})_{i,j=1}^n$ of size $n$ such that $0\leq a_{ij}\leq b_{ij}$ for any $i,j=1,\dots,n$.
\end{atheorem}

\begin{alemma}\label{lemma:TRdiagineq}
Let $\mathcal{H}$ be a separable Hilbert space, $\mathfrak{A}\in\mathcal{L}(\mathcal{H})$ be Hilbert-Schmidt and self-adjoint, and $\mathfrak{B}\in\mathcal{L}(\mathcal{H})$ be Hilbert-Schmidt, self-adjoint, and positive semi-definite. Then, we have
$$\tr(\mathfrak{A}\mathfrak{B})\leq\sum_{i=1}^{M}\lambda_i\lambda_i'$$
where $M:=\min\{\mathrm{rank}(\mathfrak{B}),\pr(\mathfrak{A})\}$, $\lambda_1\geq\lambda_2\geq\cdots>0$ are the positive eigenvalues of $\mathfrak{A}$, and $\lambda_1'\geq\lambda_2'\geq\cdots>0$ are the positive eigenvalues of $\mathfrak{B}$.
\end{alemma}
\begin{proof}
Let $\{\lambda_i\}_{i=1}^{\pr(\mathfrak{A})}$ be the positive eigenvalues with corresponding orthonormal eigenfunctions $\{\phi_i\}_{i=1}^{\pr(\mathfrak{A})}$ of $\mathfrak{A}$, $\{\zeta_i\}_{i=1}^{\nr(\mathfrak{A})}$ be the negative eigenvalues with corresponding orthonormal eigenfunctions $\{\psi_i\}_{i=1}^{\nr(\mathfrak{A})}$ of $\mathfrak{A}$, and $\{\lambda_i'\}_{i=1}^{\mathrm{rank}(\mathfrak{B})}$ be the positive eigenvalues with corresponding orthonormal eigenfunctions $\{\phi_i'\}_{i=1}^{\mathrm{rank}(\mathfrak{B})}$ of $\mathfrak{B}$. Then,
\begin{align*}
    \tr(\mathfrak{A}\mathfrak{B})&=\sum_{j=1}^{\mathrm{rank}(\mathfrak{B})}\langle\phi_j',\mathfrak{A}\mathfrak{B}\phi_j'\rangle=\sum_{j=1}^{\mathrm{rank}(\mathfrak{B})}\lambda_j'\langle\phi_j',\mathfrak{A}\phi_j'\rangle\\
    &=\sum_{j=1}^{\mathrm{rank}(\mathfrak{B})}\lambda_j'\bigg\langle\phi_j',\sum_{i=1}^{\pr(\mathfrak{A})}\lambda_i\langle\phi_i,\phi_j'\rangle\phi_i+\sum_{i=1}^{\nr(\mathfrak{A})}\zeta_i\langle\psi_i,\phi_j'\rangle\psi_i\bigg\rangle\\
    &=\sum_{j=1}^{\mathrm{rank}(\mathfrak{B})}\sum_{i=1}^{\pr(\mathfrak{A})}\lambda_j'\lambda_i\langle\phi_j',\phi_i\rangle^2+\sum_{j=1}^{\mathrm{rank}(\mathfrak{B})}\sum_{i=1}^{\nr(\mathfrak{A})}\lambda_j'\zeta_i\langle\phi_j',\psi_i\rangle^2\\
    &\leq\sum_{j=1}^{\mathrm{rank}(\mathfrak{B})}\sum_{i=1}^{\pr(\mathfrak{A})}\lambda_j'\lambda_i\langle\phi_j',\phi_i\rangle^2
\end{align*}

First, let's consider the case when $M=\mathrm{rank}(\mathfrak{B})$. Observe that $\sum_{j=1}^{\mathrm{rank}(\mathfrak{B})} \langle\phi_j',\phi_i\rangle^2\leq 1$ for each $i=1,\dots,\pr(\mathfrak{A})$, and $\sum_{i=1}^{\pr(\mathfrak{A})}\langle\phi_j',\phi_i\rangle^2\leq 1$ for each $j=1,\dots,\mathrm{rank}(\mathfrak{B})$ by the orthonormal property. Now, we carry out a case-by-case analysis. 

    \smallskip
    \noindent (1) If $\pr(\mathfrak{A})<\infty$: Observe that, if we define $A=\big(a_{ji}\big)_{j,i=1}^{\pr(\mathfrak{A})}$ where $$a_{ji}=\begin{cases}\langle\phi_j',\phi_i\rangle^2&\text{if }j=1,\dots,\mathrm{rank}(\mathfrak{B})\\ 0&\text{if }j=\mathrm{rank}(\mathfrak{B})+1,\dots,\pr(\mathfrak{A}),\end{cases}$$ then $A$ is a substochastic matrix. Then, by Theorem \ref{thm:substochasticuppbdd}, $A$ is (entrywisely) upper bounded by some doubly stochastic matrix $B=\big(b_{ji}\big)_{j,i=1}^{\pr(\mathfrak{A})}$. Hence,
   $$
       \sum_{j=1}^{\mathrm{rank}(\mathfrak{B})}\sum_{i=1}^{\pr(\mathfrak{A})}\lambda_j'\lambda_i\langle\phi_j',\phi_i\rangle^2=\sum_{j=1}^{\mathrm{rank}(\mathfrak{B})} \lambda_j'\left(\sum_{i=1}^{\pr(\mathfrak{A})} a_{ji}\lambda_i\right)
        \leq\sum_{j=1}^{\mathrm{rank}(\mathfrak{B})} \lambda_j'\left(\sum_{i=1}^{\pr(\mathfrak{A})} b_{ji}\lambda_i\right)\leq\sum_{i=1}^{\mathrm{rank}(\mathfrak{B})}\lambda_i\lambda_i'
 $$
    as we wanted, where the last inequality holds by Theorem \ref{thm:rearrangement} and Theorem \ref{thm:Birkhoff}.
    
\smallskip
    \noindent (2) When $\pr(\mathfrak{A})=\infty$, there are two sub-cases:
    \begin{enumerate}
        \item If $\mathrm{rank}(\mathfrak{B})<\infty$: Choose arbitrary positive integer $l\geq k$. Then, by the similar argument of the previous case (1), we have
       $$\sum_{j=1}^{\mathrm{rank}(\mathfrak{B})}\sum_{i=1}^l\lambda_j'\lambda_i\langle\phi_j',\phi_i\rangle^2\leq\sum_{i=1}^{\mathrm{rank}(\mathfrak{B})}\lambda_i\lambda_i'.$$
        Since $l$ is arbitrary,        $\sum_{j=1}^{\mathrm{rank}(\mathfrak{B})}\sum_{i=1}^\infty\lambda_j'\lambda_i\langle\phi_j',\phi_i\rangle^2\leq\sum_{i=1}^{\mathrm{rank}(\mathfrak{B})}\lambda_i\lambda_i'$
        as we wanted.
        
        \item If $\mathrm{rank}(\mathfrak{B})=\infty$: Choose arbitrary positive integers $l\leq l'$. Then, by the similar argument of the previous case (1), we have
       $$\sum_{j=1}^l\sum_{i=1}^{l'}\lambda_j'\lambda_i\langle\phi_j',\phi_i\rangle^2\leq\sum_{i=1}^l\lambda_i\lambda_i'.$$
        Since $l,l'$ are arbitrary, we have
    $\sum_{j=1}^\infty\sum_{i=1}^\infty\lambda_j'\lambda_i\langle\phi_j',\phi_i\rangle^2\leq\sum_{i=1}^\infty\lambda_i\lambda_i'$
        as we wanted.
    \end{enumerate}
The case when $M=\pr(\mathfrak{A})$ is proved in a similar manner, which we omit.
\end{proof}

\begin{proof}[Proof of Theorem \ref{thm:optimalthmgen}]
Let's prove the first claim. Let $\{\lambda_i\}_{i=1}^{\pr(\cmdsop_\mX)}$ be the positive eigenvalues with corresponding orthonormal eigenfunctions $\{\phi_i\}_{i=1}^{\pr(\cmdsop_\mX)}$ of $\cmdsop_\mX$, $\{\zeta_i\}_{i=1}^{\nr(\cmdsop_\mX)}$ be the negative eigenvalues with corresponding orthonormal eigenfunctions $\{\psi_i\}_{i=1}^{\nr(\cmdsop_\mX)}$ of $\cmdsop_\mX$, and $\{\lambda_i'\}_{i=1}^{\mathrm{rank}(\cmdsop)}$ be the positive eigenvalues with corresponding orthonormal eigenfunctions $\{\phi_i'\}_{i=1}^{\mathrm{rank}(\cmdsop)}$ of $\cmdsop$. Note that $\mathrm{rank}(\cmdsop)\leq k\leq\pr(\cmdsop_\mX)$. Then,

\begin{align*}
    \Vert K_\mX-K \Vert_{L^2(\mu_X\otimes\mu_X)}^2&=\Vert \cmdsop_\mX-\cmdsop\Vert_2^2=\tr\big((\cmdsop_\mX-\cmdsop)^*(\cmdsop_\mX-\cmdsop\big))\\
    &=\tr\big((\cmdsop_\mX-\cmdsop)^2\big)=\tr(\cmdsop_\mX^2)+\tr(\cmdsop^2)-2\,\tr(\cmdsop_\mX\cmdsop)\quad(\because\text{item (4) of Lemma \ref{lemma:propstrHB}})\\
    &\geq \sum_{i=1}^{\pr(\cmdsop_\mX)}\lambda_i^2+\sum_{i=1}^{\nr(\cmdsop_\mX)}\zeta_i^2+\sum_{j=1}^{\mathrm{rank}(\cmdsop)}(\lambda_j')^2-2\sum_{i=1}^{\mathrm{rank}(\cmdsop)}\lambda_i\lambda_i'\quad(\because\text{Lemma \ref{lemma:TRdiagineq}})\\
    &=\sum_{i=1}^{\mathrm{rank}(\cmdsop)}(\lambda_i-\lambda_i')^2+\sum_{i=\mathrm{rank}(\cmdsop)+1}^{\pr(\cmdsop_\mX)}\lambda_i^2+\sum_{i=1}^{\nr(\cmdsop_\mX)}\zeta_i^2\\
    &\geq \sum_{i=\mathrm{rank}(\cmdsop)+1}^{\pr(\cmdsop_\mX)}\lambda_i^2+\sum_{i=1}^{\nr(\cmdsop_\mX)}\zeta_i^2\geq\Vert K_\mX-K_{\widehat{\mX}_k}\Vert_{L^2(\mu_X\otimes\mu_X)}^2.
\end{align*}

So, this concludes the first claim. The second claim is proved by a similar argument.
\end{proof}

\subsection{Proof of Lemma \ref{lemma:etandeven0}}
\begin{proof}[Proof of Lemma \ref{lemma:etandeven0}]
Recall from equation (\ref{eq:etand}) that $\eta_{n,d}=\frac{\vert \Sp^{d-2}\vert}{\vert\Sp^{d-1}\vert}\int_{-1}^1 P_{n,d}(t)\arccos(t)(1-t^2)^{\frac{d-3}{2}}\,dt$. Also, note that $P_{n,d}(-t)=P_{n,d}(t)$ for any $t\in[-1,1]$ by item (4) of Lemma \ref{lemma:Legdrprop}. Hence,
\begin{align*}
    \int_{-1}^1 P_{n,d}(t)\arccos(t)(1-t^2)^{\frac{d-3}{2}}\,dt&=\int_{-1}^1 P_{n,d}(-u)\arccos(-u)(1-u^2)^{\frac{d-3}{2}}\,du\,(\text{by substituting }u=-t)\\
    &=\int_{-1}^1 P_{n,d}(u)(\pi-\arccos(u))(1-u^2)^{\frac{d-3}{2}}\,du\\
    &=-\int_{-1}^1 P_{n,d}(u)\arccos(u)(1-u^2)^{\frac{d-3}{2}}\,du\,(\because\text{ item (6) of Lemma \ref{lemma:Legdrprop}}).
\end{align*}
Hence, $\int_{-1}^1 P_{n,d}(t)\arccos(t)(1-t^2)^{\frac{d-3}{2}}\,dt=0$ so that $\eta_{n,d}=0$.
\end{proof}

\subsection{Proof of Lemma \ref{lemma:tylrparity}}
\begin{proof}[Proof of Lemma \ref{lemma:tylrparity}] 
We carry out a case-by-case analysis.

\smallskip
    \noindent (1)
     If the parities of $n,m$ are different: Then, by item (4) of Lemma \ref{lemma:Legdrprop} and the parity assumption, $P_{n,d}(t)t^m(1-t^2)^{\frac{d-3}{2}}$ is odd function. So, $\int_{-1}^1 P_{n,d}(t)t^m(1-t^2)^{\frac{d-3}{2}}\,dt=0$.
    
    \smallskip
    \noindent (2) If the parities of $n,m$ are the same, and $n>m$: It is easy to show that
    $t^m=\sum_{k=0}^m b_k P_{k,d}(t)$
    since each $P_{k,d}(t)$ is a degree $k$ polynomial on $t$. So, by  orthogonality (see item (6) of Lemma \ref{lemma:Legdrprop}) and $n>m$, we have
    $$\int_{-1}^1 P_{n,d}(t)t^m(1-t^2)^{\frac{d-3}{2}}\,dt=0.$$
    
   \smallskip
    \noindent (3) If the parities of $n,m$ are the same, and $m=n+2k$: Recall the Rodrigues representation formula (see item (3) of Lemma \ref{lemma:Legdrprop}):
    $$P_{n,d}(t)=(-1)^n R_{n,d}(1-t^2)^{\frac{3-d}{2}}\left(\frac{d}{dt}\right)^n (1-t^2)^{n+\frac{d-3}{2}}$$ where $R_{n,d}=\frac{\Gamma\left(\frac{d-1}{2}\right)}{2^n\Gamma\left(n+\frac{d-1}{2}\right)}$.     Therefore,
    \begin{align*}
        \int_{-1}^1 P_{n,d}(t)t^{n+2k}(1-t^2)^{\frac{d-3}{2}}\,dt&=(-1)^n R_{n,d}\int_{-1}^1\left(\frac{d}{dt}\right)^n (1-t^2)^{n+\frac{d-3}{2}} t^{n+2k}\,dt\\
        &=R_{n,d}\frac{(n+2k)!}{(2k)!}\int_{-1}^1 (1-t^2)^{n+\frac{d-3}{2}}t^{2k}\,dt\quad(\because\text{integration by parts }n\text{ times})\\
        &=R_{n,d}\frac{(n+2k)!}{(2k)!}\int_0^1 (1-s)^{n+\frac{d-3}{2}} s^{k-\frac{1}{2}}\,ds\\
        &=R_{n,d}\frac{(n+2k)!}{(2k)!}B\left(k+\frac{1}{2},n+\frac{d-1}{2}\right)
    \end{align*}
    where $B(x,y):=\frac{\Gamma(x)\Gamma(y)}{\Gamma(x+y)}$ represents the Beta function (see \cite[1.14]{atkinson2012spherical}). Hence,
    \begin{align*}
        \int_{-1}^1 P_{n,d}(t)t^{n+2k}(1-t^2)^{\frac{d-3}{2}}\,dt&=R_{n,d}\frac{(n+2k)!}{(2k)!}\frac{\Gamma(k+\frac{1}{2})\Gamma(n+\frac{d-1}{2})}{\Gamma(k+n+\frac{d}{2})}
        =\frac{(n+2k)!}{2^n(2k)!}\frac{\Gamma(k+\frac{1}{2})\Gamma(\frac{d-1}{2})}{\Gamma(k+n+\frac{d}{2})}.
    \end{align*}
\end{proof}

\subsection{Proof of Lemma \ref{lemma:russianguys}}
\begin{proof}[Proof of Lemma \ref{lemma:russianguys}]
We follow a  strategy akin to  \cite[Section 6]{bogomolny2007distance}. Recall that
$$\eta_{n,d}=\frac{\vert \Sp^{d-2}\vert}{\vert\Sp^{d-1}\vert}\int_{-1}^1 P_{n,d}(t)\arccos(t)(1-t^2)^{\frac{d-3}{2}}\,dt.$$
Moreover, by Erdelyi's argument \cite[pg.248]{bateman1953higher}, it is shown that
\begin{equation}\label{eq:eta}\eta_{n,d}=\frac{i^n(2\pi)^{\frac{d}{2}}}{\vert\Sp^{d-1}\vert}\int_{-\infty}^\infty t^{1-\frac{d}{2}}J_{n+\frac{d}{2}-1}(t)\hat{f}(t)\,dt\end{equation}
where $J_{n+\frac{d}{2}-1}(t)$ denotes Bessel function \cite[Chapter VII]{bateman1953higher}, and
$\hat{f}(t)=\frac{1}{2\pi}\int_{-1}^1e^{-ixt}\arccos(x)\,dx.$
Moreover, by substituting $x=\cos\theta$,
\begin{align*}
    \hat{f}(t)&=\frac{1}{2\pi}\int_{-1}^1e^{-ixt}\arccos(x)\,dx=\frac{1}{2\pi}\int_0^\pi \theta\sin\theta e^{-it\cos\theta}\,d\theta=\frac{1}{2it}(e^{it}-J_0(t))
\end{align*}
Corresponding to the two terms in $\hat{f}(t)$ there are two terms in $\eta_{n,d}$ via equation (\ref{eq:eta}). The integral including $e^{it}$ is zero for all $n\neq 0$ and the integral with $J_0(t)$ is,
$$\eta_{n,d}=-\frac{i^{n-1}(2\pi)^{\frac{d}{2}}}{\vert\Sp^{d-1}\vert}\int_0^\infty t^{-\frac{d}{2}}J_{n+\frac{d}{2}-1}(t)J_0(t)\,dt$$

The last integral can be computed using the following integral \cite[7.7.4.30]{bateman1953higher}
$$\int_0^\infty t^{-\rho}J_\mu(t)J_\nu(t)\,dt=\frac{\Gamma(\rho)\Gamma(\frac{1}{2}(\nu+\mu+1-\rho))}{2^\rho\Gamma(\frac{1}{2}(1+\nu-\mu+\rho))\Gamma(\frac{1}{2}(1+\nu+\mu+\rho))\Gamma(\frac{1}{2}(1+\mu-\nu+\rho))}.$$
Hence, the final result is
\begin{align*}
    \eta_{n,d}&=-\frac{1}{\vert\Sp^{d-1}\vert}(-1)^{\frac{n-1}{2}}\pi^{\frac{d}{2}}\frac{\Gamma(\frac{d}{2})\Gamma(\frac{n}{2})}{\Gamma(1-\frac{n}{2})(\Gamma(\frac{n+d}{2}))^2}=-\frac{\pi^{\frac{d-1}{2}}n!!\Gamma(\frac{d}{2})\Gamma(\frac{n}{2})}{n2^{\frac{n-1}{2}}(\Gamma(\frac{n+d}{2}))^2\vert\Sp^{d-1}\vert}.
\end{align*}
\end{proof}

\subsection{Proof of Lemma \ref{lemma:tlfrtrc}}
\begin{proof}[Proof of Lemma \ref{lemma:tlfrtrc}]
Observe that $\arccos^2(t),(1-t^2),(1+r^2+2rt)\geq 0$ for any $r\in (0,1)$ and $t\in [-1,1]$. So, obviously $\int_{-1}^1\frac{\arccos^2(t)(1-t^2)^{\frac{d-3}{2}}}{(1+r^2+2rt)^{\frac{d}{2}}}\,dt\geq 0$. Moreover,
\begin{align*}
    \int_{-1}^1\frac{\arccos^2(t)(1-t^2)^{\frac{d-3}{2}}}{(1+r^2+2rt)^{\frac{d}{2}}}\,dt&\leq \pi^2\int_{-1}^1\frac{(1-t^2)^{\frac{d-3}{2}}}{(1+r^2+2rt)^{\frac{d}{2}}}\,dt\quad(\because\,\arccos(t)\in [0,\pi])\\
    &\leq \pi^2\int_{-1}^1\frac{(1+r^2+2rt)^{\frac{d-3}{2}}}{(1+r^2+2rt)^{\frac{d}{2}}}\,dt\quad(\because\,1-t^2\leq 1+r^2+2rt)\\
    &=\pi^2\int_{-1}^1\frac{1}{(1+r^2+2rt)^{\frac{3}{2}}}\,dt
\end{align*}
which equals $\frac{2\pi^2}{1-r^2}$, as we wanted.
\end{proof}

\subsection{Proof of Lemma \ref{lemma:Laplaciandistance}}
\begin{proof}[Proof of Lemma \ref{lemma:Laplaciandistance}]
Let $u=(u_1,\dots,u_d)\in\Sp^{d-1} \subset \mathbb{R}^d$. Observe that $d_{\Sp^{d-1}}(e_d,u)=\arccos(\langle e_d,u\rangle)=\arccos (u_d)$. We use the following expression for the Laplace-Beltrami operator $\Delta$ on $\Sp^{d-1}$ in these coordinates (see \cite[(3.18)]{atkinson2012spherical}): $\Delta=-\sum_{1\leq j<i\leq d}\left(u_i\frac{\partial}{\partial u_j}-u_j\frac{\partial}{\partial u_i}\right)^2.$ Then,
\begin{align*}
   \Delta\,d_{\Sp^{d-1}}(e_d,\cdot)(u)&=-\sum_{1\leq j<i\leq d}\left(u_i\frac{\partial}{\partial u_j}-u_j\frac{\partial}{\partial u_i}\right)^2\arccos (u_d)\\
    &=-\sum_{j=1}^{d-1}\left(u_d\frac{\partial}{\partial u_j}-u_j\frac{\partial}{\partial u_d}\right)\frac{u_j}{\sqrt{1-u_d^2}}\\
    &=-\sum_{j=1}^{d-1}\left(\frac{u_d}{\sqrt{1-u_d^2}}-\frac{u_j^2 u_d}{(1-u_d^2)^{\frac{3}{2}}}\right)\\
    &=-(d-1)\cdot\frac{u_d}{\sqrt{1-u_d^2}}+\frac{u_d}{\sqrt{1-u_d^2}}\quad\left(\because\,\sum_{j=1}^{d-1}u_j^2=1-u_d^2\right)\\
    &=-(d-2)\cdot\frac{u_d}{\sqrt{1-u_d^2}}=-(d-2)\cdot\cot d_{\Sp^{d-1}}(e_d,u).
\end{align*}
\end{proof}

\subsection{Proof of Lemma \ref{lemma:lambdaf3dbdd}}
\begin{proof}[Proof of Lemma \ref{lemma:lambdaf3dbdd}]
By equation (\ref{eq:lambdan2f}),
$$\lambda_{n,3}^f=-\frac{\vert\Sp^1\vert}{2\vert\Sp^2\vert}\int_{-1}^1 P_{n,3}(t) f^2(t)\,dt=-\frac{1}{2\vert\Sp^{2}\vert}\int_{\Sp^2}P_{n,3}(\langle e_3, u\rangle)F(u)\,d\,\vol_{\Sp^2}(u).$$

Observe that, for any $\delta\in(0,1)$,
\begin{align}
    \vert\lambda_{n,3}^f\vert&=\frac{1}{2\vert\Sp^2\vert}\left\vert\int_{\Sp^2}P_{n,3}(\langle e_3\cdot u \rangle)F(u)\,d\vol_{\Sp^2}(u)\right\vert\nonumber\\
    &=\frac{1}{2n(n+1)\vert\Sp^2\vert}\left\vert\int_{\Sp^2}\langle \nabla_u P_{n,3}(\langle e_3\cdot u \rangle), \nabla F(u)\rangle\,d\vol_{\Sp^2}(u)\right\vert\quad(\because\text{ item (9) of Lemma \ref{lemma:Legdrprop}})\nonumber\\
    &=\frac{1}{2n(n+1)\vert\Sp^2\vert}\left\vert\int_{\Sp^2}P_{n,3}(\langle e_3\cdot u \rangle)\Delta F(u)\,d\vol_{\Sp^2}(u)\right\vert\quad(\because\text{ item (8) of Lemma \ref{lemma:Legdrprop}})\nonumber\\
    &\leq \frac{1}{2n(n+1)\vert\Sp^2\vert}\Vert P_{n,3}(\langle e_3\cdot u \rangle)\Delta F(u) \Vert_{L^1(\vol_{\Sp^2})}\nonumber\\
    &\leq \frac{1}{2n(n+1)\vert\Sp^2\vert}\Vert P_{n,3}(\langle e_3\cdot u \rangle) \Vert_{L^{1+\frac{1}{\delta}}(\vol_{\Sp^2})}\cdot\Vert\Delta F(u) \Vert_{L^{1+\delta}(\vol_{\Sp^2})}\quad(\because\text{Holder's inequality})\label{eq:holdereq}
\end{align}
where the second and third equality make sense because $f$ is smooth almost everywhere so that $\nabla F$ and $\Delta F$ are well-defined almost everywhere on $\Sp^2$. Moreover, note that

\begin{align*}
    \int_{\Sp^2} \vert P_{n,3}(\langle e_3, u\rangle) \vert^{1+\frac{1}{\delta}}\,d\,\vol_{\Sp^2}(u)&=\int_{\Sp^2} \vert P_{n,3}(\langle e_3, u\rangle) \vert^2\cdot\vert P_{n,3}(\langle e_3, u\rangle) \vert^{\frac{1}{\delta}-1}\,d\,\vol_{\Sp^2}(u)\\
    &\leq\int_{\Sp^2} \vert P_{n,3}(\langle e_3, u\rangle) \vert^2\,d\,\vol_{\Sp^2}(u)\quad(\because\text{ item (1) of Lemma \ref{lemma:Legdrprop}})\\
    &=\frac{\vert\Sp^2 \vert}{N_{n,3}}\quad(\because\text{ item (2) of Lemma \ref{lemma:Legdrprop}}).
\end{align*}
This implies that $\Vert P_{n,3}(\langle e_3\cdot u \rangle) \Vert_{L^{1+\frac{1}{\delta}}(\vol_{\Sp^2})}\leq C_\delta N_{n,3}^{-\frac{1}{1+\frac{1}{\delta}}}$ for some positive constant $C_\delta>0$. Finally, from equation (\ref{eq:holdereq}) we conclude that
$$\vert \lambda_{n,3}^f\vert\leq\frac{C_\delta' \,\Vert\Delta F(u) \Vert_{L^{1+\delta}(\vol_{\Sp^2})}}{n(n+1)}N_{n,3}^{-\frac{1}{1+\frac{1}{\delta}}}\text{ for some }C_\delta'>0.$$
\end{proof}

\subsection{Proof of Lemma \ref{lemma:pullbackkernel}}
\begin{proof}[Proof of Lemma \ref{lemma:pullbackkernel}]
Let's prove item (1) first. Since $f\in L^1(\mu_X\otimes\mu_Y)$, the map $x\mapsto \int_Y f(x,y)\,d\mu_Y(y)$ is well-defined $\mu_X$-almost everywhere. Furthermore, because $\mu$ is a coupling measure between $\mu_X$ and $\mu_Y$, we have
\begin{align*}
    \int_{B(x,r)}\int_Y f(x',y)\,d\mu_Y(y)\,d\mu_X(x')=\mu(B(x,r)\times Y)=\mu_X(B(x,r))
\end{align*}
for any $x\in X$ and $r>0$. Hence, by the Lebesgue differentiability property of $\mX$, we conclude 
$$\int_Y f(x,y)\,d\mu_Y(y)=\lim_{r\rightarrow 0}\frac{1}{\mu_X(B(x,r))}\int_{B(x,r)}\int_Y f(x',y)\,d\mu_Y(y)\,d\mu_X(x')=1$$
$\mu_X$-almost everywhere.

To prove item (2), it is enough to show
$$\Vert (\pi_X\times\pi_X)^*(K_\mX)-(\pi_X\times\pi_X)^*(K_\mX^+)\Vert_{L^2(\mu\otimes\mu)}\leq\Vert (\pi_X\times\pi_X)^*(K_\mX)-L\Vert_{L^2(\mu\otimes\mu)}$$
for an arbitrary fixed $L\in L^2_{\succeq 0,\mathrm{sym}}(\mu\otimes\mu)$. We define 
$$
    \overline{L}:X\times X\longrightarrow\R\,\text{  s.t.  }
    (x,x')\longmapsto\iint_{Y\times Y}L(x,y,x',y')f(x,y)f(x',y')\,d\mu_Y(y)\,d\mu_Y(y').$$

Note that $\overline{L}$ is well-defined since $L\in L^2(\mu\otimes\mu)$.

\begin{claim}\label{clim:avgLL2sympostv}
$\overline{L}\in L^2_{\succeq 0,\mathrm{sym}}(\mu_X\otimes\mu_X)$.
\end{claim}
\begin{proof}[Proof of Claim \ref{clim:avgLL2sympostv}]
Observe that
\begin{align*}
    &\iint_{X\times X}\vert\overline{L}(x,x')\vert^2\,d\mu_X(x)\,d\mu_X(x')\\
    &\leq\iint_{X\times X}\left(\iint_{Y\times Y}\vert L(x,y,x',y')\vert f(x,y)f(x',y')\,d\mu_Y(y)\,d\mu_Y(y')\right)^2\,d\mu_X(x)\,d\mu_X(x')\\
    &\leq\iiiint_{X\times Y\times X\times Y}\vert L(x,y,x',y')\vert^2f(x,y)f(x',y')\,d\mu_X(x)\,d\mu_Y(y)\,d\mu_X(x')\,d\mu_Y(y')\\
    &=\iiiint_{X\times Y\times X\times Y}\vert L(x,y,x',y')\vert^2\,d\mu(x,y)\,d\mu(x',y')<\infty
\end{align*}
since $L\in L^2(\mu\otimes\mu)$. Hence, $\overline{L}\in L^2(\mu_X\otimes\mu_X)$. $\overline{L}$ is obviously symmetric since $L$ is so. Finally, fix an arbitrary $\phi\in L^2(\mu_X)$ and note that,
\begin{align*}
    &\iint_{X\times X} \overline{L}(x,x')\phi(x)\phi(x')\,d\mu_X(x)\,d\mu_X(x')\\
    &=\iiiint_{X\times Y\times X\times Y} L(x,y,x',y')\phi(x)\phi(x')f(x,y)f(x',y')\,d\mu_X(x)\,d\mu_Y(y)\,d\mu_X(x')\,d\mu_Y(y')\\
    &=\iiiint_{X\times Y\times X\times Y} L(x,y,x',y')\phi(x)\phi(x')\,d\mu(x,y)\,d\mu(x',y')\geq 0
\end{align*}
since the integral operator associated to $L$ is positive semi-definite. Hence, we can conclude that $\overline{L}\in L^2_{\succeq 0,\mathrm{sym}}(\mu_X\otimes\mu_X)$.
\end{proof}

\begin{claim}\label{claim:avgLleqL}
$\Vert K_\mX-\overline{L}\Vert_{L^2(\mu_X\otimes\mu_X)}\leq\Vert (\pi_X\times\pi_X)^*(K_\mX)-L\Vert_{L^2(\mu\otimes\mu)}$
\end{claim}
\begin{proof}[Proof of Claim \ref{claim:avgLleqL}]
\begin{align*}
    &\Vert K_\mX-\overline{L}\Vert_{L^2(\mu_X\otimes\mu_X)}^2\\
    &=\iint_{X\times X}\vert K_\mX(x,x')-\overline{L}(x,x')\vert^2\,d\mu_X(x)\,d\mu_X(x')\\
    &=\iint_{X\times X} \left\vert \iint_{Y\times Y} (K_\mX(x,x')-L(x,y,x',y'))f(x,y)f(x',y')\,d\mu_Y(y)\,d\mu_Y(y')\right\vert^2\,d\mu_X(x)\,d\mu_X(x')\\
    &\quad(\because\text{ item (1)})\\ 
    &\leq\iint_{X\times X} \left(\iint_{Y\times Y} \vert K_\mX(x,x')-L(x,y,x',y')\vert f(x,y)f(x',y')\,d\mu_Y(y)\,d\mu_Y(y')\right)^2\,d\mu_X(x)\,d\mu_X(x')\\
    &\leq\iiiint_{X\times Y\times X\times Y}\vert K_\mX(x,x')-L(x,y,x',y')\vert^2 f(x,y)f(x',y')\,d\mu_X(x)\,d\mu_Y(y)\,d\mu_X(x')\,d\mu_Y(y')\\
    &=\iiiint_{X\times Y\times X\times Y}\vert K_\mX(x,x')-L(x,y,x',y')\vert^2 \,d\mu(x,y)\,d\mu(x',y')\\
    &=\Vert (\pi_X\times\pi_X)^*(K_\mX)-L\Vert_{L^2(\mu\otimes\mu)}^2.
\end{align*}
\end{proof}

Then, finally we can prove the following inequality:

\begin{align*}
    \Vert (\pi_X\times\pi_X)^*(K_\mX)-(\pi_X\times\pi_X)^*(K_\mX^+)\Vert_{L^2(\mu\otimes\mu)}
    &=\Vert K_\mX-K_\mX^+\Vert_{L^2(\mu_X\otimes\mu_X)}\\
    &\leq\Vert K_\mX-\overline{L}\Vert_{L^2(\mu_X\otimes\mu_X)}\,(\because\text{  definition of }K_\mX^+\text{ and Claim \ref{clim:avgLL2sympostv}})\\
    &\leq\Vert (\pi_X\times\pi_X)^*(K_\mX)-L\Vert_{L^2(\mu\otimes\mu)}\,(\because\text{ Claim \ref{claim:avgLleqL}}).
\end{align*}

This completes the proof of item (2).\medskip

Lastly, for item (3), observe that $$\Vert(\pi_X\times\pi_X)^*(K_\mX^+)-(\pi_Y\times\pi_Y)^*(K_\mY^+)\Vert_{L^2(\mu\otimes\mu)}\leq\Vert(\pi_X\times\pi_X)^*(K_\mX)-(\pi_Y\times\pi_Y)^*(K_\mY)\Vert_{L^2(\mu\otimes\mu)}$$
by item (2) and Theorem \ref{thm:projnonexpansive}. Hence,
\begin{align*}
    \Vert K_\mX^+-K_\mY^+\Vert_{L^2(\mu\otimes\mu)}&=\Vert(\pi_X\times\pi_X)^*(K_\mX^+)-(\pi_Y\times\pi_Y)^*(K_\mY^+)\Vert_{L^2(\mu\otimes\mu)}\\
    &\leq\Vert(\pi_X\times\pi_X)^*(K_\mX)-(\pi_Y\times\pi_Y)^*(K_\mY)\Vert_{L^2(\mu\otimes\mu)}=\Vert K_\mX-K_\mY\Vert_{L^2(\mu\otimes\mu)}.
\end{align*}
\end{proof}

\subsection{Proof of Lemma \ref{lemma:KXL2bdd}}
\begin{proof}[Proof of Lemma \ref{lemma:KXL2bdd}]
By the definition of $K_\mX$ and $K_\mY$ we have:
\begin{align*}
    &2\,\big\vert K_\mX(x,x')-K_\mY(y,y') \big\vert
    \leq \vert d_X^2(x,x')-d_Y^2(y,y')\vert +\left\vert \int_X d_X^2(x,u)\,d\mu_X(u)-\int_Y d_Y^2(y,v)\,d\mu_Y(v)\right\vert\\
    &\quad+\left\vert \int_X d_X^2(x',u)\,d\mu_X(u)-\int_Y d_Y^2(y',v)\,d\mu_Y(v)\right\vert +\big\vert \big(\diamna_2(\mX)\big)^2-\big(\diamna_2(\mY)\big)^2\big\vert
\end{align*}
for any $x,x'\in X$ and $y,y'\in Y$. Also, if $D:=\max\{\diamna(X),\diamna(Y)\}$, we have
    \begin{align*}
    \vert d_X^2(x,x')-d_Y^2(y,y')\vert&=\vert d_X(x,x')+d_Y(y,y')\vert\cdot\vert d_X(x,x')-d_Y(y,y')\vert
    \leq 2\,D\cdot\vert d_X(x,x')-d_Y(y,y')\vert.
\end{align*}
Also, by employing  similar manipulations together with Remark \ref{rmk:dgwpdiam}, we can deduce:
\begin{align*}
    \left\vert \int_X d_X^2(x,u)\,d\mu_X(u)-\int_Y d_Y^2(y,v)\,d\mu_Y(v)\right\vert
    &\leq 2D\cdot\left(\int_{X\times Y} \vert d_X(x,u)-d_Y(y,v)\vert^2\,d\mu(u,v)\right)^{\frac{1}{2}}\\
    \left\vert \int_X d_X^2(x',u)\,d\mu_X(u)-\int_Y d_Y^2(y',v)\,d\mu_Y(v)\right\vert
    &\leq 2D\cdot\left(\int_{X\times Y} \vert d_X(x',u)-d_Y(y',v)\vert^2\,d\mu(u,v)\right)^{\frac{1}{2}}\\
    \big\vert \big(\diamna_2(\mX)\big)^2-\big(\diamna_2(\mY)\big)^2\big\vert&\leq 4D\cdot d_{\mathrm{GW},2}(\mX,\mY).
\end{align*}

The claim is obtained by combining all the above.
\end{proof}

\section{Concepts from Functional analysis and spectral theory of operators}\label{sec:sec:functanal}

We review the basics of spectral theory, including  the notions of compact and self-adjoint operators on Hilbert spaces. See \cite{reed2012methods} and \cite[Chs. 6 and 8]{renardy2006introduction} for a complete treatment.

\begin{adefinition}[bounded operator]
Let $(\mathcal{B},\Vert\cdot\Vert_{\mathcal{B}})$ and $(\mathcal{C},\Vert\cdot\Vert_{\mathcal{C}})$ be Banach spaces, and $\mathfrak{A}:\mathcal{B}\rightarrow\mathcal{C}$ be an operator from $\mathcal{B}$ to $\mathcal{C}$. Then, $\mathfrak{A}$ is said to be \textbf{bounded} if
$$\Vert \mathfrak{A} \Vert:=\sup_{b\in\mathcal{B}\backslash\{0\}}\frac{\Vert \mathfrak{A}b\Vert_{\mathcal{C}}}{\Vert b\Vert_{\mathcal{B}}}<\infty.$$
\end{adefinition}

From now on, $\mathcal{L}(\mathcal{B},\mathcal{C})$ denotes the set of bounded operators from $\mathcal{B}$ to $\mathcal{C}$. Also, for the simplicity, we use $\mathcal{L}(\mathcal{B})$ instead of $\mathcal{L}(\mathcal{B},\mathcal{B})$.

\begin{atheorem}[{adjoint \cite[Section VI.2]{reed2012methods}}]
Let $\mathcal{H}$ be a Hilbert space. Then, for any $\mathfrak{A}\in \mathcal{L}(\mathcal{H})$, there is a unique bounded operator $\mathfrak{A}^*:\mathcal{H}\longrightarrow\mathcal{H}$ satisfying:
$$\langle u,\mathfrak{A}v \rangle=\langle \mathfrak{A}^*u,v \rangle\,\,\mbox{for any $u,v\in\mathcal{H}$.}$$
\end{atheorem}

From now on, the operator $\mathfrak{A}^*$ will be called the \textbf{adjoint} of $\mathfrak{A}$.

\begin{adefinition}[self-adjoint operator]
Let $\mathcal{H}$ be a Hilbert space. Then, $\mathfrak{A}\in \mathcal{L}(\mathcal{H})$ is said to be \textbf{self-adjoint} if $\mathfrak{A}=\mathfrak{A}^*$. Equivalently, if $\mathfrak{A}$ satisfies the following:
$$\langle \mathfrak{A}u,v\rangle=\langle u,\mathfrak{A}v\rangle
\,\,\mbox{for any $u,v\in\mathcal{H}$.}$$
\end{adefinition}

\begin{adefinition}
Let $\mathcal{H}$ be a Hilbert space and $\mathfrak{A}\in \mathcal{L}(\mathcal{H})$. Then, a complex number $\lambda$ is said to be in the \textbf{resolvent set} $\rho(\mathfrak{A})$ of $\mathfrak{A}$ if $\lambda I-\mathfrak{A}$ is a bijection with a bounded inverse. If $\lambda\notin\rho(\mathfrak{A})$, then $\lambda$ is said to be in the \textbf{spectrum} $\sigma(\mathfrak{A})$ of $\mathfrak{A}$.
\end{adefinition}

\begin{adefinition}
Let $\mathcal{H}$ be a Hilbert space and $\mathfrak{A}\in \mathcal{L}(\mathcal{H})$.
\begin{enumerate}
    \item An $u\neq 0$ which satisfies $\mathfrak{A}u=\lambda u$ for some $\lambda\in\mathbb{C}$ is called an \textbf{eigenvector} of $\mathfrak{A}$; $\lambda$ is called the corresponding eigenvalue. If $\lambda$ is an eigenvalue, then $\lambda I-\mathfrak{A}$ is not injective so $\lambda\in\sigma(\mathfrak{A})$. The set of all eigenvalues is called the \textbf{point spectrum} of $\mathfrak{A}$.
    
    \item If $\lambda$ is not an eigenvalue and if the range of $\lambda I-\mathfrak{A}$ is not dense, then $\lambda$ is said to be in the \textbf{residual spectrum}.
\end{enumerate}
\end{adefinition}

\begin{atheorem}[{\cite[Theorem VI.6, Theorem VI.8]{reed2012methods}}]\label{thm"selfadjtspectrmprpty}
Let $\mathfrak{A}$ be a bounded and self-adjoint operator on a Hilbert space $\mathcal{H}$. Then,
\begin{enumerate}
    \item $\mathfrak{A}$ has no residual spectrum.
    
    \item The spectrum $\sigma(\mathfrak{A})$ of $\mathfrak{A}$ is a subset of $\mathbb{R}$. In particular, $\sup_{\lambda\in\sigma(\mathfrak{A})}\vert\lambda\vert=\Vert \mathfrak{A} \Vert$.
    
    \item Eigenvectors corresponding to distinct eigenvalues of $\mathfrak{A}$ are orthogonal.
\end{enumerate}
\end{atheorem}

\begin{adefinition}[positive semi-definite operator]\label{def:psdop}
Let $\mathcal{H}$ be a Hilbert space and $\mathfrak{A}\in \mathcal{L}(\mathcal{H})$. $\mathfrak{A}$ is called \textbf{positive semi-definite} if $\langle u,\mathfrak{A}u \rangle\geq 0$ for any $u\in\mathcal{H}$. We write $\mathfrak{A}\succeq 0$ if $\mathfrak{A}$ is positive semi-definite.
\end{adefinition}

\begin{aremark}\label{rml:A^*Apos}
If $\mathfrak{A}$ is a bounded operator on a Hilbert space $\mathcal{H}$. Then, $\mathfrak{A}^*\mathfrak{A}$ is always positive semi-definite.
\end{aremark}

\begin{atheorem}[{\cite[Theorem VI.9]{reed2012methods}}]\label{thm:sqroot}
Let $\mathfrak{A}$ be a positive semi-definite and bounded operator on a Hilbert space $\mathcal{H}$. Then, there is a unique positive semi-definite  bounded operator $\mathfrak{B}$ on $\mathcal{H}$ satisfying $\mathfrak{B}^2=\mathfrak{A}$. Furthermore, $\mathfrak{B}$ commutes with every bounded operator  commuting with $\mathfrak{A}$.
\end{atheorem}

By Theorem \ref{thm:sqroot} and Remark \ref{rml:A^*Apos}, we define $\vert \mathfrak{A} \vert$ in the following way.

\begin{adefinition}\label{def:operatorpmdecomposition}
Let $\mathfrak{A}$ be a bounded operator on a Hilbert space $\mathcal{H}$. Define the following associated operators: $\vert \mathfrak{A} \vert:=\sqrt{\mathfrak{A}^*\mathfrak{A}}$, $\mathfrak{A}^+:=\frac{1}{2}(\vert \mathfrak{A} \vert+\mathfrak{A})$, and $\mathfrak{A}^-:=\frac{1}{2}(\vert \mathfrak{A} \vert-\mathfrak{A})$. Note that $\mathfrak{A}=\mathfrak{A}^+-\mathfrak{A}^-$.
\end{adefinition}

\begin{adefinition}[compact operator]
Let $(\mathcal{B},\Vert\cdot\Vert_{\mathcal{B}})$ and $(\mathcal{C},\Vert\cdot\Vert_{\mathcal{C}})$ be Banach spaces. An operator $\mathfrak{A}\in \mathcal{L}(\mathcal{B},\mathcal{C})$ is called \textbf{compact} if $\mathfrak{A}$ takes bounded sets in $\mathcal{B}$ into precompact sets in $\mathcal{C}$.
\end{adefinition}

\begin{atheorem}[{Riesz-Schauder Theorem, \cite[Theorem VI.15]{reed2012methods}}]\label{thm:RieszSchauder}
Let $\mathcal{H}$ be a separable Hilbert space and $\mathfrak{A}\in \mathcal{L}(\mathcal{H})$ be a compact operator. Then, $\sigma(\mathfrak{A})$ is a discrete set having no limit points  except perhaps $\lambda=0$. Further, any nonzero $\lambda\in\sigma(\mathfrak{A})$ is an eigenvalue with finite multiplicity (i.e. the corresponding space of eigenvectors is finite dimensional).
\end{atheorem}

\begin{atheorem}[{Hilbert-Schmidt Theorem, \cite[Theorem VI.16]{reed2012methods} and \cite[Theorem 8.94]{renardy2006introduction}}]\label{thm:HilbertSchmidt}
Let $\mathcal{H}$ be a separable Hilbert space and $\mathfrak{A}\in \mathcal{L}(\mathcal{H})$ be compact and self-adjoint. Then, there is a complete orthonormal basis $\{\phi_i\}_{i=1}^N$ for $\mathcal{H}$ where $N=\mathrm{dim}(\mathcal{H})$ so that $\mathfrak{A}\phi_i=\lambda_i\phi_i$, $\vert\lambda_i\vert$ is monotone nonincreasing, and if $N=\infty$, $\lim_{i\rightarrow\infty}\lambda_i=0$. Furthermore, $\mathfrak{A}$ can be represented by
$$\mathfrak{A}(\cdot)=\sum_{i=1}^N\lambda_i\langle\phi_i,\cdot\rangle\phi_i.$$
\end{atheorem}

\begin{adefinition}[singular values]
Let $\mathcal{H}$ be a separable Hilbert space and $\mathfrak{A}\in \mathcal{L}(\mathcal{H})$ be a compact operator. Then, the non-negative square roots of eigenvalues of the self-adjoint operator $\mathfrak{A}^*\mathfrak{A}$ are called \textbf{singular values} of $\mathfrak{A}$. 
\end{adefinition}

For a given mm-space $\mX=(X,d_X,\mu_X)\in\mathcal{M}_w$, the following theorem gives conditions under which the Hilbert space $L^2(\mu_X)$ is separable.

\begin{aproposition}[{\cite[Proposition 3.4.5]{cohn2013measure}}]
Let $(X,\Sigma_X,\mu_X)$ be a measure space, and let p satisfy $1\leq p <\infty$. If $\mu_X$ is $\sigma$-finite and $\Sigma_X$ is countably generated, then $L^p(X,\Sigma_X,\mu_X)$ is separable.
\end{aproposition}

In particular, if $(X,d_X)$ is complete and separable metric space (every compact metric space satisfies these conditions) and $\mu_X$ is Borel probability measure, then $L^2(\mu_X)$ is separable.\medskip

Next, we introduce particularly important kinds of compact operators: Trace class and Hilbert-Schmidt operators.

\begin{atheorem}[{\cite[Theorem VI.18]{reed2012methods}}]\label{thm:psdtr}
Let $\mathcal{H}$ be a separable Hilbert space and let $\{\phi_i\}_{i=1}^N$ be an orthonormal basis of $\mathcal{H}$ . Then, for any positive semi-definite operator $\mathfrak{A}\in \mathcal{L}(\mathcal{H})$, the value
$\sum_{i=1}^N \langle \phi_i,\mathfrak{A}\phi_i \rangle$
does not depend on the choice of basis.
\end{atheorem}

Note that each summand in the expression in Theorem \ref{thm:psdtr} is non-negative since $\mathfrak{A}$ is positive semi-definite. Hence, the summation is well-defined though it can be infinite.

\begin{adefinition}[trace of a positive semi-definite operator]\label{def:psdtr}
Let $\mathcal{H}$ be a separable Hilbert space and $\mathfrak{A}\in \mathcal{L}(\mathcal{H})$ be a positive semi-definite operator. Then,
$\mathrm{tr}(\mathfrak{A}):=\sum_{n=1}^N \langle \phi_n,\mathfrak{A}\phi_n \rangle$
where $\{\phi_n\}_{n=1}^N$ is an orthonormal basis of $\mathcal{H}$.
\end{adefinition}

The notion of trace can be extended from  positive semi-definite operators to trace class operators.

\begin{adefinition}[trace class] \label{def:trace-norm}
Let $\mathcal{H}$ be a separable Hilbert space. Then, for any $\mathfrak{A}\in \mathcal{L}(\mathcal{H})$, we define the \textbf{trace norm} of $\mathfrak{A}$ in the following way:
$$\Vert \mathfrak{A} \Vert_1:=\mathrm{tr}(\vert \mathfrak{A} \vert).$$
Moreover, $\mathfrak{A}$ is said to be \textbf{trace class} or \textbf{traceable} if $\Vert \mathfrak{A} \Vert_1<\infty$. The set of all traceable operators in $\mathcal{L}(\mathcal{H})$ is denoted by $\mathcal{L}_1(\mathcal{H})$.
\end{adefinition}

\begin{atheorem}[{\cite[Theorems VI.19 and  VI.20]{reed2012methods}}]
Let $\mathcal{H}$ be a separable Hilbert space. Then,
\begin{enumerate}
    \item $\mathcal{L}_1(\mathcal{H})$ is a Banach space with the trace norm $\Vert\cdot\Vert_1$.
    
    \item If $\mathfrak{A}\in\mathcal{L}_1(\mathcal{H})$ and $\mathfrak{B}\in\mathcal{L}(\mathcal{H})$, then $\mathfrak{A}\mathfrak{B}\in\mathcal{L}_1(\mathcal{H})$ and $\mathfrak{B}\mathfrak{A}\in\mathcal{L}_1(\mathcal{H})$.
    
    \item If $\mathfrak{A}\in\mathcal{L}_1(\mathcal{H})$, then $\mathfrak{A}^*\in\mathcal{L}_1(\mathcal{H})$.
\end{enumerate}
\end{atheorem}

\begin{atheorem}[{\cite[Theorem VI.21]{reed2012methods}}]\label{thm:tracesingvaluecndtn}
If $\mathfrak{A}$ is trace class, then $\mathfrak{A}$ is compact. Moreover, a compact operator $\mathfrak{A}$ is trace class iff $\sum_{i=1}^N \lambda_i<\infty$ where $\{\lambda_i\}_{i=1}^N$ are the singular values of $\mathfrak{A}$. 
\end{atheorem}

Moreover, the notion of trace can be extended to trace class operators via the following theorem.

\begin{atheorem}[{\cite[Theorem VI.24]{reed2012methods}}]
Let $\mathcal{H}$ be a separable Hilbert space, and $\{\phi_i\}_{i=1}^N$ be an orthonormal basis of $\mathcal{H}$ . Then for any $\mathfrak{A}\in \mathcal{L}_1(\mathcal{H})$, the value
$\sum_{i=1}^N \langle \phi_i,\mathfrak{A}\phi_i \rangle$
converges absolutely and the limit does not depend on the choice of basis.
\end{atheorem}

\begin{adefinition}[trace of a trace class operator]
Let $\mathcal{H}$ be a separable Hilbert space and $\mathfrak{A}\in\mathcal{L}_1(\mathcal{H})$. Then,
$\tr(\mathfrak{A}):=\sum_{n=1}^N \langle \phi_n,\mathfrak{A}\phi_n \rangle$
where $\{\phi_n\}_{n=1}^N$ is an orthonormal basis of $\mathcal{H}$. Observe that $\tr(\mathfrak{A})=\mathrm{tr}(\mathfrak{A})$ if $\mathfrak{A}$ is positive semi-definite.
\end{adefinition}

\begin{atheorem}[{\cite[Theorem VI.25]{reed2012methods}}]
\begin{enumerate}
    \item $\tr(c\mathfrak{A}+\mathfrak{B})=c\tr(\mathfrak{A})+\tr(\mathfrak{B})$ for any $\mathfrak{A},\mathfrak{B}\in\mathcal{L}_1(\mathcal{H})$ and $c\in\mathbb{C}$.
    
    \item $\tr(\mathfrak{A})=\overline{\tr(\mathfrak{A}^*)}$ for any $\mathfrak{A}\in\mathcal{L}_1(\mathcal{H})$.
    
    \item $\tr(\mathfrak{A}\mathfrak{B})=\tr(\mathfrak{B}\mathfrak{A})$ if $\mathfrak{A}\in\mathcal{L}_1(\mathcal{H})$ and $\mathfrak{B}\in\mathcal{L}(\mathcal{H})$.
\end{enumerate}
\end{atheorem}

Next, we introduce the notion of Hilbert-Schmidt operators.

\begin{adefinition}[Hilbert-Schmidt operator]\label{def:HSop}
Let $\mathcal{H}$ be a separable Hilbert space. Then, for any $\mathfrak{A}\in \mathcal{L}(\mathcal{H})$, we define the \textbf{Hilbert-Schmidt norm} of $\mathfrak{A}$ in the following way:
$$\Vert \mathfrak{A} \Vert_2:=\big(\tr(\mathfrak{A}^*\mathfrak{A})\big)^{\frac{1}{2}}.$$
Moreover, $\mathfrak{A}$ is called \textbf{Hilbert-Schmidt operator} if $\Vert \mathfrak{A} \Vert_2<\infty$. The set of all Hilbert-Schmidt operators in $\mathcal{L}(\mathcal{H})$ is denoted by $\mathcal{L}_2(\mathcal{H})$.
\end{adefinition}

Let us collect some useful properties of trace, trace class and Hilbert-Schmidt operators.

\begin{alemma}[{\cite[Theorem VI.19, Theorem VI.21, Theorem VI.22, Theorem VI.25]{reed2012methods}}]\label{lemma:propstrHB}
Let $\mathcal{H}$ be a separable Hilbert space and $\mathfrak{A},\mathfrak{B},\mathfrak{C}\in \mathcal{L}(\mathcal{H})$. Then,
\begin{enumerate}
    \item $\Vert \mathfrak{A} \Vert\leq\Vert \mathfrak{A} \Vert_2\leq\Vert \mathfrak{A} \Vert_1$.
    
    \item If $\mathfrak{A}$ is Hilbert-Schmidt operator, then $\mathfrak{A}$ is compact. Moreover, a compact operator $\mathfrak{A}$ is Hilbert-Schmidt if and only if $\sum_{i=1}^N \lambda_i^2<\infty$ where $\{\lambda_i\}_{i=1}^N$ are the singular values of $\mathfrak{A}$.
    
    \item $\mathfrak{A}$ is trace class if and only if $\mathfrak{A}=\mathfrak{B}\mathfrak{C}$ for some Hilbert-Schmidt operators $\mathfrak{B},\mathfrak{C}$.
    
    \item $\tr(\mathfrak{A}\mathfrak{B})=\tr(\mathfrak{B}\mathfrak{A})$ if both of $\mathfrak{A}$ and $\mathfrak{B}$ are Hilbert-Schmidt operators.
\end{enumerate}
\end{alemma}
\begin{proof}
To prove item (4), which is the only one having no proof in \cite{reed2012methods}, we use item (3) of this lemma and follow the proof of \cite[Theorem VI.25]{reed2012methods}.
\end{proof}

The following theorem connects Hilbert-Schmidt operator and integral kernel.

\begin{atheorem}[{\cite[Theorem VI.23]{reed2012methods}}]\label{thm:Hilbertschmidtintgkernell}
Let $(X,\Sigma_X,\mu_X)$ be a measure space with the separable Hilbert space $L^2(\mu_X)$. Then, $\cmdsop\in\mathcal{L}(L^2(\mu_X))$ is Hilbert-Schmidt if and only if there is a function
$$K\in L^2(\mu_X\otimes\mu_X)\,\,\mbox{with}\,\,\cmdsop\phi(\cdot)=\int_X K(\cdot,x')\,\phi(x')\,d\mu_X(x').$$
Moreover,
$\Vert \cmdsop\Vert_2^2=\Vert K \Vert_{L^2(\mu_X\otimes\mu_X)}^2.$
\end{atheorem}

Furthermore, the integral kernel associated to a given Hilbert-Schmidt operator on $L^2(\mu_X)$ is uniquely determined (save for a set with zero measure); cf.  the following lemma and corollary.

\begin{alemma}[\cite{wofsey}]\label{lemma:zerokernelzerooperator}
Suppose a measure space $(X,\Sigma_X,\mu_X)$ and the separable Hilbert space $L^2(\mu_X)$ are given. Let $K\in L^2(\mu_X\otimes\mu_X)$ and let $\cmdsop\in\mathcal{L}(L^2(\mu_X))$ be the Hilbert-Schmidt operator associated to $K$. If $\cmdsop=0$, then $K=0$ $(\mu_X\otimes\mu_X)$-almost everywhere.
\end{alemma}
\begin{proof}
For each $x\in X$, let $K_x:=K(x,\cdot)$. Then, since $K\in L^2(\mu_X\otimes\mu_X)$, we have $K_x\in L^2(\mu_X)$ for $\mu_X$-almost every $x$. Observe that $\cmdsop\psi(x)=\langle \psi,K_x\rangle$ for any $\psi\in L^2(\mu_X)$. Now, choose a complete orthonormal basis $\{\phi_i\}_{i=1}^N$ of $L^2(\mu_X)$. For each $i$, $\cmdsop\phi_i=0$ implying that $K_x$ is orthogonal to $\phi_i$ $\mu_X$-almost every $x$. Since a countable union of null sets is null, $K_x$ is orthogonal to $\phi_i$ for all $i$ and for $\mu_X$-almost every $x$. Hence, $K_x=0$ almost everywhere for almost every $x\in X$. Therefore, $K=0$ $(\mu_X\otimes\mu_X)$-almost everywhere. 
\end{proof}

\begin{acorollary}\label{cor:HSuniquekernel}
Suppose a measure space $(X,\Sigma_X,\mu_X)$ and the separable Hilbert space $L^2(\mu_X)$ are given. Let $K_1,K_2\in L^2(\mu_X\otimes\mu_X)$ and $\cmdsop_1,\cmdsop_2\in\mathcal{L}(L^2(\mu_X))$ be the Hilbert-Schmidt operators associated to $K_1$ and $K_2$ respectively. If $\cmdsop_1=\cmdsop_2$, then $K_1=K_2$ $(\mu_X\otimes\mu_X)$-almost everywhere.
\end{acorollary}

\begin{adefinition}[symmetric kernel]\label{def:symker}
Let $(X,\Sigma_X,\mu_X)$ be a measure space. A kernel $K:X\times X\rightarrow\R$ is said to be \textbf{symmetric} if $K(x,x')=K(x',x)$ $\mu_X\otimes\mu_X$-a.e.
\end{adefinition}

\begin{atheorem}[{\cite[Corollary 3.2]{brislawn1991traceable}}]\label{thm:trintegdiag}
Let $\mu_X$ be a $\sigma$-finite Borel measure on a second countable space $X$, and let $K$ be the integral kernel of a trace class operator $\cmdsop$ on $L^2(\mu_X)$. If $K$ is continuous at $(x,x)$ for $\mu_X$-almost every $x$ then
$$\tr(\cmdsop)=\int_X K(x,x)\,d\mu_X(x).$$
\end{atheorem}

\section{Spherical harmonics, Legendre Polynomials, and the Funk-Hecke Theorem}\label{sec:sphehar}

Now we review basic results on spherical harmonics. Our main reference is \cite{atkinson2012spherical}.

For any $d\geq 2$ and $n\geq 0$, let $\mathbb{H}_n^d$ be the space of all homogeneous polynomials of degree $n$ in $d$ dimensions.

\begin{aexample}
Here are some examples of $\mathbb{H}_n^d$.
\begin{itemize}
    \item $\mathbb{H}_2^2=\{a_1x_1^2+a_2x_1x_2+a_3x_2^2:a_1,a_2,a_3\in\R\}$.
    
    \item $\mathbb{H}_3^2=\{a_1x_1^3+a_2x_1x_2^2+a_3x_1^2x_2+a_4x_2^3:a_1,a_2,a_3,a_4\in\R\}$.
    \item $\mathbb{H}_2^3=\{a_1x_1^2+a_2x_2^2+a_3x_3^2+a_4x^1x_2+a_5x_2x_3+a^5x_3x_1:a_1,\dots,a_6\in\R\}$.
\end{itemize}
\end{aexample}

Recall that a function $f:\R^d\rightarrow\R$ is called \textbf{harmonic} if $\Delta f=0$. The space of homogeneous harmonics of degree $n$ in $d$ dimension is denoted by $\mathbb{Y}_n(\R^d)$. More precisely,

$$\mathbb{Y}_n(\R^d):=\{f\in\mathbb{H}_n^d:\Delta f=0\}.$$

Let $N_{n,d}:=\mathrm{dim}(\mathbb{Y}_n(\R^d))$. The following equality is well-known (\cite[(2.10))]{atkinson2012spherical}). 

\begin{equation}
    N_{n,d}=\frac{(2n+d-2)(n+d-3)!}{n!(d-2)!}.
\end{equation}

Moreover, $N_{n,d}=O(n^{d-2})$ for $n$ sufficiently large (\cite[(2.12))]{atkinson2012spherical}).\medskip

Finally, we are ready to introduce spherical harmonics.

\begin{adefinition}[spherical harmonics]\label{def:spheharmon}
$\mathbb{Y}_n^d:=\mathbb{Y}_n(\R^d)\vert_{\Sp^{d-1}}$ is called the spherical harmonic space of order $n$ in $d$ dimensions. Any function in $\mathbb{Y}_n^d$ is called a spherical harmonic of order $n$ in $d$ dimensions.
\end{adefinition}

\begin{aremark}
It is known that $\mathrm{dim}(\mathbb{Y}_n^d)=\mathrm{dim}(\mathbb{Y}_n(\R^d))=N_{n,d}$.
\end{aremark}

Now, let us introduce Legendre polynomials, which will be useful later. A more detailed explanation on Legendre polynomials can be found in \cite[Section 2]{atkinson2012spherical}.

\begin{adefinition}[Legendre polynomial]\label{def:Legendrepol}
The map $P_{n,d}:[-1,1]\longrightarrow\mathbb{R}$
$$
    t\longmapsto n!\,\Gamma\left(\frac{d-1}{2}\right)\,\sum_{k=0}^{\lfloor n/2 \rfloor}(-1)^k\frac{(1-t^2)^k\,t^{n-2k}}{4^k\,k!\,(n-2k)!\,\Gamma\left(k+\frac{d-1}{2}\right)}
$$
is said to be the \emph{Legendre polynomial of degree $n$ in $d$ dimensions.}
\end{adefinition}

The following theorem connects spherical harmonics and Legendre polynomials.

\begin{atheorem}[addition theorem {\cite[Theorem 2.9]{atkinson2012spherical}}]\label{thm:addition}
Let $\{Y_{n,j}^d:1\leq j\leq N_{n,d}\}$ be an orthonormal basis of $\mathbb{Y}_n^d$, i.e.,
$$\int_{\Sp^{d-1}}Y_{n,j}^d(u)Y_{n,k}^d(u)\,d\vol_{\Sp^{d-1}}(u)=\delta_{jk}.$$
Then
$$\sum_{j=1}^{N_{n,d}}Y_{n,j}^d(u)Y_{n,j}^d(v)=\frac{N_{n,d}}{\vert\Sp^{d-1} \vert}P_{n,d}(\langle u,v\rangle)$$
for any $u,v\in\Sp^{d-1}$.
\end{atheorem}

\begin{alemma}[{\cite[Corollary 2.15]{atkinson2012spherical}}]\label{cor:sphharmonicsorthonogal}
For $m\neq n$, $\mathbb{Y}_m^d\perp\mathbb{Y}_n^d$ in $L^2(\nvol_{\Sp^{d-1}})$.
\end{alemma}

\begin{acorollary}\label{cor:L1Ynzero}
For any $n>0$ and $Y_n^d\in\mathbb{Y}_n^d$, $\int_{\Sp^{d-1}}Y_n^d(u)\,d\nvol_{\Sp^{d-1}}(u)=0$.
\end{acorollary}
\begin{proof}
Apply Lemma \ref{cor:sphharmonicsorthonogal} and the fact $Y_0^d\equiv 1\in \mathbb{Y}_0^d$.
\end{proof}

Let us collect some useful properties of Legendre polynomials which will be used in the main body.

\begin{alemma}\label{lemma:Legdrprop}
For any $d\geq 3$, $m,n\geq 0$, and $t\in [-1,1]$,
\begin{enumerate}
    \item $\vert P_{n,d}(t)\vert\leq 1=P_{n,d}(1)$ (see \cite[(2.39)]{atkinson2012spherical}).
    
    \item $\int_{\Sp^{d-1}} \vert P_{n,d}(\langle u,v\rangle )\vert^2\,d\vol_{\Sp^{d-1}}(v)=\frac{\vert\Sp^{d-1}\vert}{N_{n,d}}$ for any $u\in\Sp^{d-1}$ (see \cite[(2.40)]{atkinson2012spherical}).

    \item \textbf{(Rodrigues representation formula)} $P_{n,d}(t)=(-1)^n R_{n,d}(1-t^2)^{\frac{3-d}{2}}\left(\frac{d}{dt}\right)^n (1-t^2)^{n+\frac{d-3}{2}}$ where $R_{n,d}=\frac{\Gamma\left(\frac{d-1}{2}\right)}{2^n\Gamma\left(n+\frac{d-1}{2}\right)}$ (see \cite[Theorem 2.23]{atkinson2012spherical}).
    
    \item  $P_{n,d}(-t)=(-1)^n P_{n,d}(t)$ (see \cite[(2.73)]{atkinson2012spherical}).
    
    \item If $f\in C^n([-1,1])$, then
    $$\int_{-1}^1 f(t)P_{n,d}(t)(1-t^2)^{\frac{d-3}{2}}\,dt=R_{n,d}\int_{-1}^1 f^{(n)}(t)(1-t^2)^{n+\frac{d-3}{2}}\,dt$$
    where $R_{n,d}=\frac{\Gamma\left(\frac{d-1}{2}\right)}{2^n\Gamma\left(n+\frac{d-1}{2}\right)}$ (see \cite[Proposition 2.26]{atkinson2012spherical}).
    
    \item $\int_{-1}^1 P_{m,d}(t) P_{n,d}(t)(1-t^2)^{\frac{d-3}{2}}\,dt=\frac{\vert\Sp^{d-1}\vert}{N_{n,d}\vert\Sp^{d-2}\vert}\delta_{mn}$ (see \cite[(2.79)]{atkinson2012spherical}).
    
    \item \textbf{(Poisson identity)} For any $r\in (-1,1)$ and any $t\in [-1,1]$,
    $$\sum_{n=0}^\infty N_{n,d} P_{n,d}(t)\,r^n=\frac{1-r^2}{(1+r^2-2rt)^{\frac{d}{2}}}.$$
    Moreover, for each fixed $r\in (-1,1)$, the convergence is uniform as a function of $t$ (see \cite[Proposition 2.28]{atkinson2012spherical}).
    
    \item \textbf{(Green-Beltrami identity)} For any $f\in C^2(\Sp^{d-1})$ and any $g\in C^1(\Sp^{d-1})$, we have
    $$\int_{\Sp^{d-1}} g\Delta f\,d\vol_{\Sp^{d-1}}=-\int_{\Sp^{d-1}} \langle\nabla g, \nabla f\rangle\,d\vol_{\Sp^{d-1}}$$
    (see \cite[Proposition 3.3]{atkinson2012spherical}).
    
    \item For any $f\in C^1(\Sp^{d-1})$ and $u\in\Sp^{d-1}$, we have
    $$\int_{\Sp^{d-1}} \nabla f(v)\cdot\nabla_v P_{n,d}(\langle u,v \rangle)\,d\vol_{\Sp^{d-1}}(v)=n(n+d-2)\int_{\Sp^{d-1}} f(v)P_{n,d}(\langle u,v \rangle)\,d\vol_{\Sp^{d-1}}(v)$$
    (see \cite[(3.21)]{atkinson2012spherical}).
    
    \item For any $u\in\Sp^{d-1}$, we have
    $$\int_{\Sp^{d-1}} \Vert\nabla_v P_{n,d}(\langle u, v \rangle) \Vert^2\,d\vol_{\Sp^{d-1}}(v)=n(n+d-2)\frac{\vert\Sp^{d-1}\vert}{N_{n,d}}$$
    (see \cite[Proposition 3.6]{atkinson2012spherical}).
\end{enumerate}
\end{alemma}
\begin{proof}
For item (7) (Poisson identity), the uniform convergence holds because of the following argument: Observe that, for any positive integer $k>0$,

$$\left\vert\sum_{n=k}^\infty N_{n,d} P_{n,d}(t)\,r^n\right\vert\leq\sum_{n=k}^\infty N_{n,d} \vert P_{n,d}(t)\vert\,\vert r\vert^n=\sum_{n=k}^\infty N_{n,d}\vert r\vert^n\quad(\because\text{ item (1)}).$$

But, since
$$\sum_{n=0}^\infty N_{n,d}\vert r\vert^n=\sum_{n=0}^\infty N_{n,d}P_{n,d}(1)\,\vert r\vert^n=\frac{1+\vert r\vert}{(1-\vert r \vert)^d}<\infty,$$

$\sum_{n=k}^\infty N_{n,d}\vert r\vert^n$ becomes arbitrarily small for large enough $k$.\medskip

For the proof of the other items, see \cite{atkinson2012spherical}.
\end{proof}
\end{document}